\documentclass[10pt]{article}

\usepackage[T1]{fontenc}
\usepackage{lmodern}
\usepackage[margin=1in]{geometry}
\usepackage{amsmath,amssymb,amsthm,mathtools,mathrsfs}
\usepackage{microtype}
\usepackage{enumitem}
\usepackage{booktabs}
\usepackage[colorlinks=true,linkcolor=blue,citecolor=blue,urlcolor=blue]{hyperref}

\usepackage{xcolor}

\allowdisplaybreaks
\setlist[itemize]{leftmargin=1.55em,itemsep=2pt,topsep=3pt}
\setlist[enumerate]{leftmargin=1.8em,itemsep=2pt,topsep=3pt}

\newtheorem{theorem}{Theorem}[section]
\newtheorem{corollary}{Corollary}[section]

\newtheorem{proposition}{Proposition}[section]
\newtheorem{lemma}[proposition]{Lemma}

\theoremstyle{definition}
\newtheorem{definition}[proposition]{Definition}

\theoremstyle{remark}
\newtheorem{remark}[proposition]{Remark}

\numberwithin{equation}{section}
\newtheorem{othertheorem}{Theorem}

\newcommand{\R}{\mathbb R}
\newcommand{\Sph}{\mathbb S}
\newcommand{\Cyl}{\mathcal C}
\newcommand{\cX}{\mathcal X}
\newcommand{\cY}{\mathcal Y}
\newcommand{\dd}{\,\mathrm d}
\newcommand{\ind}{\operatorname{ind}}
\newcommand{\sech}{\operatorname{sech}}
\newcommand{\Span}{\operatorname{span}}
\newcommand{\Ker}{\operatorname{Ker}}
\newcommand{\Ran}{\operatorname{Ran}}
\newcommand{\tr}{\operatorname{tr}}
\newcommand{\spec}{\sigma}
\newcommand{\ess}{\mathrm{ess}}
\newcommand{\cA}{\mathfrak a}
\newcommand{\cB}{\mathfrak b}
\newcommand{\Hyp}{\,{}_2F_1}
\newcommand{\ip}[2]{\left\langle #1,#2\right\rangle}
\newcommand{\norm}[1]{\left\lVert #1\right\rVert}
\newcommand{\abs}[1]{\left|#1\right|}

\title{Uniqueness and Symmetry Breaking for \\ Singular $Q$-curvature Equations in $\R^n$}

\author{
	Haoyu Li
	\and
	Tuoxin Li
	\and
	Juncheng Wei
	\and
	Wen Yang
}

\date{}

\begin{document}
\maketitle

\begin{abstract}
For $n\ge3$ and $\gamma>-1$, we study normal solutions of
\begin{equation*}
 (-\Delta)^{n/2}u=(n-1)!\,|x|^{n\gamma}e^{nu}
 \quad\hbox{in }\R^n.
\end{equation*}
We reformulate this problem on the cylinder via the Emden--Fowler transform. We first establish Kelvin symmetry for normal solutions. We then prove existence and uniqueness, up to dilation, of the radial normal solution and establish its radial nondegeneracy. Next, we carry out a detailed spectral analysis of the linearized operator and construct branches of nonradial normal solutions by the Crandall--Rabinowitz theorem. Finally, we prove that every normal solution is radial when $\gamma>0$ is sufficiently small.
\end{abstract}
\tableofcontents

\section{Introduction}
In this paper we consider, for $n\ge3$ and $\gamma>-1$, the singular critical equation
\begin{equation}\label{eq:euclidean}
 (-\Delta)^{n/2}u=(n-1)!\,|x|^{n\gamma}e^{nu}
 \quad\hbox{in }\R^n,
 \qquad
 \Lambda:=\int_{\R^n}|x|^{n\gamma}e^{nu}\dd x<\infty.
\end{equation}
Throughout the paper, we write
\begin{equation*}
	\alpha:=1+\gamma>0.
\end{equation*}

After adding a multiple of $\log|x|$ and a constant, equation \eqref{eq:euclidean} can be written as
\begin{equation}\label{eq:w}
	(-\Delta)^{n/2}w=e^{nw}-\frac{(n-1)!|\Sph^n|\gamma}{2}\delta_0
	\quad\hbox{in }\R^n,  \qquad
	\int_{\R^n}e^{nw}\dd x<\infty.
\end{equation}
Here $\delta_0$ denotes the Dirac mass at the origin. Equation \eqref{eq:w} arises in the prescribed $Q$-curvature problem with singularities. Let $g_E$ be the standard metric on $\R^n$.  If $g_u=e^{2u}g_E$ is a conformal metric on $\R^n$, then the conformal covariance law for the critical GJMS operator has the form
\begin{equation*}
	P_{g_E}u+Q_{g_E}=Q_{g_u}e^{nu}.
\end{equation*}
In the flat case, $Q_{g_E}=0$ and the critical operator is $P_{g_E}=(-\Delta)^{n/2}$.  It is a differential operator when $n$ is even and a nonlocal elliptic pseudodifferential operator when $n$ is odd. The GJMS operators and the associated $Q$-curvature were introduced in the work of Graham--Jenne--Mason--Sparling \cite{GJMS1992}. We refer to \cite{Branson1995,CLY2019,CM2023,ChangChen2001,ChangGonzalez2011,DHL2000,DM2008,FG2013,GurMal2015} and the references therein for more results and background on GJMS operators and related problems.

The case $n=2$ is well understood. In fact, Chen--Li \cite{ChenLi1993} showed that when $n=2$, every solution of \eqref{eq:euclidean} necessarily has mass
\begin{equation*}
	\int_{\R^2}|x|^{2\gamma}e^{2u}\dd x=4\pi(1+\gamma).
\end{equation*}
In \cite{PrajapatTarantello2001}, Prajapat and Tarantello showed that, in complex variables, all solutions have the form
\begin{equation}\label{2D sol}
	u_{\Lambda,\xi}(z)=\log\frac{2(\gamma+1)\Lambda}{\Lambda^2+|z^{\gamma+1}-\xi|^2}
\end{equation}
for some $\Lambda\in \R_+$ and $\xi\in\mathbb{C}$ and $\xi$ must be $0$ unless $\gamma$ is a nonnegative integer. When $\gamma=0$, $\xi$ corresponds to a translation, but when $\gamma\in \mathbb{N}_+$, the free parameter $\xi$ yields an additional explicit family of nonradial solutions. In the notation $\alpha=1+\gamma$, the loss of radial symmetry occurs precisely at the positive integers $\alpha=k$. For integer singular sources in dimension two, the bounded kernel of the linearized equation has also been classified in \cite{DEM2012}. For $n\ge3$, as shown below, singular solutions are generally non-explicit, and symmetry breaking is considerably subtler in both location and mechanism.

\medskip
A new phenomenon in higher dimensions is the basic distinction between \emph{normal} and \emph{nonnormal} solutions. We recall the following decomposition theorem due to Hyder--Mancini--Martinazzi \cite{HMM2021}.
\begin{othertheorem}  \cite[Theorem 1.1]{HMM2021}\label{thm:NormalSolutionHMM}
    Let $u$ solve \eqref{eq:euclidean}. Then
    \begin{equation*}
        u(x)=\frac{2}{|\Sph^n|}\int_{\R^n}\log\left(\frac{1+|y|}{|x-y|}\right)|y|^{n\gamma}e^{nu(y)}\dd y+p(x),
    \end{equation*}
    where $p$ is a polynomial bounded above and of degree at most $n-1$. Moreover, 
    \begin{equation*}
        \lim_{|x|\to\infty}\frac{u(x)-p(x)}{|\log |x||}=-\frac{2\Lambda}{|\mathbb{S}^n|}.
    \end{equation*}
    If $p$ is constant, then
    \begin{equation}\label{eq:normal-mass-intro}
 \Lambda=\int_{\R^n}|x|^{n\gamma}e^{nu}\dd x
 =\alpha|\Sph^n|.
\end{equation}
\end{othertheorem}

A solution is called \emph{normal} when the polynomial $p$ is constant. Otherwise, it is called \emph{nonnormal}.

Nonnormal solutions may contain a genuine polynomial part at infinity and their volume need not equal the value in \eqref{eq:normal-mass-intro}. The normal class is therefore the geometrically distinguished class in which the asymptotic behavior is purely logarithmic.  As mentioned above, all solutions are normal in dimension two. However, nonnormal solutions also exist in higher dimensions. See, for example, \cite{Hyder2016,HM2015,JMMX2015,WeiYe2008}.

For the nonsingular higher-dimensional equation, extensive classification and existence results are available. See \cite{Hyder2016,Hyder2019,JMMX2015,Lin1998,Martinazzi2009,WeiXu1999} and the references therein. For the singular equation, Hyder--Mancini--Martinazzi \cite[Theorem~1.2]{HMM2021} proved the existence of a radial normal solution for every $n\ge3$ and $\gamma>-1$. In dimension four, Fang--Ma \cite[Theorems~1.5 and~6.7]{FangMa2022} established uniqueness for the corresponding radial equation with prescribed linear growth. Ahmedou--Wu--Zhang \cite[Theorem~1.3]{AWZ2022} proved radial symmetry in the range $-1<\gamma<0$. For more results about isolated singularities, we refer to \cite{GhoshHyder2025,GLW2026,HM2021Power,JSTW2022,KonigLaurain2022}. For geometric problems with several conical points or on closed conic manifolds, see \cite{FangMa2022,HLW2020,JSY2025}.

 To the best of our knowledge, several basic questions concerning normal solutions remain open in general higher dimensions, including uniqueness of the radial normal solution, its nondegeneracy and the construction of nonradial normal solutions. The purpose of the present paper is to address these questions in every dimension $n\ge3$. Our approach treats the local and nonlocal cases within an unified formulation. The starting point is the Emden--Fowler transform
\begin{equation}\label{EF trans}
 x=e^t\omega, \qquad (t,\omega)\in\Cyl:=\R\times\Sph^{n-1},
 \qquad V(t,\omega)=u(e^t\omega)+\alpha t.
\end{equation}
By conformal covariance, equation \eqref{eq:euclidean} becomes the following autonomous equation on the cylinder:
\begin{equation}\label{eq:cylinder-equation}
 P_nV=(n-1)!e^{nV}\quad\hbox{on }\Cyl,
\end{equation}
where $P_n$ is the critical cylinder GJMS operator. Let $\mathcal{H}_k$ denote the space of spherical harmonics of degree $k$ and let $\{Y_{k,\ell}\}_{\ell}$ be an orthonormal basis of $\mathcal{H}_k$. Then
\begin{equation*}
 -\Delta_{\Sph^{n-1}}Y_{k,\ell}=\lambda_kY_{k,\ell}\text{ and }\lambda_k:=k(k+n-2).
\end{equation*}
For a function $\phi=\phi(t)$, the operator separates as
\begin{equation*}
 P_n(\phi Y_{k,\ell})=(A_{n,k}\phi)Y_{k,\ell},
\end{equation*}
where $A_{n,k}$ is the one-dimensional Fourier multiplier with symbol
\begin{equation}\label{eq:mode-symbol}
 a_{n,k}(\xi)=2^n
 \frac{\Gamma\left(\frac{n+k+i\xi}{2}\right)
       \Gamma\left(\frac{n+k-i\xi}{2}\right)}
      {\Gamma\left(\frac{k+i\xi}{2}\right)
       \Gamma\left(\frac{k-i\xi}{2}\right)}.
\end{equation}
In fact, \cite[Theorem~1.1]{DelaTorreGonzalez2018} computed the cylinder symbols for a family of subcritical operators. The critical symbol $a_{n,k}$ is the limit of these cylinder symbols. 

For real $\xi$, $a_{n,k}$ is positive except when $k=0$ and $\xi=0$.  If $n=2m$ is even, the Gamma quotient reduces to
\begin{equation}\label{a 2m symbol}
 a_{2m,k}(\xi)=\prod_{j=0}^{m-1}\bigl((k+2j)^2+\xi^2\bigr),
 \qquad A_{2m,k}=\prod_{j=0}^{m-1}\left(-\partial_t^2+(k+2j)^2\right).
\end{equation}
For odd $n$, the same Gamma-function quotient defines a self-adjoint pseudodifferential operator of order $n$. 

Kelvin symmetry connects the behavior at the origin with the behavior at infinity. Our first result shows that every normal solution is Kelvin symmetric.
	\begin{theorem}\label{thm:kelvin}
	Let $u$	be a normal solution of \eqref{eq:euclidean}. Then there is a unique $R>0$ such that
		\begin{equation}\label{u kelvin}
	u(x)=u\!\left(\frac{R^2x}{|x|^2}\right)				-2\alpha\log\frac{|x|}{R},			\qquad x\neq 0.
		\end{equation}
      Here $R$ is the unique radius such that
		\begin{equation}\label{R half mass}
			\int_{B_R}|x|^{n\gamma}e^{nu(x)}\dd x=\frac{1}{2}\alpha|\Sph^n|.
		\end{equation}
		Equivalently, for $t_0=\log R$ and $V$ given by \eqref{EF trans}, we have
		\begin{equation}\label{V sym}
				V(t_0+t,\omega)=V(t_0-t,\omega),\qquad (t,\omega)\in\mathcal{C}.
		\end{equation}
	\end{theorem}
\begin{remark}
    Up to a dilation, we may assume $R=1$ in (\ref{R half mass}). In cylindrical variables, this fixes the center of reflection at $t=0$. We will use this normalization in both the bifurcation argument and the compactness analysis.
\end{remark}

Next we consider radial normal solutions. A radial normal solution $u=u(r)$ corresponds, under $v_\alpha(t)=u(e^t)+\alpha t$, to a function satisfying
\begin{equation}\label{eq:radial-equation}
 A_{n,0}v_\alpha=(n-1)!e^{nv_\alpha}.
\end{equation}
Normality determines its asymptotic behavior:
\begin{equation}\label{eq:v al infty}
 v_\alpha'(t)\longrightarrow
 \begin{cases}
  \alpha,&t\to-\infty,\\
  -\alpha,&t\to+\infty.
 \end{cases}
\end{equation}
The derivation is contained in Lemma \ref{lem:V infty}. Translation in $t$ corresponds to dilation in the Euclidean variables.

Set
\begin{equation*}
 W_\alpha(t):=n!e^{nv_\alpha(t)}.
\end{equation*}
The linearized operator of \eqref{eq:cylinder-equation} at $v_\alpha$ is
\begin{equation*}
 L_\alpha=P_n-W_\alpha,
\end{equation*}
and on the degree-$k$ spherical-harmonic subspace, it reduces to
\begin{equation}\label{L al k}
 L_{\alpha,k}:=A_{n,k}-W_\alpha.
\end{equation}
The mode-$0$ part of the kernel of $L_{\alpha,0}$ contains $v_\alpha'$.  At the regular value $\alpha=1$, the mode-$1$ part represents Euclidean translations.  Subspaces with degree  $k\ge2$ are the genuinely symmetry-breaking modes considered below.

The following theorem gives a complete description of the radial normal solution.

\begin{theorem}\label{thm radial}
 For every $\alpha>0$, \eqref{eq:euclidean} has a unique radial normal solution up to dilation. Up to a  translation in $t$, the corresponding solution $v_\alpha$ of \eqref{eq:radial-equation} is even, strictly concave and strictly decreasing on $(0,\infty)$. Moreover, if $L_{\alpha,0}\phi=0$ for some bounded function $\phi$, then
$\phi\in\Span\{v_\alpha'\}$.
\end{theorem}
The proof is variational. After factoring the radial operator, the equation becomes the Euler--Lagrange equation of a free-energy functional on probability densities on $\R$. The direct method gives existence. Strict convexity modulo translation yields uniqueness up to translation, and the second variation identity gives radial nondegeneracy.

We next determine all higher-mode parameters at which the radial solution becomes degenerate beyond dilation.
\begin{theorem}\label{thm e value L}
 For every integer $k\ge2$, there is exactly one parameter $\alpha_{n,k}>1$ such that the kernel of $L_{\alpha_{n,k},k}$ is nontrivial. At this parameter, zero is the principal eigenvalue of $L_{\alpha_{n,k},k}$.  If $k_1\ne k_2$, then $\alpha_{n,k_1}\ne\alpha_{n,k_2}$.  Moreover,
 \begin{equation}\label{eq:main-range-intro}
  \cA_{n,k}<\alpha_{n,k}<\cA_{n,k}+\cB_{n,k},
 \end{equation}
 where
 \begin{equation}\label{def ak bk}
  \cA_{n,k}:=\frac{\Gamma(n+k-1)}{n\Gamma(k)\Gamma(n-1)},  \qquad  \cB_{n,k}:=\frac{k(n+k-2)}{n(n-2)}.
 \end{equation}
\end{theorem}

\begin{remark}\label{al n2 nonint}
 For $k=2$, the estimate becomes
 \begin{equation*}
  n-1<\alpha_{n,2}<n-1+\frac{2}{n-2}.
 \end{equation*}
 Thus $\alpha_{n,2}$ is not an integer whenever $n\ge4$. 
\end{remark}

The proof uses the Birman--Schwinger formulation, which turns the kernel condition for $L_{\alpha,k}$ into an eigenvalue condition at the level $1$. We show that only the principal Birman--Schwinger eigenvalue can reach $1$. Then we prove that it is strictly increasing in $\alpha$ and strictly decreasing in the angular parameter $k$. This yields a unique crossing in each angular mode and separates the crossing parameters for different modes. The explicit interval in \eqref{eq:main-range-intro} follows from the trace formula for the operator and a comparison with an explicit integral kernel.

Combining radial nondegeneracy with the mode analysis gives the following nondegeneracy consequence. 
\begin{corollary}\label{coro:nondegen}
Let $0<\alpha<\alpha_{n,2}$. Then the solution $v_\alpha$ given in Theorem~\ref{thm radial} is nondegenerate in the following sense: 
let $h$ be a bounded distributional solution of $L_\alpha h=0$ on $\Cyl$. If $\alpha\neq 1$, then $h\in \Span\{v_\alpha'\}$. If $\alpha=1$, then
\begin{equation}\label{h kernel}
    h(t,\omega)=-c_0\tanh{t}+\sech t\sum_{i=1}^nc_i\omega_i
    \end{equation}
    for some constants $c_0,\dots, c_n$. In particular, every bounded distributional solution of
\begin{equation*}
    	(-\Delta)^{n/2}h=n!e^{nU}h\quad\hbox{in }\R^n,
\end{equation*}
    where $U(x)=\log\frac{2}{1+|x|^2}$, is of the form $h=\sum_{i=0}^nc_iZ_i$, where
    \begin{equation}\label{eq:kernel}
Z_0(x)=\frac{1-|x|^2}{1+|x|^2},\qquad
Z_i(x)=\frac{2x_i}{1+|x|^2}.
\end{equation}
\end{corollary}

The degeneracy at $\alpha_{n,k}$ produces nonradial normal solutions.
\begin{theorem}\label{thm:main-bifur}
Let $n\ge3$ and let $k\ge2$ be an integer.
Then there exist $\varepsilon>0$ and a smooth curve
\begin{equation*}
 s\longmapsto(\alpha(s),\mathfrak u(s)), \qquad |s|<\varepsilon,    
\end{equation*}
of normal solutions of \eqref{eq:euclidean} such that
\begin{equation*}
 (\alpha(0),\mathfrak u(0))=(\alpha_{n,k},u_{\alpha_{n,k}})    
\end{equation*}
and $\mathfrak u(s)$ is nonradial for every sufficiently small $s\ne0$. Moreover, writing $x=r\omega$, we have
\begin{equation*}
 \mathfrak u(s)(r\omega)=u_{\alpha(s)}(r)+s\psi_{\alpha_{n,k},k}(\log r)Y_k(\omega)+o(s),    
\end{equation*}
where $Y_k\in\mathcal{H}_k$, $\psi_{\alpha_{n,k},k}>0$ spans $\Ker L_{\alpha_{n,k},k}$ and the remainder is $o(s)$ in $L^\infty(\R^n)$.
\end{theorem}

We would like to compare our results with the corresponding results for the H\'enon equation
\begin{equation}\label{eq Henon}
	\left\{
	\begin{aligned}
		&-\Delta u=(N+\alpha)(N-2)|x|^\alpha u^{\frac{N+2+2\alpha}{N-2}} &\text{ in } \R^N,\\
		&u>0 &\text{ in } \R^N,
	\end{aligned}
	\right.
\end{equation}
where $N\ge 3$. Gladiali--Grossi--Neves \cite{GGN2013} constructed nonradial solutions to \eqref{eq Henon} that decay like $|x|^{2-N}$ at infinity using bifurcation theory. Moreover, the bifurcations occur exactly at
\begin{equation*}
	\alpha_k=2(k-1),\quad k\in \mathbb{N}.
\end{equation*}
They further conjectured that non-radial solutions exist only at these special values of $\alpha_k$. However, Jiang--Xiong \cite{JX2026} have recently disproved this conjecture by expanding the bifurcation parameter to the next order. This is in strong contrast with the two-dimensional case of \eqref{eq:euclidean}, where the nonradial solution \eqref{2D sol} exists exactly when $\alpha\in \mathbb{N}$.

Our problem is subtler. For $n\ge 3$, unlike in either the two-dimensional problem \eqref{eq:euclidean} or the H\'enon equation \eqref{eq Henon}, the bifurcation parameters in the present higher-dimensional problem are difficult to compute explicitly. Instead, Theorem \ref{thm e value L} provides quantitative refined bounds for $\alpha_{n,k}$. Remark \ref{al n2 nonint} shows that $\alpha_{n,2}$ is not an integer whenever $n\ge4$. This distinguishes the higher-dimensional equation from the integer-parameter pattern in dimension two.

\medskip

Finally, complementing the existence result for nonradial solutions, we show that every normal solution is radial when $\gamma>0$ is sufficiently small.
\begin{theorem}\label{thm:rad small ga}
	Let $n\ge3$. There exists $\gamma_n>0$ such that, for every $0<\gamma<\gamma_n$, every normal solution of \eqref{eq:euclidean} is radial, and any two such solutions differ only by a dilation.
\end{theorem}

The paper is organized as follows.  Section~\ref{sec:prelimin} collects the special-function identities and basic properties of normal solutions used in the proofs. Section \ref{sec:kelvin} establishes Kelvin symmetry. In Section~\ref{sec:radial}, we construct the radial normal solution, prove its uniqueness up to translation and establish radial nondegeneracy. Section~\ref{sec:BS} is devoted to the spectral analysis of the linearized operator and the location of the bifurcation parameters.
In Section~\ref{sec:nonradial}, we introduce the weighted function spaces and apply the Crandall--Rabinowitz theorem to prove Theorem \ref{thm:main-bifur}.
Finally, in Section \ref{sec:rad small ga}, we prove Theorem \ref{thm:rad small ga} by combining compactness with a linearized argument at the standard bubble.

\section{Preliminaries}\label{sec:prelimin}
In this section, we first record the Fourier and special-function identities used in the analysis of the cylinder operators. We then establish asymptotic estimates, integral identities, and a removability criterion for normal solutions.

\subsection{Special-function and Fourier identities}

Throughout the paper, we use the Fourier convention
\begin{equation*}
 \widehat f(\xi)=\int_\R e^{-it\xi}f(t)\dd t,
 \qquad f(t)=\frac{1}{2\pi}\int_\R e^{it\xi}\widehat f(\xi)\dd\xi.
\end{equation*}
We need the Beta and Gamma identities 
\begin{equation}\label{eq:beta-gamma}
 B(p,q)=\int_0^1 s^{p-1}(1-s)^{q-1}\dd s=\frac{\Gamma(p)\Gamma(q)}{\Gamma(p+q)}, \quad \Re p,\Re q>0,
\end{equation}
\begin{equation*}
 \Gamma(z+1)=z\Gamma(z), \qquad \Gamma(z)\Gamma\left(z+\frac{1}{2}\right) =2^{1-2z}\sqrt\pi\,\Gamma(2z).
\end{equation*}
The reflection formula also yields, for $s>0$,
\begin{equation*}
 |\Gamma(is)|^2=\frac{\pi}{s\sinh(\pi s)},
 \qquad
 |\Gamma(\frac{1}{2}+is)|^2=\frac{\pi}{\cosh(\pi s)}.
\end{equation*}

The Pochhammer symbol is defined by $(a)_0=1$ and $(a)_j=a(a+1)\cdots(a+j-1)$ for $j\ge1$.  

The Gauss hypergeometric function is
\begin{equation*}
 \Hyp(a,b;c;z)=\sum_{j=0}^\infty
 \frac{(a)_j(b)_j}{(c)_j\,j!}z^j,
 \qquad |z|<1.
\end{equation*}
It satisfies the equation
\begin{equation}\label{eq:hyp-ode}
 z(1-z)y''+[c-(a+b+1)z]y'-aby=0.
\end{equation}
For $\Re c>\Re b>0$, it has Euler's integral representation
\begin{equation}\label{eq:hyp-euler}
 \Hyp(a,b;c;z)=\frac{\Gamma(c)}{\Gamma(b)\Gamma(c-b)}\int_0^1 s^{b-1}(1-s)^{c-b-1}(1-zs)^{-a}\dd s.
\end{equation}
In particular, if $\Re(c-a-b)>0$, we have
\begin{equation}\label{eq:hyp-gauss}
 \Hyp(a,b;c;1)=\frac{\Gamma(c)\Gamma(c-a-b)}{\Gamma(c-a)\Gamma(c-b)}.
\end{equation}
We also use the quadratic transformation
\begin{equation}\label{eq:quadratic-transformation}
 z^{k/2}\Hyp\left(1-\frac n2,k;k+\frac n2;z\right)=2^{-k}x^{k/2}\Hyp\left(\frac k2,\frac{k+1}{2};k+\frac n2;x\right),
 \qquad x=\frac{4z}{(1+z)^2}.
\end{equation}
It follows either from the standard quadratic transformations or by checking that the two sides solve the same transformed equation and have the same asymptotic behavior at $z=0$. Finally,
\begin{equation}\label{eq:sech-transform}
 \int_\R e^{-it\xi}\sech^p t\dd t =\frac{2^{p-1}}{\Gamma(p)}\Gamma\left(\frac{p+i\xi}{2}\right)\Gamma\left(\frac{p-i\xi}{2}\right), \qquad p>0.
\end{equation}
This is obtained by the substitution $s=(1+\tanh t)/2$. These identities are listed, for example, in the NIST Digital Library of Mathematical Functions \cite{DLMF}.

\subsection{Fundamental properties of normal solutions}
The results in this subsection provide the asymptotic and integral information used in the symmetry arguments. We also prove a removability and normality criterion for solutions obtained by returning from the cylinder to Euclidean variables. Recall that by Theorem \ref{thm:NormalSolutionHMM}, a normal solution has the representation
\begin{equation}\label{rig:eq:normal-representation}
 u(x)=c+\frac{2}{|\Sph^n|}\int_{\R^n} \log\left(\frac{1+|y|}{|x-y|}\right)|y|^{n\gamma}e^{nu(y)}\dd y.
\end{equation}

\begin{lemma}\label{lem:V infty}
Let $\gamma>0$. Let $u$ be a normal solution of \eqref{eq:euclidean} and let $V$ be given by \eqref{EF trans}. Then there exist constants $b_-, b_+\in \R$ such that uniformly in $\omega\in \Sph^{n-1}$,
    	\begin{equation}\label{V infty}
		\begin{aligned}
			V(t,\omega)=\alpha t+b_-+O(e^t),\quad
			&\partial_t V(t,\omega)=\alpha+O(e^t),&t\to-\infty,\\
			V(t,\omega)=-\alpha t+b_++O(e^{-t}),\quad
			&\partial_t V(t,\omega)=-\alpha+O(e^{-t}),&t\to+\infty.
		\end{aligned}
	\end{equation}
If $-1<\gamma\le 0$, then for $\delta<\min\{1, n\alpha\}$, \eqref{V infty} is still true if $O(e^{\pm t})$ is replaced by $O(e^{\pm \delta t})$ respectively. 

In particular, $\|\partial_t V\|_{L^\infty(\Cyl)}<\infty$ and $e^{nV}\le Ce^{-n\alpha|t|}$.
\end{lemma}
\begin{proof}
  Write $f(x)=|x|^{n\gamma}e^{nu(x)}$. Then $\int_{\R^n}f(x)\dd x=\alpha|\Sph^n|$. The normal logarithmic behavior of $u$ implies that, for every sufficiently small $\varepsilon>0$,
	\begin{equation*}
		f(x)\leq C_\varepsilon |x|^{-n(1+\alpha)+\varepsilon},
		\qquad |x|\geq 1.
	\end{equation*}
	We first assume $\gamma>0$. Then since $n\alpha>2$, we can choose $\beta>0$ so that
	\begin{equation}\label{f decay}
		f(x)\leq C(1+|x|)^{-n-2-\beta}.
	\end{equation}
	
	We may therefore absorb the numerator in \eqref{rig:eq:normal-representation} into its constant and write
	\begin{equation}\label{u rep f}
		u(x)=c-\frac{2}{|\Sph^n|}\int_{\R^n}\log|x-y|\,f(y)\,dy
	\end{equation}
	Differentiation is legitimate by \eqref{f decay}. Hence
	\begin{equation}\label{grad u rep f}
		\nabla u(x)=-\frac{2}{|\Sph^n|}\int_{\R^n}\frac{x-y}{|x-y|^2}\,f(y)\,dy.
	\end{equation}
	
	To obtain the behavior at infinity, we rewrite \eqref{u rep f} as
    \begin{equation}\label{uneq 0rep f}
		u(x)=c-\frac{2}{|\Sph^n|}\int_{\R^n}\log\frac{|x-y|}{|x|}\,f(y)\,dy-2\alpha\log|x|, \qquad x\neq 0.
	\end{equation}
    For $r=|x|\geq 2$, we split $\R^n$ into
	\begin{equation*}
		\Omega_1=\{y: |y|<r/2\},\quad
		\Omega_2=\{y: |y|\geq r/2,\ |x-y|<r/2\},\quad
		\Omega_3=\{y: |y|\geq r/2,\ |x-y|\geq r/2\}.
	\end{equation*}
	On $\Omega_1$,
	\begin{equation*}
		\left|\log\frac{|x-y|}{r}\right|\leq C\frac{|y|}{r},
	\end{equation*}
    so
    \begin{equation*}
        \int_{\Omega_1}\left|\log\frac{|x-y|}{r}\right|f(y)\,dy\le \frac{C}{r}.
    \end{equation*}
	 On $\Omega_2$, \eqref{f decay} gives
	\begin{equation*}
		\begin{aligned}
			\int_{\Omega_2}
			\left|\log\frac{|x-y|}{r}\right|f(y)\,dy\leq Cr^{-2-\beta}
		\end{aligned}
	\end{equation*}
	On $\Omega_3$, we have
	\begin{equation*}
		\left|\log\frac{|x-y|}{r}\right|
		\leq C+\log\left(1+\frac{|y|}{r}\right),
	\end{equation*}
    so the integral on $\Omega_3$ is again $O(r^{-2-\beta})$. Consequently, for $b_+=c$, we have
	\begin{equation*}
	u(x)=-2\alpha\log|x|+b_++O(|x|^{-1}),\qquad x\to \infty.
	\end{equation*}
A similar analysis applied to \eqref{grad u rep f} gives
\begin{equation*}
    \nabla u(x)=-2\alpha\frac{x}{|x|^2}+O(|x|^{-2}), \qquad x\to \infty.
\end{equation*}
    
	At the origin, $C^1$ regularity (see \cite[Theorem 2.1]{HMM2021}) gives $u(x)=u(0)+O(|x|)$. Thus, writing $b_-=u(0)$ gives \eqref{V infty}.

    Finally, when $-1<\gamma\leq 0$, we take $\delta<\min\{1,n\alpha\}$. Then for $0<\varepsilon<n\alpha-\delta$, we have
    \begin{equation*}
        f(x)\le \begin{cases}
            C|x|^{n\alpha-n},\qquad |x|<1,\\
            C_\varepsilon |x|^{-n-n\alpha+\varepsilon}, \qquad |x|\ge 1.
        \end{cases}
    \end{equation*}
   For $r=|x|$ small, by splitting the logarithmic integral, we can obtain
    \begin{equation*}
        |u(x)-u(0)|+r|\nabla u(x)|\le Cr^{\delta}.
    \end{equation*}
    At infinity, following the previous argument, we use $\left|\log\frac{|x-y|}{r}\right|\leq C\frac{|y|}{r}\le C(\frac{|y|}{r})^\delta$ to obtain
    \begin{equation*}
        \int_{\Omega_1}\left|\log\frac{|x-y|}{r}\right|f(y)\,dy\le \frac{C}{r^\delta}.
    \end{equation*}
    Then the same splitting of integral gives
    \begin{equation*}
        |u(x)+2\alpha\log r-b_+|+|x\cdot \nabla u(x)+2\alpha|\le Cr^{-\delta}.
    \end{equation*}
    This gives the weaker asymptotics for $-1<\gamma\le 0$.
\end{proof}

\begin{lemma}\label{rig:lem:translation-identity}
Let $u$ be a normal solution of \eqref{eq:euclidean} with $\gamma>0$ and let $f(x)=|x|^{n\gamma}e^{nu(x)}$. Then
\begin{equation}\label{rig:eq:inverse-moment}
 \int_{\R^n}\frac{x}{|x|^2}f(x)\dd x=0.
\end{equation}
\end{lemma}

\begin{proof}
By \eqref{grad u rep f} and Fubini's theorem, we obtain
\begin{equation}\label{int f grad u=0}
     \int_{\R^n}f(x)\nabla u(x)\dd x=-\frac{2}{|\Sph^n|}\iint \frac{x-y}{|x-y|^2}f(x)f(y)\dd y\dd x=0
\end{equation}
due to its antisymmetry in $x$ and $y$.

A direct computation gives
\begin{equation}\label{grad f}
    \nabla f=n\gamma\frac{x}{|x|^2} f+n f\nabla u.
\end{equation}
Since $\int_{\R^n}\nabla f \dd x=0$, integrating \eqref{grad f} and using \eqref{int f grad u=0} gives  \eqref{rig:eq:inverse-moment}.
\end{proof}

\begin{lemma}\label{lem:removability-normality}
Let $n\ge3$, $\gamma>-1$ and $u\in C^0(\R^n)$ satisfy
\begin{equation*}
     |u(x)|\le C(1+\log(1+|x|)).
\end{equation*}
Suppose that $|x|^{n\gamma}e^{nu}\in L^1(\R^n)$ and that $(-\Delta)^{n/2}u=(n-1)!|x|^{n\gamma}e^{nu}$ on $\R^n\setminus\{0\}$. Then the equation holds on $\R^n$ and $u$ is normal.
\end{lemma}

\begin{proof}
Write $P=(-\Delta)^{\frac{n}{2}}$, $f=(n-1)!|x|^{n\gamma}e^{nu}$ and $T=Pu-f$, so that $\operatorname{supp}T\subset\{0\}$. For $\phi\in C_c^\infty(\R^n)$, we have $\int_{\R^n}P\phi\dd y=0$ and $(1+\log(1+|y|))P\phi\in L^1(\R^n)$; in odd dimensions the latter follows from \cite[Proposition~2.1]{Hyder2019}.
Therefore, a scaling and the dominated convergence theorem give
\begin{equation*}
 \left\langle T,\phi(\cdot/\varepsilon)\right\rangle
 =\int_{\R^n}[u(\varepsilon y)-u(0)]P\phi(y)\dd y-\int_{\R^n}f(x)\phi(x/\varepsilon)\dd x\longrightarrow0.
\end{equation*}
By the structure theorem for distributions supported at a point (se \cite[Theorem~6.24(d) $\And$ Theorem~6.25]{Rudin1991}), this forces $T=0$. 

Applying Theorem~\ref{thm:NormalSolutionHMM} to this global solution gives $u=v+p$, where $v$ has logarithmic growth and $p$ is a polynomial. The growth assumption on $u$ implies that $p$ is constant, so $u$ is normal.
\end{proof}

\section{Kelvin symmetry}\label{sec:kelvin}
In this section, we prove Theorem \ref{thm:kelvin} on the cylinder, where Kelvin inversion becomes reflection in the $t$-variable. The argument applies the moving plane method to the integral representation of a normal solution. Related integral moving-plane arguments are developed in \cite{ChenLiOu2006}.

For a normal cylinder solution $V$, set
	\begin{equation*}
		p(t,\omega)=\frac{e^{nV(t,\omega)}}{\alpha|\Sph^n|}.
	\end{equation*}
	Then $\int_{\Cyl}p\dd t\dd \omega=1$, and \eqref{V infty} implies
	\begin{equation*}
		0<p(t,\omega)\leq Ce^{-n\alpha|t|},
		\qquad |p_t(t,\omega)|=n|\partial_t V(t,\omega)|p(t,\omega)\leq Cp(t,\omega).
	\end{equation*}
	
Using the identity
	\begin{equation*}
		\log|e^t\omega-e^s\eta|
		=\frac{t+s}{2}
		+\frac{1}{2}\log\bigl(2(\cosh(t-s)-\omega\cdot\eta)\bigr),
	\end{equation*}
 the normal representation of $u$ can be written as
\begin{equation}\label{int rep V}
	V(t,\omega)=C-\alpha\int_{\R}\int_{\Sph^{n-1}}
	L(t-s,\omega\cdot\eta)p(s,\eta)\,d\eta\,ds,
	\end{equation}
	in terms of $V$, where
	\begin{equation*}
		L(\tau,a)=\log\bigl(2(\cosh\tau-a)\bigr).
	\end{equation*}
Here the finite constant $-\alpha\int_{\Cyl} sp(s,\eta)\,ds\,d\eta$ is absorbed into $C$.
	
\begin{proof}[Proof of Theorem \ref{thm:kelvin}]
    We first start the moving plane far to the left. We then show that it must stop at a finite position and prove that the stopping plane is a plane of symmetry by ruling out further continuation.
    \noindent\textbf{Step 1: Start of the moving plane.} 	For a real number $\lambda$, reflect in $t=\lambda$ and set
	\begin{equation*}
	t^\lambda=2\lambda-t,\qquad w_\lambda(t,\omega)=V(t^\lambda,\omega)-V(t,\omega) \quad(t<\lambda).
	\end{equation*}
	Then taking difference of the two representations in \eqref{int rep V} yields
	\begin{equation}\label{int rep w la}
		\begin{aligned}
			w_\lambda(t,\omega)
			&=\alpha\int_{s<\lambda}\int_{\Sph^{n-1}}K_\lambda(t,s,\omega\cdot\eta)\bigl(p(s^\lambda,\eta)-p(s,\eta)\bigr)\,d\eta\,ds\\
            &=\alpha\int_{s<\lambda}\int_{\Sph^{n-1}}K_\lambda(t,s,\omega\cdot\eta)p(s,\eta)\bigl(e^{nw_\lambda(s,\eta)}-1\bigr)\,d\eta\,ds,
		\end{aligned}
	\end{equation}
	where
	\begin{equation*}
		K_\lambda(t,s,a)
		=\log\frac{\cosh(2\lambda-t-s)-a}{\cosh(t-s)-a}.
	\end{equation*}
	
	For $t,s<\lambda$, $K_\lambda>0$ since $2\lambda-t-s>|t-s|$. To obtain an upper bound, we set $A=2\lambda-t-s$ and $D=|t-s|$. Then	for $a\in[-1,1]$ and $A, D>0$, we have elementary inequalities
	\begin{equation*}
		\begin{aligned}
			\cosh A-a\leq 2e^A,\qquad
			\cosh D-a\geq\cosh D-1\geq\frac{e^D}{2}(1-e^{-1})^2\min\{D^2,1\}.
		\end{aligned}
	\end{equation*}
  Therefore,
	\begin{equation}\label{K upper}
			0<K_\lambda(t,s,a)
			\leq 2(\lambda-s)+C_*+2\log^+\frac{1}{|t-s|},
			\qquad
			C_*=\log\frac{4}{(1-e^{-1})^2}.
	\end{equation}
	Here $\log^+r=\max\{\log r,0\}$. 
	
	We will use the following elementary bounds:
	\begin{equation}\label{log int}
		\int_\R\log^+\frac{1}{|t-s|}\dd s=2,
	\end{equation}
	and, for every interval $I$ of length $0<\ell\leq1$,
	\begin{equation}\label{log int upper}
	\sup_{t\in\R}\int_I\log^+\frac{1}{|t-s|}\dd s\leq\ell\left(1+\log\frac{2}{\ell}\right).
	\end{equation}

	To start the moving plane, we define
	\begin{equation*}
		N_\lambda=\{(s,\eta):s<\lambda,\ w_\lambda(s,\eta)<0\},
		\qquad w_\lambda^-=\max\{-w_\lambda,0\}.
	\end{equation*}
On $N_\lambda$, we have
	\begin{equation*}
1-e^{nw_\lambda(s,\eta)}=1-e^{-nw_\lambda^-(s,\eta)}\leq nw_\lambda^-(s,\eta).
	\end{equation*}
At a point where $w_\lambda(t,\omega)<0$, discarding the integral in \eqref{int rep w la} outside $N_\lambda$ yields
	\begin{equation}\label{w- le}
			w_\lambda^-(t,\omega)\leq n\alpha\int_{N_\lambda}K_\lambda(t,s,\omega\cdot\eta)p(s,\eta)
			w_\lambda^-(s,\eta)\,ds\,d\eta.
	\end{equation}
	
	For each fixed $\lambda$, the expansions \eqref{V infty} give $\lim_{t\to-\infty}w_\lambda(t,\omega)=b_+-b_--2\alpha\lambda$ uniformly in $\omega$. Together with continuity and $w_\lambda(\lambda,\omega)=0$ due to its definition, this proves that
	\begin{equation}\label{M la finite}
		M_\lambda:=\sup_{t<\lambda,\,\omega\in\Sph^{n-1}}w_\lambda^-(t,\omega)<\infty.
	\end{equation}

	For sufficiently negative $\lambda$, the bound $p(s,\eta)\leq Ce^{n\alpha s}$ holds for all $s<\lambda$. Using \eqref{log int}, the upper bound for $K_\lambda$ \eqref{K upper} gives
\begin{equation*}
	\begin{aligned}
	\sup_{t<\lambda,\,\omega}\int_{s<\lambda}\int_{\Sph^{n-1}}K_\lambda(t,s,\omega\cdot\eta)p(s,\eta)\,d\eta\,ds\leq C|\Sph^{n-1}|e^{n\alpha\lambda}\left(\frac{2}{(n\alpha)^2}+\frac{C_*}{n\alpha}+4\right)\longrightarrow0, \quad \lambda\to-\infty.
		\end{aligned}
\end{equation*}
	Consequently, for $\lambda_0$ negative enough, taking the supremum in \eqref{w- le} gives $M_\lambda\leq\frac{1}{2}M_\lambda$, so $w_\lambda\geq0$ for every $\lambda\leq\lambda_0$.	This starts the moving plane.
\medskip

    \noindent\textbf{Step 2: The stopping value and strict positivity.}
    Define the stopping value by
	\begin{equation}\label{def t0}
		t_0=\sup\Bigl\{\lambda\in\R:w_\mu(t,\omega)\geq0
		\text{ for every }\mu\leq\lambda,\ t<\mu,\ \omega\in\Sph^{n-1}\Bigr\}.
	\end{equation}
 By \eqref{V infty},	$\lim_{t\to-\infty}w_\lambda(t,\omega)=b_+-b_--2\alpha\lambda<0$	for sufficiently large $\lambda$, so $t_0$ is finite. By continuity,	$w_{t_0}(t,\omega)\geq0$ for all $t<t_0$. Our goal is to prove that $w_{t_0}\equiv 0$ and then the symmetry follows.

	Suppose, for contradiction, that $w_{t_0}$ is not identically zero. It is then positive on a nonempty open subset. On that subset $	p(s,\eta)\bigl(e^{nw_{t_0}(s,\eta)}-1\bigr)>0$. Since $K_{t_0}>0$ by \eqref{K upper}, the integral representation \eqref{int rep w la} then yields
	\begin{equation}\label{w>0 strict}
		w_{t_0}(t,\omega)>0\quad\text{for every }t<t_0,\ \omega\in\Sph^{n-1}.
	\end{equation}
We now show that this contradicts the definition of $t_0$.
    \medskip

    \noindent\textbf{Step 3: Excluding further continuation.}
    For $\lambda>t_0$ close to $t_0$, we divide the left half-cylinder into the three regions
\begin{equation*}
 (-\infty,-T)\times\Sph^{n-1},\qquad
 [-T,t_0-\delta]\times\Sph^{n-1},\qquad
 (t_0-\delta,\lambda)\times\Sph^{n-1}.   
\end{equation*}

Strict positivity of $w_{t_0}$ and continuity preserve positivity on the middle region. On the other two regions, choosing $T$ large and $\delta$ small makes the kernel integrals sufficiently small to exclude a negative part by \eqref{w- le}.
    	
We first provide a uniform estimate for the integral over $(-\infty,-T)\times\Sph^{n-1}$. To this end, we choose $T>|t_0|+2$ sufficiently large that $p(s,\eta)\leq Ce^{n\alpha s}$ for all $s<-T$. For all $\lambda\in[t_0,t_0+1]$, we have
	\begin{equation*}
	\int_{s<-T}\int_{\Sph^{n-1}}\bigl(2(\lambda-s)+C_*\bigr)p(s,\eta)\,d\eta\,ds\leq C\int_{-\infty}^{-T}(1+|s|)e^{n\alpha s}\,ds
	\leq C(1+T)e^{-n\alpha T}.
	\end{equation*}
	For the logarithmic part, we have
	\begin{equation*}
		\int_{s<-T}\int_{\Sph^{n-1}}\log^+\frac{1}{|t-s|}\,p(s,\eta)\,d\eta\,ds
        \leq Ce^{-n\alpha T}\int_{\R}\log^+\frac{1}{|t-s|}\,ds=2Ce^{-n\alpha T}.
	\end{equation*}
uniformly in $t$. Consequently, (\ref{K upper}) implies 
	\begin{equation}\label{int far left}
		\sup_{\substack{t_0\leq\lambda\leq t_0+1\\t<\lambda,\ \omega\in\Sph^{n-1}}}\int_{s<-T}\int_{\Sph^{n-1}}K_\lambda(t,s,\omega\cdot\eta)p(s,\eta)\,d\eta\,ds
		\leq C(1+T)e^{-n\alpha T}.
	\end{equation}

We also need an estimate on $(t_0-\delta,\lambda)\times\Sph^{n-1}$. Let $0<\delta<1/4$ and $\lambda\in[t_0,t_0+\delta]$. Set
	\begin{equation*}
		I=(t_0-\delta,\lambda),\qquad \ell=|I|\leq2\delta<1.
	\end{equation*}
	On this strip $0<\lambda-s<\ell$. Since $p$ is bounded, \eqref{K upper} and \eqref{log int upper} imply
	\begin{equation}\label{int strip}
		\begin{aligned}
			\sup_{t<\lambda,\,\omega}
			&\int_{t_0-\delta<s<\lambda}\int_{\Sph^{n-1}}K_\lambda(t,s,\omega\cdot\eta)p(s,\eta)\,d\eta\,ds\\
			&\quad\leq\|p\|_\infty|\Sph^{n-1}|\left[\ell^2+C_*\ell+2\ell\left(1+\log\frac{2}{\ell}\right)\right]\\
			&\quad\leq C\delta(1+|\log\delta|).
		\end{aligned}
	\end{equation}

	Now we fix $T$ large enough and $\delta$ small enough such that the integrals in \eqref{int far left} and \eqref{int strip} are less than $\frac{1}{4n\alpha}$. Then by the compactness of $[-T,t_0-\delta]\times\Sph^{n-1}$ and \eqref{w>0 strict}, we get
	\begin{equation*}
		m:=\min_{[-T,t_0-\delta]\times\Sph^{n-1}}w_{t_0}>0.
	\end{equation*}
Moreover,
	\begin{equation*}
	|w_\lambda(t,\omega)-w_{t_0}(t,\omega)|=|V(2\lambda-t,\omega)-V(2t_0-t,\omega)|\leq2\|\partial_t V\|_\infty|\lambda-t_0|.
	\end{equation*}
	
	Choose
	\begin{equation*}
		0<\varepsilon\leq\min\left\{\delta,\frac{m}{4(1+\|\partial_t V\|_\infty)}\right\}.
	\end{equation*}
	Then for every $\lambda\in[t_0,t_0+\varepsilon]$,
	\begin{equation*}
		w_\lambda\geq\frac{m}{2}\quad\text{on }[-T,t_0-\delta]\times\Sph^{n-1}.
	\end{equation*}
	Therefore its negative set satisfies
	\begin{equation*}
		N_\lambda\subset\bigl(( -\infty,-T)\cup(t_0-\delta,\lambda)\bigr)\times\Sph^{n-1}.
	\end{equation*}

	By our choices of $T$, $\delta$, and $\varepsilon$, we obtain
	\begin{equation*}
n\alpha\sup_{t<\lambda,\,\omega}\int_{N_\lambda}K_\lambda(t,s,\omega\cdot\eta)p(s,\eta)\,ds\,d\eta<\frac{1}{2}
	\end{equation*}
	 for every $\lambda\in[t_0,t_0+\varepsilon]$. By \eqref{w- le} and \eqref{M la finite},
	\begin{equation*}
		w_\lambda^-(t,\omega)\leq n\alpha M_\lambda\int_{N_\lambda}K_\lambda p\dd s\dd\eta
	\leq\frac{1}{2}M_\lambda.
	\end{equation*}
	Taking the supremum yields $M_\lambda\leq\frac{1}{2}M_\lambda$, so $M_\lambda=0.$ Consequently, $w_\lambda\geq0$ for every $\lambda\in [t_0,t_0+\varepsilon]$, which contradicts the definition of $t_0$ \eqref{def t0}. 
    
    Therefore, we proved that $w_{t_0}\equiv0$, which is exactly \eqref{V sym}. As a result, \eqref{u kelvin} holds for $R=e^{t_0}$. Finally, \eqref{R half mass} follows from a simple change of variable.
\end{proof}

By Theorem \ref{thm:kelvin}, a change of variable in Lemma \ref{rig:lem:translation-identity} gives the following corollary. 
\begin{corollary}\label{coro:translation-identity}
Let $u$ be a normal solution of \eqref{eq:euclidean} with $\gamma>0$ and let $f(x)=|x|^{n\gamma}e^{nu(x)}$. Then
\begin{equation*}
 \int_{\R^n}xf(x)\dd x=0.
\end{equation*}
\end{corollary}

\section{The radial normal solution}\label{sec:radial}
In this section, we prove Theorem \ref{thm radial}. We first factor the radial operator and derive an equation for the normalized density. Then we associate the equation with a free-energy functional that is convex under interpolation of inverse distribution functions and prove Theorem \ref{thm radial} by a variational approach.

\subsection{Properties of the radial solution and the radial kernel}

By \eqref{eq:mode-symbol}, the radial symbol $a_{n,0}(\xi)$ has a double zero at the origin, so we can write
\begin{equation}\label{eq:radial-factor}
 A_{n,0}=(-\partial_t^2)B_n,
 \qquad
 b_n(\xi):=\frac{a_{n,0}(\xi)}{\xi^2},
 \qquad
 \beta_n:=b_n(0)=2^{n-2}\Gamma\left(\frac n2\right)^2.
\end{equation}
The inverse of $B_n$ is convolution with a kernel $g_n$ satisfying
\begin{equation}\label{eq:g-hat}
 \widehat g_n(\xi)=\frac{1}{b_n(\xi)}
 =2^{2-n}\left|
 \frac{\Gamma(1+i\xi/2)}{\Gamma(n/2+i\xi/2)}
 \right|^2.
\end{equation}

\begin{lemma}\label{lem:radial-kernel-g}
	Let
	\begin{equation*}
	f_n(s)=e^{-2s}(1-e^{-2s})^{n/2-2}\mathbf 1_{(0,\infty)}(s),
	\qquad
	\widetilde f_n(s)=f_n(-s).
	\end{equation*}
	Then 
	\begin{equation}\label{eq:g-autocorrelation-revised}
		g_n=d_n f_n*\widetilde f_n,
	\end{equation}
    where $d_n=\frac{2^{4-n}}{\Gamma(n/2-1)^2}$.	Consequently, $g_n$ is even and positive, and $g_n(t)\le Ce^{-2|t|}$ for $|t|\ge1$ and some constant $C>0$. The integral of $g_n$ on $\mathbb{R}$ is
\begin{equation}\label{int gn}
		\int_\R g_n(t)\dd t=\frac{1}{\beta_n}.
	\end{equation}
 Moreover, $g_n'(t)<0$ for $t>0$. For $n=3$, $g_3$ has a logarithmic singularity at the origin and is smooth away from it.
\end{lemma}

\begin{proof}
The Beta identity \eqref{eq:beta-gamma} gives
\begin{equation*}
 \frac{\Gamma(1+i\xi/2)}{\Gamma(n/2+i\xi/2)}
 =\frac{1}{\Gamma(n/2-1)}
 \int_0^1 r^{i\xi/2}(1-r)^{n/2-2}\dd r.
\end{equation*}
With $r=e^{-2s}$, this becomes
\begin{equation*}
 \frac{\Gamma(1+i\xi/2)}{\Gamma(n/2+i\xi/2)}
 =\frac{2}{\Gamma(n/2-1)}\int_0^\infty e^{-i\xi s}f_n(s)\dd s
 =\frac{2}{\Gamma(n/2-1)}\widehat f_n(\xi).
\end{equation*}
Using \eqref{eq:g-hat} and $\widehat{\widetilde f_n}=\overline{\widehat f_n}$ proves \eqref{eq:g-autocorrelation-revised} with $d_n=\frac{2^{4-n}}{\Gamma(n/2-1)^2}$. Evenness and positivity follow immediately.
	
For $t>0$, using $\widetilde f_n(s)=f_n(-s)$ and changing variables $s\mapsto-s$, we obtain
\begin{equation}\label{gn convol}
 g_n(t) =d_n\int_\R f_n(t-s)f_n(-s)\dd s=d_n\int_0^\infty f_n(t+s)f_n(s)\dd s,
\end{equation}
For $t\ge1$ and $s>0$, $f_n(t+s)\le Ce^{-2t}e^{-2s}$. Since $f_n\in L^1(\R)$, this gives $g_n(t)\le Ce^{-2t}$. The corresponding bound for $t\le-1$ follows by evenness.
	
The integral identity \eqref{int gn} follows from
	\begin{equation*}
		\int_\R g_n(t)\dd t=\widehat g_n(0)=\frac{1}{\beta_n}.
	\end{equation*}
    
	Now we prove that $g_n'(t)<0$ for $t>0$.
For $n>4$, direct differentiation gives
	\begin{equation*}
	\frac{\dd^2}{\dd s^2}\log f_n(s)
	=-\frac{4(\frac{n}{2}-2)e^{-2s}}{(1-e^{-2s})^2}<0\text{ for }s>0.
	\end{equation*}
	Set $K(x,y)=f_n(x-y)$ and fix $x_1<x_2$ and $y_1<y_2$.
If $y_2\ge x_1$, then $K(x_1,y_2)=0$, so
\[
 \det[K(x_i,y_j)]_{i,j=1}^2
 =K(x_1,y_1)K(x_2,y_2)\ge0.
\]
If $y_2<x_1$, all four arguments are positive. Since
$(\log f_n)'$ is strictly decreasing and $x-y_2<x-y_1$,
integration over $x\in[x_1,x_2]$ gives
\[
 \log\frac{f_n(x_2-y_2)}{f_n(x_1-y_2)}
 >
 \log\frac{f_n(x_2-y_1)}{f_n(x_1-y_1)}.
\]
Exponentiating and rearranging yields
\[
 \det[K(x_i,y_j)]_{i,j=1}^2
 =f_n(x_1-y_1)f_n(x_2-y_2)
  -f_n(x_1-y_2)f_n(x_2-y_1)>0.
\]
By \eqref{eq:g-autocorrelation-revised}, we get $g_n(x-y)=d_n\int_\R K(x,z)K(y,z)\dd z$.
Expanding the product of the two determinants and applying
Fubini's theorem and the symmetry, we get

\[
 \det[g_n(x_i-y_j)]_{i,j=1}^2
 =d_n^2\iint_{z_1<z_2}
 \det[K(x_i,z_j)]_{i,j=1}^2
 \det[K(y_i,z_j)]_{i,j=1}^2
 \dd z_1\dd z_2>0.
\]
The strict inequality follows because both determinants are
nonnegative on $z_1<z_2$ and strictly positive on the set
$z_1<z_2<\min\{x_1,y_1\}$, which has positive measure.
The integrand is nonnegative everywhere and strictly positive
when $z_1<z_2<\min\{x_1,y_1\}$.
Since $f_n(0)=0$ for $n>4$, differentiation of
\eqref{gn convol} and integration by parts give
\[
 \begin{aligned}
g_n'(t)=d_n\int_0^\infty f_n'(s+t)f_n(s)\dd s=-d_n\int_0^\infty f_n(s+t)f_n'(s)\dd s.
 \end{aligned}
\]
Averaging these expressions yields
\[
 g_n'(t)=\frac{d_n}{2}\int_0^\infty
 f_n(s+t)f_n(s)
 \left[
 \frac{f_n'(s+t)}{f_n(s+t)}
 -\frac{f_n'(s)}{f_n(s)}
 \right]\dd s<0,
 \qquad t>0,
\]
because the preceding derivative formula shows that $f_n'/f_n$ is strictly decreasing.
 
For $n=4$, a direct computation gives $\hat{g}_4(\xi)=\frac{1}{\xi^2+4}$ and  $g_4(t)=\frac14e^{-2|t|}$. 
  
For $n=3$,  differentiating \eqref{gn convol} under the integral sign gives
   \begin{equation*}
   	g_3'(t)=d_3\int_0^\infty f_3'(s+t)f_3(s)\dd s<0,
   	\qquad t>0.
   \end{equation*}
This proves the strict decrease of $g_n$ on $(0,\infty)$.
   
   By \eqref{eq:g-hat}, $\widehat g_3(\xi)\sim |\xi|^{-1}$ as $|\xi|\to\infty$. Then the Fourier inversion gives the logarithmic singularity of $g_3$ at the origin. The same multiplier formula shows that $g_3$ is smooth on $\R\setminus\{0\}$.
\end{proof}

 We introduce the concept of total positivity of order two. For more background on total positivity, see \cite{Karlin,Pinkus}.
\begin{definition}\label{def:TP2}
A kernel $K:\R\times\R\to(0,\infty)$ is called \emph{strictly totally positive of order two} if
\[
 \det\begin{pmatrix}
 K(t_1,s_1)&K(t_1,s_2)\\
 K(t_2,s_1)&K(t_2,s_2)
 \end{pmatrix}>0
\]
whenever $t_1<t_2$ and $s_1<s_2$.
\end{definition}

To study radial solutions, we need to understand the fundamental solution of the radial operator $A_{n,0}$. Let $\mathcal K_n$ be the even function determined by
\begin{equation*}
\mathcal K_n(0)=\mathcal K_n'(0)=0,
\qquad \mathcal K_n''=g_n.
\end{equation*}
Then $\mathcal K_n$ is strictly convex since $g_n>0$. Moreover, we can write
 \begin{equation}\label{eq:interaction-kernel}
	\mathcal{K}_n(t)=\int_0^{|t|}(|t|-s)g_n(s)\dd s.
\end{equation}

Since $g_n$ is even and $\int_\R g_n(t)\dd t=\frac{1}{\beta_n}$, we have
\begin{equation*}
\mathcal K_n'(t)\longrightarrow
\begin{cases}
	(2\beta_n)^{-1},&t\to+\infty,\\
	-(2\beta_n)^{-1},&t\to-\infty.
\end{cases}
\end{equation*}
Moreover, \eqref{eq:interaction-kernel} and the exponential decay of $g_n$ give
\begin{equation}\label{eq:K-linear-revised}
	\mathcal K_n(t)=\frac{|t|}{2\beta_n}-\int_{0}^\infty sg_n(s)\dd s+O(e^{-2|t|}),\qquad |t|\to \infty.
\end{equation}
Noticing that $(-\partial_t^2)B_n(-\mathcal K_n)=B_n\mathcal K_n''=B_ng_n=\delta_0$, we get
\begin{equation}\label{eq:K-fundamental-revised}
	A_{n,0}(-\mathcal K_n)=\delta_0
\end{equation}
in the sense of distributions.
Applying $B_n^{-1}$ to \eqref{eq:radial-equation} gives
\begin{equation}\label{eq:strict-concavity}
 -v_\alpha''=(n-1)!\,g_n*e^{nv_\alpha}>0.
\end{equation}
Thus every radial normal solution is strictly concave. Integrating this identity and using Tonelli's theorem together with \eqref{int gn} and \eqref{eq:v al infty}, we obtain
\[
 2\alpha
 =\int_\R(-v_\alpha'')\dd t
 =\frac{(n-1)!}{\beta_n}\int_\R e^{nv_\alpha}\dd t.
\]
Therefore
\begin{equation}\label{int e nv}
 (n-1)!\int_\R e^{nv_\alpha}\dd t=2\beta_n\alpha.
\end{equation}
The duplication formula
\cite[\S5.5(iii), Eq.~(5.5.5)]{DLMF} gives $(n-1)!|\Sph^n|/|\Sph^{n-1}|=2\beta_n$.
Consequently,
\begin{equation}\label{eq:W-mass}
 \int_\R W_\alpha(t)\dd t
 =2^{n-1}n\Gamma(n/2)^2\alpha
 =2n\beta_n\alpha.
\end{equation}

Set
\begin{equation}\label{def M rho C}
	 M_\alpha:=\frac{2\beta_n}{(n-1)!}\alpha,	\qquad	\rho_\alpha(t):=\frac{e^{nv_\alpha(t)}}{M_\alpha},	\qquad C_{n,\alpha}:=n!M_\alpha=2n\alpha\beta_n.
\end{equation}
Then $\rho_\alpha$ is a probability density, i.e., $\int_{\R}\rho_\alpha=1$ and $\rho_\alpha\geq 0$, and
\begin{equation*}
 A_{n,0}\log\rho_\alpha=C_{n,\alpha}\rho_\alpha.
\end{equation*}
Since $A_{n,0}(-\mathcal K_n)=\delta_0$, we can write
\begin{equation*}
 A_{n,0}\bigl(\log\rho_\alpha+C_{n,\alpha}\mathcal K_n*\rho_\alpha\bigr)=0.
\end{equation*}

The expression in parentheses is a tempered function. Indeed, normality gives $\log\rho_\alpha(t)=-n\alpha|t|+O(1)$, while \eqref{eq:K-linear-revised} and the exponential decay of $\rho_\alpha$ give
$|\mathcal K_n*\rho_\alpha(t)|\le C(1+|t|)$.

\begin{lemma}\label{lem:tempered-radial-kernel}
	If $A_{n,0}h=0$ for a tempered distribution $h$, then $h(t)=at +b$ for some constants $a$ and $b$.
\end{lemma}
\begin{proof}
	Taking Fourier transforms gives $a_{n,0}\widehat h=0$.  Since $a_{n,0}$ is nonzero away from the origin, we get $\operatorname{supp}\widehat h\subset\{0\}$. By \cite[Theorem~6.24(d) $\And$ Theorem~6.25]{Rudin1991}, we get $\widehat h=\sum_{|k|\leq N}c_k\delta^{(k)}_0$. Since $\operatorname{supp}\widehat h\subset\{0\}$ and $b_n(0)>0$, the equation $\xi^2b_n\widehat h=0$ implies $\xi^2\widehat h=0$. By \cite[Theorem~3.1.16, p.~67]{HormanderI}, $\xi\widehat h=c\delta_0$ for some constant $c$. By the fact that $\xi\delta_0'=-\delta_0$, a second application of the same theorem gives $\widehat h=a\delta_0'+b\delta_0$ for some constants $a$ and $b$. Taking the inverse Fourier transform completes the proof.
\end{proof}

It follows from the above lemma that $\log\rho_\alpha+C_{n,\alpha}\mathcal K_n*\rho_\alpha=at +b$ for some constants $a$ and $b$. 

%\begin{definition}\label{def:affine}
%A function $h:\R\to\R$ is called \emph{affine} if there exist constants $a,b\in\R$ such that $h(t)=at+b$ for $t\in\R$.
%\end{definition}

We next examine the asymptotic behavior of this affine function. In fact, \eqref{eq:K-linear-revised} and the exponential decay of
$\rho_\alpha$ give
\[
 (\mathcal K_n*\rho_\alpha)(t)=\frac{|t|}{2\beta_n}+O(1),
 \qquad
 \log\rho_\alpha(t)=-n\alpha|t|+O(1).
\]
Since $C_{n,\alpha}/(2\beta_n)=n\alpha$, it follows that
$\log\rho_\alpha+C_{n,\alpha}\mathcal K_n*\rho_\alpha$
is bounded.  Hence every radial normal solution satisfies
\begin{equation}\label{eq:density-EL}
	\log\rho_\alpha+C_{n,\alpha}\mathcal K_n*\rho_\alpha=\lambda_\alpha
\end{equation}
for a constant $\lambda_\alpha$.

\subsection{Existence and uniqueness by the variational method}
Let
\begin{equation*}
    \mathcal{P}_1(\R):=\{\mu: \mu \text{ is a Borel probability measure on }\R, \int_{\R}|t|\dd\mu<\infty \}.
\end{equation*}
Define the entropy by
\begin{equation*}
    \operatorname{Ent}(\mu)=\begin{cases}
        &\int_\R\rho\log\rho\dd t,\quad \mu=\rho \dd t,\\
        &+\infty,\quad \mu\not\ll \dd t.
    \end{cases}
\end{equation*}
Then equation \eqref{eq:density-EL} is the Euler--Lagrange equation associated with the energy
\begin{equation}\label{eq:free-energy}
 \mathcal F_{n,\alpha}(\mu) =\operatorname{Ent}(\mu)+\frac{C_{n,\alpha}}{2} \iint_{\R^2}\mathcal K_n(t-s)\dd \mu(t)\dd \mu(s)
\end{equation}
The entropy takes value in $\R\cup\{+\infty\}$ and the interaction term is finite since $\mathcal{K}_n$ has at most linear growth. We regard $\mathcal F_{n,\alpha}$ as an extended-value functional on $\mathcal{P}_1(\R)$. In what follows we use the same symbol for a measure and its density function (if exists) by a slight abuse of notation.

We first establish the coercivity of $\mathcal{F}_{n,\alpha}$ modulo translations.
\begin{proposition}\label{prop:coercivity}
	Let $\rho$ be a probability measure translated so that $0$ is its median, that is,
\begin{equation}\label{0 median}
\rho\bigl((-\infty,0]\bigr)\ge\frac{1}{2}\text{ and } \rho\bigl([0,\infty)\bigr)\ge\frac{1}{2}.
\end{equation}
	 Then
	\begin{equation*}
	\mathcal F_{n,\alpha}(\rho)
	\ge c\int_\R|t|\rho(t)\dd t-C
	\end{equation*}
	for constants $c,C>0$ depending only on $n$ and $\alpha$.
\end{proposition}

\begin{proof} 
	By \eqref{0 median}, we obtain
	\begin{equation*}
		\iint_{\R^2}|t-s|\rho(t)\rho(s)\dd t \dd s
		\ge\frac{1}{2}\int_\R|t|\rho(t) \dd t.
	\end{equation*}
From \eqref{eq:K-linear-revised}, we have $\mathcal K_n(t)\ge \frac{|t|}{2\beta_n} -C_0$. Therefore,
\begin{equation}\label{int ge}
	\iint_{\R^2}\mathcal K_n(t-s)\rho(t)\rho(s)\dd t\dd s \ge \frac{1}{4\beta_n}\int_\R|t|\rho(t) \dd t-C_0.
\end{equation}

	For $\sigma>0$, we define an auxiliary function
    \begin{equation*}
		q_\sigma(t)=\frac{\sigma}{2}e^{-\sigma|t|},
	\end{equation*} 
    Since $x\mapsto x\log x$ is convex on $(0,\infty)$ and its derivative is $1+\log x$, we have
\begin{equation}\label{xlogx convex}
 x\log x\ge y\log y+(1+\log y)(x-y)
 =x\log y+x-y \text{ for }x\ge0\text{ and }y>0,   
\end{equation}
where the case $x=0$ follows by continuity. Taking $x=\rho(t)$ and $y=q_\sigma(t)>0$ gives
\[
 \rho(t)\log\rho(t)
 \ge \rho(t)\log q_\sigma(t)+\rho(t)-q_\sigma(t).
\]
	Integrating, we obtain
	\begin{equation}\label{ent ge}
			\int_{\R}\rho\log\rho\dd t
		\ge\log\frac\sigma2-\sigma\int_{\R}|t|\rho(t)\dd t.
	\end{equation}
	Combining \eqref{int ge} and \eqref{ent ge}, the assertion follows by choosing $0<\sigma<\frac{C_{n,\alpha}}{8\beta_n}$.
\end{proof}

\begin{proposition}\label{prop:radial-existence}
For every $n\geq 3$ and $\alpha>0$, the infimum of $\mathcal F_{n,\alpha}$ on $\mathcal{P}_1(\R)$ is finite and attained.  Every minimizer is positive, satisfies \eqref{eq:density-EL} and gives a radial normal solution of \eqref{eq:euclidean}.
\end{proposition}

\begin{proof}
Choose a minimizing sequence $\rho_j$ and translate each density so that $0$ is a median, as in \eqref{0 median}. Proposition~\ref{prop:coercivity} gives
\[
 \sup_j\int_\R |t|\rho_j(t)\dd t\le M_0<\infty.
\]
By Prokhorov's theorem \cite[Theorem~3.10.3]{Durrett2019}, after passing to a subsequence, the probability measures $\rho_j(t)\dd t$ converge weakly to a probability measure $\mu$, with
$\int_\R|t|\dd\mu(t)\le M_0$.

Since $\mathcal{K}_n\geq 0$, the minimizing sequence satisfies $\operatorname{Ent}(\rho_j)\le \mathcal{F}_{n,\alpha}\le C$. By the lower semicontinuity of the entropy on measures with uniformly bounded first moments (see \cite[Proposition 2.1 (1)--(2)]{Santambrogio2016}), we have
\begin{equation*}
    \operatorname{Ent}(\mu)\le \liminf_{j\to\infty}\operatorname{Ent}(\rho_j)\le C.
\end{equation*}
Therefore $\mu=\rho\dd t$ for some density $\rho$ in $\mathcal{P}_1(\R)$.

Since $\mathcal K_n$ is continuous and nonnegative, weak convergence of the product measures also yields
\begin{equation*}
    \iint_{\R^2}\mathcal K_n(t-s)\rho(t)\rho(s)\dd t\dd s
 \le\liminf_{j\to\infty}\iint_{\R^2}\mathcal K_n(t-s)\rho_j(t)\rho_j(s)\dd t\dd s.
\end{equation*}
Combining these inequalities shows that $\rho$ is a minimizer with finite entropy and finite first moment. 

Moreover, since
\begin{equation*}
 (\mathcal K_n*\rho)(t)=\frac{|t|}{2\beta_n}+O(1),    
\end{equation*}
the probability density
\begin{equation*}
 \bar\rho(t):=\frac{\exp(-C_{n,\alpha}(\mathcal K_n*\rho)(t))}{\int_\R\exp(-C_{n,\alpha}(\mathcal K_n*\rho)(s))\dd s}    
\end{equation*}
is well defined and satisfies $c e^{-n\alpha|t|}\le\bar\rho(t)\le C e^{-n\alpha|t|}$ for $t\in\R$ and some $c, C>0$.  For $\rho_\varepsilon=(1-\varepsilon)\rho+\varepsilon\bar\rho$ and $\delta=\bar\rho-\rho$, convexity of the entropy and expansion of the interaction term give
\begin{equation*}
    \begin{aligned}
      0\le\mathcal F_{n,\alpha}(\rho_\varepsilon)
 -\mathcal F_{n,\alpha}(\rho)
 \le&\varepsilon \left(\operatorname{Ent}(\bar\rho)-\operatorname{Ent}(\rho)+C_{n,\alpha}\int_{\R}(\mathcal K_n*\rho)\delta\dd t\right)\\
 &+\frac{C_{n,\alpha}\varepsilon^2}{2} \iint_{\R^2}\mathcal K_n(t-s)\delta(t)\delta(s)\dd t\dd s\\
 =&-\varepsilon\int_\R\rho\log\frac{\rho}{\bar\rho}\dd t+O(\varepsilon^2).   
    \end{aligned}
\end{equation*}
The remainder is finite because both densities have finite first moment and $\mathcal K_n$ has at most linear growth. Dividing by $\varepsilon$ and letting $\varepsilon\downarrow0$ and applying inequality \eqref{xlogx convex}, we obtain that  $\rho=\bar\rho$ almost everywhere. Thus $\rho$ has a positive continuous representative and satisfies \eqref{eq:density-EL}.

With $M_\alpha$ as in \eqref{def M rho C}, set $v_\alpha=n^{-1}\log(M_\alpha\rho)$. Then \eqref{eq:K-fundamental-revised} gives \eqref{eq:radial-equation}. The symbol $a_{n,0}$ is elliptic of order $n$, so local elliptic regularity \cite[Theorem~18.1.29]{HormanderIII} first gives $v_\alpha\in H^n_{\mathrm{loc}}(\R)$, since $e^{nv_\alpha}\in L^2_{\mathrm{loc}}(\R)$.
Iterating this regularity gain together with Sobolev composition gives $v_\alpha\in H^{kn}_{\mathrm{loc}}(\R)$ for every integer $k\ge1$ and hence $v_\alpha\in C^\infty(\R)$. Using \eqref{eq:K-linear-revised} and the finite first moment of $\rho$, we obtain the following asymptotics by the dominated convergence theorem:
\begin{equation}\label{eq:v toinfty}
 v_\alpha(t)=
 \begin{cases}
  \alpha t+c_-+o(1),&t\to-\infty,\\
  -\alpha t+c_++o(1),&t\to+\infty.
 \end{cases}
\end{equation}
Differentiating \eqref{eq:density-EL} also gives
\eqref{eq:v al infty}.
Consequently,
\[
 u_\alpha(x)=v_\alpha(\log|x|)-\alpha\log|x|
\]
is continuous on $\R^n$ with $u_\alpha(0)=c_-$, has at most logarithmic growth and
\[
 \int_{\R^n}|x|^{n\gamma}e^{nu_\alpha}\dd x=|\Sph^{n-1}|M_\alpha<\infty.
\]
The Emden--Fowler transformation gives the equation away from the origin and Lemma~\ref{lem:removability-normality} shows that $u_\alpha$ is a radial normal solution on $\R^n$.
\end{proof}

It remains to prove uniqueness up to translation. Inspired by \cite{McCann1997,Villani}, we consider linear interpolation of inverse distribution functions of the density. To be more precise, for a probability measure $\rho(t)\dd t$, we define
\begin{equation*}
 F_\rho(t)=\int_{-\infty}^t\rho(s)\dd s,\qquad X_\rho(p)=\inf\{t:F_\rho(t)\ge p\},\quad 0<p<1.
\end{equation*}
It holds that
\begin{equation*}
	\int_\R\varphi(t)\rho(t)\dd t
	=\int_0^1\varphi(X_\rho(p))\dd p.
\end{equation*}
Moreover, if $\rho$ is positive and smooth,
\begin{equation*}
	\rho(X_\rho(p))X_\rho'(p)=1.
\end{equation*}

Therefore, for $X=X_\rho$, we can write
\begin{equation}\label{eq:quantile-energy}
 \mathcal F_{n,\alpha}(\rho)= \mathcal G_{n,\alpha}(X):=-\int_0^1\log X'(p)\dd p
 +\frac{C_{n,\alpha}}2\int_0^1\!\int_0^1
 \mathcal K_n(X(p)-X(q))\dd p\dd q.
\end{equation}
 Since $-\log$ and $\mathcal K_n$ are convex, the function $\theta\mapsto \mathcal G_{n,\alpha}(X_\theta)$, where $X_\theta:=\theta X_1+(1-\theta)X_0$ for given $X_0$ and $X_1$, is also convex. In addition, it is strictly convex modulo translations and equality in the convexity inequality forces that $X_1-X_0$ is constant. 
 
Now let $\rho_0,\rho_1$ correspond to two radial normal solutions at the same parameter $\alpha$ and define $X_0,X_1$ accordingly. Set 
\begin{equation*}
  D=X_1-X_0,\qquad  E(\theta)=\mathcal G_{n,\alpha}((1-\theta)X_0+\theta X_1).
\end{equation*}
By normality, $\rho_i(t)e^{n\alpha|t|}$ has positive limit as $t\to\infty$. Consequently,
\begin{equation*}
    D\in L^\infty(0,1),\qquad X_i'(p)\sim \frac{1}{n\alpha p},\ p\downarrow 0,\qquad X_i'(p)\sim \frac{1}{n\alpha (1-p},\ p\uparrow 1.
\end{equation*}
  Differentiating \eqref{eq:density-EL} gives
\begin{equation*}
 \left(\frac{1}{X_i'}\right)'(p)=-C_{n,\alpha}(\mathcal K_n'*\rho_i)(X_i(p)).
\end{equation*}
Using also the oddness of $\mathcal K_n'$, we obtain, for $i=0,1$,
\begin{align*}
 E'(i)=-\left[\frac{D}{X_i'}\right]_0^1
 +\int_0^1D(p)\left[\left(\frac{1}{X_i'}\right)'(p)
 +C_{n,\alpha}(\mathcal K_n'*\rho_i)(X_i(p))\right]\dd p=0,
\end{align*}
Since $E$ is convex, it is constant. Equality in the strict convexity inequality for $-\log$ implies $X_1'=X_0'$, and hence $X_1-X_0$ is constant. Therefore the solutions differ only by a translation in $t$, or equivalently by a dilation in Euclidean variables. In particular, the minimizer of $\mathcal{F}_{n,\alpha}$ is unique up to translation.

\subsection{Radial nondegeneracy and the global radial branch}

Translation invariance gives $L_{\alpha,0}v_\alpha'=0$. We prove that there are no other bounded radial kernel elements.

\begin{proposition}\label{prop:radial-nondegeneracy}
Let $\phi$ be a bounded radial solution of $A_{n,0}\phi=W_\alpha\phi$. Then $\phi\in\Span\{ v_\alpha'\}$.
\end{proposition}

\begin{proof}
We first claim that $\int_\R W_\alpha(t)\phi(t)\dd t=0$. In fact, the function $h:=\phi+\mathcal{K}_n*(W_\alpha\phi)$ satisfies $A_{n,0}h=0$ and it is a tempered distribution because of the exponential decay of $W_\alpha$ and the boundedness of $\phi$. By Lemma \ref{lem:tempered-radial-kernel}, $h=at+b$ for some $a, b\in\R$.

On the other hand, by \eqref{eq:K-linear-revised}, the dominated convergence theorem gives
\begin{equation*}
\frac{(\mathcal{K}_n* (W_\alpha\phi))(t)}{t}\longrightarrow
\begin{cases}
	\frac{1}{2\beta_n} \int_\R W_\alpha(s)\phi(s)\dd s,&t\to+\infty,\\
	-\frac{1}{2\beta_n} \int_\R W_\alpha(s)\phi(s)\dd s,&t\to-\infty.
\end{cases}
\end{equation*}
Since $\phi$ is bounded, we must have $\int_\R W_\alpha(t)\phi(t)\dd t=0$.

Set $\sigma(t)=W_\alpha(t)\phi(t)$. Then $\int_{\R}\sigma(t)\dd t=0$ and the equation for $\phi$ gives
\begin{equation}\label{eq:sigma}
	\frac{\sigma}{\rho_\alpha}
	+C_{n,\alpha}\mathcal{K}_n*\sigma=\text{constant}.
\end{equation}
Define $\zeta(t):=-\frac{1}{\rho_\alpha(t)} \int_{-\infty}^t\sigma(s)\dd s$. Then $\sigma=-(\rho_\alpha\zeta)'$. Since $\sigma=C_{n,\alpha}\rho_\alpha\phi$ and $\int_\R\sigma=0$, for $t\ge0$ we have
\begin{equation*}
    |\zeta(t)|=\frac{\left|\int_t^\infty\sigma(s)\dd s\right|}{\rho_\alpha(t)}
 \le C_{n,\alpha}\|\phi\|_\infty\frac{\int_t^\infty\rho_\alpha(s)\dd s}{\rho_\alpha(t)}
 \le C.
\end{equation*}
Here the last inequality follows from $\rho_\alpha(t)\sim e^{-n\alpha|t|}$ due to  \eqref{def M rho C} and \eqref{eq:v toinfty}. For $t\le0$, the same estimate follows directly from the defining integral over $(-\infty,t]$. Thus $\zeta$ is bounded.
Differentiating its definition gives
\begin{equation*}
    \zeta'=-\frac{\sigma}{\rho_\alpha}-\frac{\rho_\alpha'}{\rho_\alpha}\zeta
 =-C_{n,\alpha}\phi-nv_\alpha'\zeta.
\end{equation*}
Since $\phi$, $v_\alpha'$ and $\zeta$ are bounded, so is $\zeta'$. 

We have the following key identity.
\begin{equation}\label{Q al id}
	\begin{aligned}
		\mathcal Q_\alpha[\zeta]
		:=&\int_\R\rho_\alpha(t)|\zeta'(t)|^2\dd t+\frac{C_{n,\alpha}}2
		\iint_{\R^2}g_n(t-s)
		|\zeta(t)-\zeta(s)|^2\rho_\alpha(t)\rho_\alpha(s)\dd t\dd s\\
		=& \int_\R \rho_\alpha\zeta(\frac{\sigma}{\rho_\alpha}
		+C_{n,\alpha}\mathcal{K}_n*\sigma)'\dd t.
	\end{aligned}
\end{equation}
Indeed, using integration by parts, we can write
\begin{equation}\label{Q al 1}
	\begin{aligned}
	\int_\R \rho_\alpha\zeta(\frac{\sigma}{\rho_\alpha})'\dd t=&\int_\R \rho_\alpha\zeta\left(-\zeta'-(\log\rho_\alpha)'\zeta\right)'\dd t\\
	=&\int_\R\rho_\alpha|\zeta'|^2\dd t-\int_\R (\log\rho_\alpha)''\rho_\alpha\zeta^2\dd t\\
	=&\int_\R\rho_\alpha|\zeta'|^2\dd t+C_{n,\alpha}\iint g_n(t-s)	\rho_\alpha(t)\rho_\alpha(s)\zeta(t)^2\dd t\dd s,
	\end{aligned}
\end{equation}
where we used the equation $(\log\rho_\alpha)''=-C_{n,\alpha}g_n*\rho_\alpha$ in the last equality.

On the other hand, integration by parts gives
\begin{equation*}
	(\mathcal{K}_n'*\sigma)(t)
	=-\int_\R g_n(t-s)\rho_\alpha(s)\zeta(s)\dd s,
\end{equation*}
so
\begin{equation}\label{Q al 2}
	C_{n,\alpha}\int_\R\rho_\alpha(t)\zeta(t)(\mathcal{K}_n'*\sigma)(t)\dd t
	=-C_{n,\alpha}\iint_{\R^2} g_n(t-s)
	\rho_\alpha(t)\rho_\alpha(s)\zeta(t)\zeta(s)\dd t\dd s.
\end{equation}
\eqref{Q al id} then follows from \eqref{Q al 1} and \eqref{Q al 2}.

By \eqref{eq:sigma}, integration by parts once more yields that $\mathcal{Q}_\alpha[\zeta]=0$, so $\zeta$ is constant. Hence $\sigma=-c\rho_\alpha'$. Since $W_\alpha=n!M_\alpha\rho_\alpha$, we obtain that
\begin{equation*}
	\phi=\frac{\sigma}{W_\alpha}=-\frac{c}{n!M_\alpha}\frac{\rho_\alpha'}{\rho_\alpha}=-\frac{c}{(n-1)!M_\alpha}v_\alpha'.
\end{equation*}
\end{proof}

\section{Eigenvalues of the linearized operators}\label{sec:BS}

In this section, we analyze the eigenvalues of the linearized operator $L_{\alpha,k}$ through the associated Birman--Schwinger operator $\mathscr K_{\alpha,k}$ defined in \eqref{eq:BS operator}. We establish its positivity and oscillation properties and derive quantitative bounds for crossing parameters. Thess results leads to Theorem \ref{thm e value L}.

\subsection{The Green kernel at mode $k$}

For $k\ge1$, let $G_{n,k}$ be the convolution kernel of $A_{n,k}^{-1}$. Then we have the following representation.

\begin{proposition}\label{prop:mode-green}
For $n\ge3$, $k\ge1$ and $t\ge0$, we have
\begin{align}
 G_{n,k}(t)
 &=\frac{2^{1-n}e^{-kt}}{\Gamma(n/2)^2}
 \int_0^1s^{k-1}(1-s)^{n/2-1}
 (1-e^{-2t}s)^{n/2-1}\dd s,
 \label{eq:G-beta}\\
 &=\frac{2^{1-n}\Gamma(k)}{\Gamma(n/2)\Gamma(k+n/2)}e^{-kt}
 \Hyp\left(1-\frac n2,k;k+\frac n2;e^{-2t}\right).
 \label{eq:G-hyper}
\end{align}
Moreover,
\begin{equation}\label{eq:G-diagonal}
 G_{n,k}(0)=
 \frac{2^{1-n}\Gamma(k)\Gamma(n-1)}
 {\Gamma(n/2)^2\Gamma(n+k-1)}.
\end{equation}
The function $G_{n,k}$ is even, strictly positive and strictly
decreasing on $(0,\infty)$. Moreover,
\begin{equation}\label{eq:G-positive-minors}
 \det\bigl[G_{n,k}(t_i-s_j)\bigr]_{i,j=1}^2>0
 \qquad
 \text{whenever }t_1<t_2\text{ and }s_1<s_2.
\end{equation}
\end{proposition}

\begin{proof}
The proof is similar to the proof of Lemma \ref{lem:radial-kernel-g}. From the Beta identity \eqref{eq:beta-gamma} and a change of variables, we can write
\begin{equation*}
 \frac{\Gamma((k+i\xi)/2)}{\Gamma((k+n+i\xi)/2)}
 =\frac2{\Gamma(n/2)}\int_0^\infty
 e^{-(k+i\xi)s}(1-e^{-2s})^{n/2-1}\dd s.
\end{equation*}
Recall that the Fourier symbol of $A_{n,k}$ is $a_{n,k}(\xi)$, given in \eqref{eq:mode-symbol}. Then, by letting
\begin{equation*}
 h_{n,k}(s)=e^{-ks}(1-e^{-2s})^{n/2-1}\mathbf1_{(0,\infty)}(s),
\end{equation*}
we have
\begin{equation*}
 G_{n,k}=\frac{2^{2-n}}{\Gamma(n/2)^2}h_{n,k}*\widetilde h_{n,k},
\end{equation*}
which is exactly \eqref{eq:G-beta}. Euler's formula \eqref{eq:hyp-euler} then gives \eqref{eq:G-hyper} and \eqref{eq:G-diagonal} follows by evaluating at $t=0$.

The convolution representation gives evenness and strict positivity.
Moreover,
\[
 \frac{\dd^2}{\dd s^2}\log h_{n,k}(s)
 =-\frac{4(\frac n2-1)e^{-2s}}{(1-e^{-2s})^2}<0,
 \qquad s>0.
\]
Since $h_{n,k}>0$ on $(0,\infty)$ and vanishes on
$(-\infty,0]$, the determinant argument in the proof of
Lemma~\ref{lem:radial-kernel-g} applies with $h_{n,k}$ in
place of $f_n$ and gives \eqref{eq:G-positive-minors}.
For $0\le s<t$, taking
$(t_1,t_2)=(s,t)$ and $(s_1,s_2)=(0,s+t)$ in that inequality
and using evenness gives $G_{n,k}(s)^2-G_{n,k}(t)^2>0$. Since $G_{n,k}>0$, this proves strict decrease on $(0,\infty)$.
\end{proof}

The Beta representation \eqref{eq:G-beta} also extends to a positive real angular parameter $\kappa$. In fact, after the
substitution $x=e^{-|t|}s$, we can define
\begin{equation}\label{eq:G-kappa}
 G_{n,\kappa}(t)=\frac{2^{1-n}}{\Gamma(n/2)^2}
 \int_0^{e^{-|t|}}x^{\kappa-1}
 \bigl(1-2x\cosh|t|+x^2\bigr)^{n/2-1}\dd x.
\end{equation}
The integrand is positive on the indicated interval.  Consequently
$G_{n,\kappa}(t)$ is strictly decreasing in $\kappa$ for every fixed $t$.

\subsection{The principal eigenvalue}

For $k\ge1$, we take $A_{n,k}$ as a self-adjoint Fourier multiplier on $L^2(\R)$ with domain
\begin{equation*}
D(A_{n,k})=H^n(\R).
\end{equation*}
Its symbol $a_{n,k}(\xi)$ is continuous, positive and satisfies $a_{n,k}(\xi)\sim 1+|\xi|^n$.  Since $G_{n,k}\ge 0$,
\begin{equation*}
\frac1{a_{n,k}(\xi)}
=\int_\R e^{-it\xi}G_{n,k}(t)\dd t
\le\int_\R G_{n,k}(t)\dd t
=\frac1{a_{n,k}(0)}.
\end{equation*}
Hence $a_{n,k}(\xi)\ge a_{n,k}(0)$. 

Under the Fourier transform, $A_{n,k}$ is unitarily equivalent to the maximal multiplication operator
 \begin{equation*}
 	M_{a_{n,k}}\hat{f}(\xi)=a_{n,k}\hat{f}(\xi).
 \end{equation*}
 It is well known that the spectrum of a multiplication operator is the essential range of its multiplier. See, for example, \cite[Section~VIII.3]{ReedSimonI}. Since $a_{n,k}$ is continuous, its essential range agrees with its ordinary range. Moreover, the image of the connected set $[0, \infty)$ under $a_{n,k}$ is a connected unbounded interval containing its global minimum $a_{n,k}(0)$. Consequently,
\begin{equation*}
	\spec(A_{n,k})=[a_{n,k}(0),\infty),
	\qquad a_{n,k}(0)=2^n\frac{\Gamma((n+k)/2)^2}{\Gamma(k/2)^2}>0.
\end{equation*}

The exponential decay of $W_\alpha$ implies that multiplication by $W_\alpha$ is relatively compact with respect to $A_{n,k}$. Weyl's theorem (see \cite[Chapter IV, Section 5]{Kato1995} or \cite[Theorem~XIII.14]{ReedSimonIV}) gives
\begin{equation*}
	\spec_{\ess}(L_{\alpha,k})=[a_{n,k}(0),\infty).
\end{equation*}
In particular, zero lies strictly below the essential spectrum.  If $0\in\spec(L_{\alpha,k})$, then zero is an isolated eigenvalue of finite multiplicity.

\medskip

Motivated by the proof of the Birman--Schwinger principle \cite[Theorem~XIII.10]{ReedSimonIV}, we introduce the Birman--Schwinger operator
	\begin{equation}\label{eq:BS operator}
		\mathscr K_{\alpha,k}:=
		W_\alpha^{\frac{1}{2}}A_{n,k}^{-1}W_\alpha^{\frac{1}{2}}		\quad\hbox{on }L^2(\R).
	\end{equation}
Then its integral kernel is given by
\begin{equation}\label{int kernel K}
W_\alpha(t)^{\frac{1}{2}}G_{n,k}(t-s)W_\alpha(s)^{\frac{1}{2}}.
\end{equation}

Before we state the results of this section, we recall some definitions concerning operators on a Hilbert space.
\begin{definition}
Let $\{e_j\}_{j=1}^\infty$ be an orthonormal basis of a Hilbert space with inner product $\langle \cdot ,\cdot \rangle$. An operator $T$ is \emph{Hilbert--Schmidt} if
\begin{equation*}
	\norm{T}_{\mathfrak S_2}^2:=\sum_{j=1}^\infty \|Te_j\|^2<\infty.
\end{equation*}
$T$ is \emph{trace class} if
\begin{equation*}
	\norm{T}_{\mathfrak S_1}:= \tr(\sqrt{T^*T})<\infty,
\end{equation*}
where the \emph{trace} is 
\begin{equation*}
	\tr(T)=\sum_{j=1}^\infty\langle Te_j, e_j\rangle.
\end{equation*}
\end{definition}
It is easy to see that these definitions do not depend on the choice of the orthonormal basis. We refer to \cite[p. 207, Theorems~VI.21 $\And$ VI.24]{ReedSimonI} for these standard facts.

We have the following basic properties of the Birman--Schwinger operator $\mathscr K_{\alpha,k}$.
\begin{proposition}\label{prop:BS-basic}
	For every $n\ge3$, $\alpha>0$ and $k\ge1$, the operator $\mathscr K_{\alpha,k}$ is positive, self-adjoint, compact, trace class, injective and positivity-improving in the sense that
	\begin{equation*}
		  f\ge0,\ f\not\equiv0
		\quad\Longrightarrow\quad
		\mathscr{K}_{\alpha,k}f>0\quad\text{on }\R. 
	\end{equation*}
	Its trace is
	\begin{equation}\label{eq:exact-trace}
		\tr\mathscr K_{\alpha,k}
		=\frac{\alpha}{\cA_{n,k}},
	\end{equation}
	where $\cA_{n,k}$ is defined in \eqref{def ak bk}. Its principal eigenvalue $\mu_1(\alpha,k)$ is simple and has a strictly positive eigenfunction. $\mathscr K_{\alpha,k}$ is strictly totally positive of order two in the sense of Definition \ref{def:TP2}. Moreover,
	\begin{equation}\label{equiv L K}
	0\in\spec(L_{\alpha,k})
	\quad\Longleftrightarrow\quad
	1\in\spec(\mathscr K_{\alpha,k}),
	\end{equation}
and multiplicities agree. 
\end{proposition}

\begin{proof}
	Let $T_{\alpha,k}=A_{n,k}^{-\frac{1}{2}}W_\alpha^{\frac{1}{2}}$.  Then $\mathscr K_{\alpha,k}=T_{\alpha,k}^*T_{\alpha,k}$. Let $Q_{n,k}$ be the convolution kernel of $A_{n,k}^{-\frac{1}{2}}$.  By Plancherel's identity,
	\begin{equation*}
	\int_\R|Q_{n,k}(t)|^2\dd t
	=\frac1{2\pi}\int_\R\frac{\dd\xi}{a_{n,k}(\xi)}
	=G_{n,k}(0).
	\end{equation*}
	Consequently, since  $W_\alpha\in L^1$, we get
	\begin{align*}
		\norm{T_{\alpha,k}}_{\mathfrak S_2}^2=\iint_{\R^2}		|Q_{n,k}(t-s)|^2W_\alpha(s)\dd s\dd t=G_{n,k}(0)\int_\R W_\alpha(s)\dd s<\infty.
	\end{align*}
	$T_{\alpha,k}$ is Hilbert--Schmidt by	\cite[Theorem~VI.23]{ReedSimonI} and compact by	\cite[Theorem~VI.22(e)]{ReedSimonI}. Hence $\mathscr K_{\alpha,k}=T_{\alpha,k}^*T_{\alpha,k}$ is trace class.  Moreover,
	\begin{equation*}
	\tr\mathscr K_{\alpha,k}
	=\norm{T_{\alpha,k}}_{\mathfrak{S}_2}^2
	=G_{n,k}(0)\int_\R W_\alpha(t)\dd t.
	\end{equation*}
	Substituting \eqref{eq:G-diagonal} and \eqref{eq:W-mass} gives \eqref{eq:exact-trace}.
	
	Positivity and self-adjointness follow from the factorization.  Since $W_\alpha>0$ and $A_{n,k}^{-\frac{1}{2}}$ is injective, $T_{\alpha,k}$ and hence $\mathscr K_{\alpha,k}$ are injective. The integral kernel \eqref{int kernel K} is strictly positive, so $\mathscr{K}_{\alpha,k}$ is positivity-improving. The Perron--Frobenius theorem for compact positivity-improving operators	\cite[Theorem~V.6.6, p.~337]{Schaefer1974} then shows that its principal eigenvalue is simple and has a strictly positive eigenfunction.
	
		By \eqref{eq:G-positive-minors}, for $t_1<t_2$ and $s_1<s_2$,
\[
 \begin{aligned}
 \det\left[
 W_\alpha(t_i)^{1/2}G_{n,k}(t_i-s_j)
 W_\alpha(s_j)^{1/2}
 \right]_{i,j=1}^2=
 \prod_{i=1}^2
 \bigl[W_\alpha(t_i)W_\alpha(s_i)\bigr]^{1/2}
 \det\bigl[G_{n,k}(t_i-s_j)\bigr]_{i,j=1}^2>0.
 \end{aligned}
\]
Thus the integral kernel \eqref{int kernel K} satisfies
Definition~\ref{def:TP2}.
		
	Finally, \eqref{equiv L K} follows directly from the definition.
	\end{proof}
	
The nondegeneracy in Corollary \ref{coro:nondegen} concerns the bounded solutions of the linearized operator. The following lemma shows that bounded kernel and the kernel in the $L^2$ sense coincides except for the radial kernel.
\begin{lemma}\label{lem bdd=spec}
    Let $k\ge 1$ and $\alpha>0$. Then every bounded distributional solution of $L_{\alpha,k}\phi=0$ belongs to $H^n(\R)$. Consequently, the bounded kernel of $L_{\alpha,k}$ agrees with its $L^2$ kernel.
\end{lemma}
\begin{proof}
    Let $\varphi$ be a bounded distributional solution of $L_{\alpha,k}\phi=0$. Since $W_\alpha\in L^2(\R)\cap L^\infty(\R)$, we have $W_\alpha\phi\in L^2(\R)$. By \eqref{eq:mode-symbol},
    \begin{equation*}
        a_{n,k}(\xi)\ge c_{n,k}(1+|\xi|^2)^{\frac{n}{2}}.
    \end{equation*}
    Therefore $\psi:=A_{n,k}^{-1}(W_\alpha\phi)\in H^n(\R)$. The difference $\phi-\psi$ is a tempered distribution and satisfies $a_{n,k}(\xi)(\widehat{\phi-\psi})=0$. Since $a_{n,k}$ is positive and smooth, we conclude that $\phi=\psi$. 
    
    Conversely, if $\mathscr{K}_{\alpha,k}\varphi=\varphi$ for some $\varphi\in L^2(\R)$, then $\phi:=A_{n,k}^{-1}(W_\alpha^{1/2} \varphi)\in H^n(\R)$ and $L_{\alpha,k}\phi=0$. The boundedness of $\phi$ then follows from the embedding $H^n(\R)\subset L^\infty(\R)$.
\end{proof}

We next present an oscillation argument in the form needed here.
\begin{lemma}\label{lem:order-two-oscillation}
	Let $T$ be a compact, injective, positive, self-adjoint integral operator on $L^2(\R)$ with a continuous symmetric kernel $K$ that is strictly totally positive of order two. Assume $T$ is trace class. Let
	\begin{equation*}
	\mu_1>\mu_2\ge\mu_3\ge\cdots>0
	\end{equation*}
	be its eigenvalues.  Then $\mu_2$ is simple.  If $\phi_1>0$ and $\phi_2$ are eigenfunctions corresponding to $\mu_1$ and $\mu_2$, respectively, then $\frac{\phi_2}{\phi_1}$ is strictly monotone and $\phi_2$ has exactly one sign change.
\end{lemma}

\begin{proof}
	Let
	\begin{equation*}
	\Omega_2=\{(t_1,t_2)\in\R^2:t_1<t_2\}.
	\end{equation*}
For $(t_1,t_2),(s_1,s_2)\in\Omega_2$, we define
	\begin{equation*}
\mathcal{D}(t_1,t_2;s_1,s_2):=	\det\bigl(K(t_i,s_j)\bigr)_{i,j=1}^2=K(t_1,s_1)K(t_2,s_2)-K(t_1,s_2)K(t_2,s_1).
	\end{equation*}
	By its positivity, we get $\mathcal{D}>0$. We consider the operator $\mathcal{T}$ on $L^2(\Omega_2)$ defined by
	\begin{equation*}
		(\mathcal{T}F)(t_1,t_2):=\int_{\Omega_2}\mathcal{D}(t_1,t_2;s_1,s_2)F(s_1,s_2)\dd s_1\dd s_2.
	\end{equation*}
	Then $\mathcal{T}$ is self-adjoint and compact.
	
	We next describe the eigenfunctions of $\mathcal{T}$. Let $\phi_j$ be a normalized eigenfunction of $T$ corresponding to $\mu_j$. Let
	\begin{equation*}
		\psi_{ij}(t_1,t_2)=\phi_i(t_1)\phi_j(t_2)-\phi_i(t_2)\phi_j(t_1).
	\end{equation*} 
	Then it is straightforward to show that $ \mathcal{T}\psi_{ij}=\mu_i\mu_j\psi_{ij}$. On the other hand, by extending the functions in $\Omega_2$ to $\R^2$ antisymmetrically, we see that $\{\psi_{ij}: i<j\}$ is an orthogonal basis of $L^2(\Omega_2)$. Therefore, the eigenvalues of $\mathcal{T}$ are $\mu_i\mu_j$, $i<j$.
	
	As a result, the principal eigenvalue of $\mathcal{T}$ is $\mu_1\mu_2$ and is simple. Since $\mu_1$ is a simple eigenvalue of $T$, it follows that $\mu_2$ is simple.
	
	After choosing the sign of $\phi_2$, we may assume the principal eigenfunction $\psi_{12}$ is positive. Since $\phi_1>0$, the ratio $\frac{\phi_2}{\phi_1}$ is strictly increasing.  Orthogonality
	\begin{equation*}
	\int_\R\phi_1(t)\phi_2(t)\dd t=0
	\end{equation*}
	forces $\frac{\phi_2}{\phi_1}$ to take both signs.  Strict monotonicity then gives exactly one zero and exactly one sign change of $\phi_2$.
\end{proof}

Because $W_\alpha$ and $G_{n,k}$ are even in $t$, $\mathscr K_{\alpha,k}$ commutes with reflection, so the principal eigenfunction is even. By Lemma \ref{lem:order-two-oscillation}, the second eigenfunction of $\mathscr K_{\alpha,k}$ is simple, so reflection symmetry makes it either even or odd.  A continuous even function cannot have exactly one sign change, whereas an odd function with exactly one sign change must change sign at the origin. Therefore the second eigenfunction must be odd.  Equivalently,
\begin{equation}\label{eq:mu2-odd-sector}
	\mu_2(\alpha,k)
	=r(\mathscr K_{\alpha,k}^{\mathrm o}),
\end{equation}
where $r$ denotes the spectral radius and $\mathscr K_{\alpha,k}^{\mathrm o}$ is the odd half-line operator on $L^2(0,\infty)$ with kernel
\begin{equation}\label{eq:odd-angular-kernel}
	W_\alpha(t)^{\frac{1}{2}}
	\bigl(G_{n,k}(|t-s|)-G_{n,k}(t+s)\bigr)
	W_\alpha(s)^{\frac{1}{2}},
	\qquad t,s>0.
\end{equation}
This kernel is strictly positive because $G_{n,k}$ is strictly decreasing.

 Next, we consider the radial mode restricted to the positive half-line. On the positive half-line, the odd inverse of $B_n$ has kernel
 \begin{equation*}
 g_n^{\mathrm o}(t,s)=g_n(|t-s|)-g_n(t+s)>0.
 \end{equation*}
 The Dirichlet inverse of $-\partial_t^2$ with zero value at $t=0$ and zero derivative at infinity has kernel $\min(t,s)$.  Hence the odd inverse of $A_{n,0}=(-\partial_t^2)B_n$ has the positive kernel
 \begin{equation}\label{eq:radial-odd-green}
 	\mathcal G_{n,0}^{\mathrm o}(t,s)
 	=\int_0^\infty\bigl[g_n(|t-r|)-g_n(t+r)\bigr]\min(r,s)\dd r.
 \end{equation}
 $\mathcal G_{n,0}^{\mathrm o}$ is symmetric and the kernel grows at most linearly in $t+s$. Therefore the function
 \begin{equation*}
 W_\alpha(t)^{\frac{1}{2}}\mathcal G_{n,0}^{\mathrm o}(t,s)W_\alpha(s)^{\frac{1}{2}}
 \end{equation*}
 defines a compact positivity-improving operator
 $\mathscr K_{\alpha,0}^{\mathrm o}$ on $L^2(0,\infty)$.
 
 Since $v_\alpha$ is even and strictly concave, the function $-v_\alpha'(t)$ is positive when $t>0$ and satisfies
 \begin{equation*}
 A_{n,0}(-v_\alpha')=W_\alpha(-v_\alpha').
 \end{equation*}
 Since $v_\alpha'(0)=0$ and $W_\alpha^{\frac{1}{2}}v_\alpha'$ decays exponentially as $t\to\infty$, $W_\alpha^{\frac{1}{2}}v_\alpha'$ belongs to the domain of the half-line operator $\mathscr{K}_{\alpha,0}^{\mathrm{o}}$ and we have
 \begin{equation*}
 \mathscr K_{\alpha,0}^{\mathrm o}\bigl(-W_\alpha^{\frac{1}{2}}v_\alpha'\bigr)
 =-W_\alpha^{\frac{1}{2}}v_\alpha'.
 \end{equation*}
 This eigenfunction is strictly positive, so we conclude that
 \begin{equation}\label{odd spec rad 1}
 	r(\mathscr K_{\alpha,0}^{\mathrm o})=1.
 \end{equation}

The next proposition shows that we only need to consider the principal eigenvalues.

\begin{proposition}\label{thm:mu2}
If $\mu_1(\alpha,k)>\mu_2(\alpha,k)\ge\cdots$ are the eigenvalues of $\mathscr K_{\alpha,k}$, then
\begin{equation*}
 \mu_2(\alpha,k)<1 \quad\hbox{for every }n\ge3,\ \alpha>0,\ k\ge1.
\end{equation*}
Consequently $L_{\alpha,k}$ has Morse index at most one. If zero is an eigenvalue of $L_{\alpha,k}$, then it is the principal eigenvalue and is simple. 
Moreover, let $\psi_{\alpha,k}>0$ be the even principal generalized eigenfunction, that is,
\begin{equation}\label{eq:generalized-eigenpair}
A_{n,k}\psi_{\alpha,k}=\tau_{\alpha,k}W_\alpha\psi_{\alpha,k}, \qquad\tau_{\alpha,k}=\mu_1(\alpha,k)^{-1}.
\end{equation}
Then $\partial_t\psi_{\alpha,k}(t)<0$ for $t>0$.
\end{proposition}

\begin{proof}
By \eqref{eq:mu2-odd-sector}, $\mu_2(\alpha,k)$ is the principal eigenvalue in the class of odd functions.  Let $f\in L^2(\R)$ be the corresponding odd eigenfunction, normalized by $\|f\|_2=1$ and chosen positive on $(0,\infty)$. Define the odd function $h=W_\alpha^{\frac{1}{2}}f$. Then, since $W_\alpha$ decays exponentially, we have $h\in L^1\cap L^2$, $\widehat h(0)=0$ and $\widehat h(\xi)=O(\xi)$ as $\xi\to0$. Since $\widehat h(\xi)=O(\xi)$ and $a_{n,0}(\xi)\sim\beta_n\xi^2$ as $\xi\to0$, the quotient $|\widehat h(\xi)|^2/a_{n,0}(\xi)$ is bounded near zero. Away from zero, $a_{n,0}^{-1}$ is bounded and $\widehat h\in L^2(\R)$ by Plancherel's theorem. Thus the radial inverse quadratic form is finite:
\[
 \frac1{2\pi}\int_\R
 \frac{|\widehat h(\xi)|^2}{a_{n,0}(\xi)}\dd\xi<\infty.
\]

We claim that for every $n\ge3$, every integer $k\ge1$ and every $\xi\in\R$,
\begin{equation}\label{eq:symbol-order}
	a_{n,k}(\xi)>a_{n,0}(\xi).
\end{equation}
Indeed, by definition \eqref{eq:mode-symbol}, we have the recurrence relations
\begin{align}
	\frac{a_{n,k+2}(\xi)}{a_{n,k}(\xi)}
	&=\frac{\xi^2+(n+k)^2}{\xi^2+k^2}>1,
	\label{eq:k-recurrence}\\
	a_{n+2,k}(\xi)
	&=\bigl(\xi^2+(n+k)^2\bigr)a_{n,k}(\xi).
	\label{eq:n-recurrence}
\end{align}
For even dimensions, we begin with $a_{2,k}(\xi)=\xi^2+k^2>a_{2,0}(\xi)=\xi^2$. Then the inequality follows from \eqref{eq:n-recurrence}.

For odd dimensions, start with $n=3$.  With $s=\frac{|\xi|}{2}$, the reflection identities for the gamma function give
\begin{equation*}
a_{3,0}(\xi)=8s\left(s^2+\frac14\right)\tanh(\pi s),
\qquad a_{3,1}(\xi)=8s(s^2+1)\coth(\pi s).
\end{equation*}
Thus $a_{3,1}>a_{3,0}$.  Formula \eqref{eq:k-recurrence} gives $a_{3,k}>a_{3,0}$ for every even $k\ge2$ and it gives $a_{3,k}>a_{3,1}>a_{3,0}$ for every odd $k\ge3$. Finally, \eqref{eq:n-recurrence} propagates the inequality to all odd dimensions.

Set $f_+(t)=\sqrt2\,f(t)$ for $t>0$.
Since $f$ is odd and $\|f\|_{L^2(\R)}=1$, we have
$\|f_+\|_{L^2(0,\infty)}=1$.
By Plancherel's identity, \eqref{odd spec rad 1}, and \eqref{eq:symbol-order},
\begin{align*}
 \mu_2(\alpha,k)
 &=\frac1{2\pi}\int_\R
   \frac{|\widehat h(\xi)|^2}{a_{n,k}(\xi)}\dd\xi\\
 &<\frac1{2\pi}\int_\R
   \frac{|\widehat h(\xi)|^2}{a_{n,0}(\xi)}\dd\xi\\
 &=\left\langle
   f_+,\mathscr K_{\alpha,0}^{\mathrm o}f_+
   \right\rangle_{L^2(0,\infty)}\\
 &\le r(\mathscr K_{\alpha,0}^{\mathrm o})
       \|f_+\|_{L^2(0,\infty)}^2=1.
\end{align*}
By the Birman--Schwinger principle,
\begin{equation*}
     \operatorname{ind}^{-}(L_{\alpha,k})=\#\{j:\mu_j(\alpha,k)>1\},
\end{equation*}
where $\operatorname{ind}^{-}$ denotes the number of negative eigenvalues and both sides count multiplicities. Since $\mu_2(\alpha,k)<1$, this number is at most one. If zero is an eigenvalue of $L_{\alpha,k}$, then \eqref{equiv L K} gives $\mu_1(\alpha,k)=1$. Thus $L_{\alpha,k}$ has no negative eigenvalues, and zero is its principal eigenvalue and is simple.

It remains to prove the monotonicity of $\psi_{\alpha,k}$. Since $\psi_{\alpha,k}\in H^n(\R)$ and $W_\alpha,W_\alpha'\in L^\infty(\R)$, we can differentiate the equation \eqref{eq:generalized-eigenpair} with respect to $t$. Then, the odd function
\begin{equation*}
 z_{\alpha,k}:=-W_\alpha^{1/2}\partial_t\psi_{\alpha,k}\in L^2(\R)   
\end{equation*}
 satisfies
\begin{align}
\bigl(\mu_1(\alpha,k)I-\mathscr K_{\alpha,k}^{\mathrm o}\bigr)z_{\alpha,k}&=-W_\alpha^{1/2}A_{n,k}^{-1}
\bigl(W_\alpha'\psi_{\alpha,k}\bigr)\nonumber\\
&=- W_\alpha(t)^{1/2}\int_0^\infty
 \bigl[G_{n,k}(|t-s|)-G_{n,k}(t+s)\bigr]W_\alpha'(s)\psi_{\alpha,k}(s)\dd s\nonumber\\
 &=:F_{\alpha,k} \text{ on } (0,\infty).\nonumber
\end{align}
Here in the second equation, we use oddness and the half-line kernel \eqref{eq:odd-angular-kernel}. Since $ W_\alpha'(t)=n\,\partial_t v_\alpha(t)\,W_\alpha(t)<0$ for $t>0$ and $G_{n,k}$ is strictly decreasing, $F_{\alpha,k}(t)>0$ for every $t>0$.

By \eqref{eq:mu2-odd-sector} and the positivity and self-adjointness of the operator,
\[
 \norm{\mathscr K_{\alpha,k}^{\mathrm o}}_{L^2\to L^2}
 =\mu_2(\alpha,k)<\mu_1(\alpha,k).
\]
Consequently, we can invert the operator $(\mu_1(\alpha,k)I-\mathscr K_{\alpha,k}^{\mathrm o}\bigr)$ to obtain
\begin{equation*}
 z_{\alpha,k}=\frac{1}{\mu_1(\alpha,k)}\sum_{j=0}^\infty
 \left(\frac{\mathscr K_{\alpha,k}^{\mathrm o}}{\mu_1(\alpha,k)}\right)^j F_{\alpha,k},   
\end{equation*}
with convergence in $L^2(0,\infty)$. The kernel \eqref{eq:odd-angular-kernel} is positive, so $z_{\alpha,k}\ge \frac{F_{\alpha,k}}{\mu_1(\alpha,k)}>0$ almost everywhere. Since $z_{\alpha,k}$ and $F_{\alpha,k}$ are continuous, this inequality holds everywhere on $(0,\infty)$. Therefore, we conclude that
\begin{equation*}
    \partial_t\psi_{\alpha,k}(t)=-\frac{z_{\alpha,k}(t)}{W_\alpha(t)^{1/2}}<0
 \text{  for }t>0.
\end{equation*}
\end{proof}

%Thus, $L_{\alpha,k}$ has a zero eigenvalue if and only if $\tau_{\alpha,k}=1$.

For a real angular parameter $\kappa>0$, using \eqref{eq:G-kappa}, we define $\mathscr K_{\alpha,\kappa}$ by replacing $k$ with $\kappa$ in its definition \eqref{eq:BS operator}. Then the principal eigenvalue $\mu_1(\alpha,\kappa)$ is strictly monotone in both $\alpha$ and $\kappa$.

\begin{lemma}\label{lem:monotone mu}
	\begin{equation}\label{eq:mu-kappa-decreasing}
		\kappa\longmapsto\mu_1(\alpha,\kappa) \quad\text{is strictly decreasing}.
	\end{equation}
	$\mu_1(\alpha,\kappa)$ is differentiable with respect to $\alpha$ and
	\begin{equation}\label{eq:mu-al-increasing}
		\partial_\alpha\mu_1(\alpha,\kappa)>0.
	\end{equation}
\end{lemma}
\begin{proof}
	For $\kappa_2>\kappa_1$, let $f_2>0$ be the normalized principal eigenfunction of $\mathscr K_{\alpha,\kappa_2}$. Since $G_{n,\kappa}(t)$ is strictly decreasing in $\kappa$, for each fixed $t$, we see that
	\begin{equation*}
	\mu_1(\alpha,\kappa_2)
	=\ip{f_2}{\mathscr K_{\alpha,\kappa_2}f_2}
	<\ip{f_2}{\mathscr K_{\alpha,\kappa_1}f_2}
	\le\mu_1(\alpha,\kappa_1).
	\end{equation*}
	The monotonicity in $\alpha$ requires smooth dependence of the radial solution on the parameter. We establish this dependence in Corollary \ref{coro v al smooth} and prove \eqref{eq:mu-al-increasing} in Subsection \ref{sec:transversality}.
\end{proof}

Equation \eqref{eq:G-kappa} gives, locally uniformly in $t$,
\begin{equation*}
    G_{n,\kappa}(t)=\frac{2^{1-n}}{\Gamma(n/2)^2\kappa}+O(1),
 \qquad \kappa\downarrow0.
\end{equation*}
Moreover, since $0\le 1-2x\cosh|t|+x^2\le 1$ for $0<\kappa\le1$, we have
\begin{equation*}
   0<\kappa G_{n,\kappa}(t)\le\kappa G_{n,\kappa}(0)\le \frac{2^{1-n}}{\Gamma(n/2)^2},\qquad 0<\kappa\le1, 
\end{equation*}
where the last bound follows from \eqref{eq:G-diagonal}. Since $W_\alpha\in L^1(\R)$, the integral kernel of $\kappa\mathscr K_{\alpha,\kappa}$ converges pointwise to
\begin{equation*}
     \frac{2^{1-n}}{\Gamma(n/2)^2} W_\alpha(t)^{1/2}W_\alpha(s)^{1/2}
\end{equation*}
and is dominated in $L^2(\R^2)$. By dominated convergence, $\kappa\mathscr K_{\alpha,\kappa}$ converges to the rank-one operator
\begin{equation*}
    f(t)\mapsto \frac{2^{1-n}}{\Gamma(n/2)^2} W_\alpha(t)^{1/2}\int_{\R}W_\alpha(s)^{1/2}f(s)\dd s
\end{equation*}
in Hilbert--Schmidt norm and operator norm. By \eqref{eq:W-mass}, the principal eigenvalue of this operator is
\begin{equation*}
 \frac{2^{1-n}}{\Gamma(n/2)^2}\int_\R W_\alpha(t)\dd t=n\alpha.   
\end{equation*}
 Consequently,
\begin{equation*}
 \kappa\mu_1(\alpha,\kappa)\to\frac{2^{1-n}}{\Gamma(n/2)^2}\int_\R W_\alpha(t)\dd t ={n\alpha},
 \qquad \kappa\downarrow0.
\end{equation*}

On the other hand, the trace formula \eqref{eq:exact-trace}, extended to real $\kappa$, gives
\begin{equation*}
 \mu_1(\alpha,\kappa)\le\tr\mathscr K_{\alpha,\kappa}\longrightarrow0,
 \qquad \kappa\to\infty.   
\end{equation*}
The same dominated-convergence argument on compact $\kappa$-intervals gives operator-norm continuity in $\kappa$, and hence continuity of $\mu_1(\alpha,\kappa)$. Together with \eqref{eq:mu-kappa-decreasing}, these limits show that, for each fixed $\alpha>0$, there is a unique generalized angular threshold $\kappa_*(\alpha)>0$ such that
\begin{equation*}
 \mu_1(\alpha,\kappa_*(\alpha))=1.   
\end{equation*}

\subsection{Location of the bifurcation parameters}\label{sec:crossings}
In this subsection, we prove the estimate \eqref{eq:main-range-intro}. The lower bound follows directly from the trace formula. The upper bound requires a comparison between the normalized angular Green kernel and a function obtained from the radial kernel. We establish this comparison and combine it with an identity for the radial solution.

At a bifurcation parameter $\alpha=\alpha_{n,k}$, the principal eigenvalue $\mu_1(\alpha_{n,k},k)=1$, while all remaining eigenvalues are positive. By the trace formula \eqref{eq:exact-trace},
\begin{equation}\label{al nk lower}
	\alpha_{n,k}-\cA_{n,k}=\cA_{n,k}\sum_{j\ge2}\mu_j>0.
\end{equation}
This proves the lower bound for $\alpha_{n,k}$.

\medskip

In the rest of this subsection, we study the upper bound for $\alpha_{n,k}$. For $t\ge0$, we define
\begin{equation}\label{eq:H-definition}
	H_n(t):=2\beta_n\mathcal K_n'(t)=2\beta_n\int_0^t g_n(s)\dd s.
\end{equation}
Then $H_n(0)=0$, $H_n(t)>0$ and by \eqref{int gn}, $H_n(t)\to1$ as $t\to+\infty$. 
\begin{lemma}\label{lem:H-representations}
	Let
\begin{equation*}
a=\frac{n-2}{2}, \qquad q=e^{-2t},\qquad x=\sech^2t=\frac{4q}{(1+q)^2}.
\end{equation*}
	For $t>0$, we have the representations
	\begin{align}
		H_n(t)&=2a\int_0^1(1-s)^{a-1}(1-qs)^a\dd s-1,\label{eq:H-beta}\\
		&=\tanh t\,F_a(x),		\label{eq:H-hyper}
	\end{align}
	where $F_a(x)=\Hyp\left(\frac{1}{2},1;a+1;x\right)$. Moreover,
	\begin{equation}\label{eq:F-first-order}
		x(1-x)F_a'(x)+\left(a-\frac x2\right)F_a(x)=a,
	\end{equation}
	and $F_a(x)\ge1$ for $0\le x<1$. Finally, we have the identity
	\begin{equation}\label{eq:virial-identity}
		\iint_{\R^2}|t-s|H_n(|t-s|)\rho_\alpha(t)\rho_\alpha(s)\dd t \dd s=\frac{2}{n\alpha}.	
	\end{equation}
\end{lemma}

\begin{proof}
	For $t>0$, the Beta-integral calculation in Proposition~\ref{prop:mode-green} remains valid for every positive dimension parameter, since the integrand is integrable near both $s=0$ and $s=1$. We may therefore apply it with dimension parameter $n-2$ and degree $2$, using $b_n(\xi)=a_{n-2,2}(\xi)$.
By \eqref{eq:H-definition}, integrating $g_n$ from zero to $t$ gives \eqref{eq:H-beta}. Euler's integral representation \eqref{eq:hyp-euler} and the quadratic transformation \eqref{eq:quadratic-transformation} give \eqref{eq:H-hyper}.

	Since
	\begin{equation*}
	F_a(x)=\sum_{j=0}^\infty\frac{(\frac{1}{2})_j}{(a+1)_j}x^j,
	\end{equation*}
	all coefficients are nonnegative and $F_a\ge1$. Substitution of this series into \eqref{eq:F-first-order} proves the identity term by term.

To prove \eqref{eq:virial-identity}, we evaluate the energy \eqref{eq:free-energy} along the mass-preserving dilation $\rho_\lambda(t)=\lambda\rho_\alpha(\lambda t)$.
	
	A change of variables gives
	\begin{equation*}
\int_\R\rho_\lambda\log\rho_\lambda\dd t
=\log\lambda+\int_\R\rho_\alpha\log\rho_\alpha\dd t,
	\end{equation*}
	and the interaction term in \eqref{eq:free-energy} becomes
	\begin{equation*}
	\frac{C_{n,\alpha}}{2}\iint_{\R^2}\mathcal K_n\!\left(\frac{t-s}{\lambda}\right)\rho_\alpha(t)\rho_\alpha(s)\dd t\dd s.
	\end{equation*}
	
	Since $\rho_\alpha$ is a critical point, we have
	\begin{equation*}
	0=\frac{d\mathcal{F}_{n,\alpha}(\rho_\lambda)}{d\lambda}\Big|_{\lambda=1}
    =1-\frac{C_{n,\alpha}}{2}\iint_{\R^2}|t-s|\mathcal K_n'(|t-s|)\rho_\alpha(t)\rho_\alpha(s)\dd t\dd s.
	\end{equation*}
	Then \eqref{eq:virial-identity} follows from the identities $H_n=2\beta_n\mathcal K_n'$ and $C_{n,\alpha}=2n\alpha\beta_n$.
\end{proof}

We introduce the normalized kernel
\begin{equation*}
R_{n,k}(t):=\frac{G_{n,k}(t)}{G_{n,k}(0)}.
\end{equation*}
Combining \eqref{eq:G-hyper} with the quadratic transformation gives
\begin{equation*}
	R_{n,k}(t)=	\frac{x^{k/2}\Hyp\left(\frac k2,\frac{k+1}{2};k+\frac n2;x\right)}{\Hyp\left(\frac k2,\frac{k+1}{2};k+\frac n2;1\right)},
\end{equation*}
where $x=\sech^2t$. By definition, $R_{n,k}(0)=1$ and $R_{n,k}(t)\to0$ as $t\to\infty$. 

 We observe that $R_{n,k}$ satisfies the following ODE
\begin{equation}\label{eq:R-ODE}
	\mathcal D_nR_{n,k}=\lambda_{n,k}R_{n,k},
\end{equation}
where $\mathcal D_n:=\partial_t^2-(n-2)\coth t\,\partial_t$ and $ \lambda_{n,k}:=k(k+n-2)$. Indeed, by \eqref{eq:hyp-ode}, the function $y(x)=x^{k/2}\Hyp\left(\frac k2,\frac{k+1}{2};k+\frac n2;x\right)$ satisfies the hypergeometric equation
\begin{equation*}
4x^2(1-x)y''+2x(n-3x)y'-\lambda_{n,k}y=0.
\end{equation*}
Since $x'=-2x\tanh t$ and $x''=4x-6x^2$, this is exactly \eqref{eq:R-ODE}.

The next comparison lemma plays a central role in deriving the upper bound for $\alpha_{n,k}$.
\begin{lemma}\label{lem:kernel-comparison}
	For every $n\ge3$, $k\ge2$ and $t>0$, we have
	\begin{equation}\label{eq:kernel-comparison}
		1-R_{n,k}(t) <\frac{\lambda_{n,k}}{2(n-2)}\tanh t\,H_n(t).
	\end{equation}
\end{lemma}

\begin{proof}
	Let $a=\frac{n-2}{2}$, $x=\sech^2t$ and $F_a$ be as in
	Lemma~\ref{lem:H-representations}.  Define
	\begin{equation*}
		\Phi_n(t):=\frac{\tanh t\,H_n(t)}{2(n-2)}
		=\frac{(1-x)F_a(x)}{2(n-2)}.
	\end{equation*}
	In the $x$ variable,
	\begin{equation*}
	\mathcal D_n=4x^2(1-x)\partial_x^2+2x(n-3x)\partial_x.
	\end{equation*}
	Using \eqref{eq:F-first-order} and the fact that $F_a\ge1$, we obtain
	\begin{equation*}
		(\mathcal D_n-2n)\Phi_n+1
		=(1-x)\left[1-\frac{n+x}{n-2}F_a(x)\right]<0.
	\end{equation*}
	Since $k\ge2$, $\lambda_{n,k}-2n=(k-2)(k+n)\ge0$, so
	\begin{equation*}
		(\mathcal D_n-\lambda_{n,k})\Phi_n<-1.
	\end{equation*}
	
	Set $f=\frac{1-R_{n,k}}{\lambda_{n,k}}$. Then \eqref{eq:R-ODE} implies
	\begin{equation*}
	(\mathcal D_n-\lambda_{n,k})f=-1.
	\end{equation*}
	
	At $t=0$, $f(0)=\Phi_n(0)=0$; while at $t=\infty$, we have
	\begin{equation*}
	\lim_{t\to\infty}f(t)=\frac1{\lambda_{n,k}}, \qquad	\lim_{t\to\infty}\Phi_n(t)=\frac1{2(n-2)}.
	\end{equation*}
	Therefore, the function $w:=\Phi_n-f$ satisfies
	\begin{equation*}
		w(0)=0,\quad \lim_{t\to\infty}w(t)>0, \quad (\mathcal D_n-\lambda_{n,k})w<0.
	\end{equation*}
	A nonpositive interior minimum would satisfy	$w'=0$, $w''\ge0$ and hence $(\mathcal D_n-\lambda_{n,k})w=w''-\lambda_{n,k}w\ge0$, a contradiction. Thus $f<\Phi_n$, which is
	\eqref{eq:kernel-comparison}.
\end{proof}

	On $L^2(\R,\rho_\alpha\dd t)$, we define the operator
	\begin{equation*}
	(S_{\alpha,k}f)(t)=\int_\R R_{n,k}(t-s)f(s)\rho_\alpha(s)\dd s.
	\end{equation*}
	The map
	\begin{equation*}
		U:L^2(\R,\rho_\alpha\dd t)\longrightarrow L^2(\R,\dd t), \qquad Uf=\rho_\alpha^{\frac{1}{2}}f,
	\end{equation*}
	is unitary.  Since $G_{n,k}=G_{n,k}(0)R_{n,k}$ and $C_{n,\alpha}G_{n,k}(0)=\frac{\alpha}{\cA_{n,k}}$, by \eqref{eq:BS operator}, we have
\begin{equation}\label{eq:unitary-equivalence}
\mathscr K_{\alpha,k}=\frac{\alpha}{\cA_{n,k}}US_{\alpha,k}U^{-1}.
	\end{equation}
	
	Let $\Lambda_1(\alpha,k)$ be the principal eigenvalue of $S_{\alpha,k}$.  The constant function $1$ has norm one in $L^2(\R,\rho_\alpha\dd t)$ because $\int_\R\rho_\alpha(t)\dd t=1$, so
	\begin{equation}\label{eq:constant-trial}
		\Lambda_1(\alpha,k)	\ge\iint_{\R^2}R_{n,k}(t-s)\rho_\alpha(t)\rho_\alpha(s)\dd t\dd s.
	\end{equation}
	
	On the other hand, since $\tanh r<r$ for $r>0$, by \eqref{eq:kernel-comparison} and \eqref{eq:virial-identity}, we have
	\begin{align}
		\iint_{\R^2}
		\bigl[1-R_{n,k}(t-s)\bigr]
		\rho_\alpha(t)\rho_\alpha(s)\dd t \dd s
		<&\frac{\lambda_{n,k}}{2(n-2)}
		\iint_{\R^2}
		|t-s|H_n(|t-s|)
		\rho_\alpha(t)\rho_\alpha(s)\dd t\dd s 
		\notag\\
		=&\frac{\lambda_{n,k}}{n(n-2)\alpha}.
		\label{eq:integral-kernel-bound}
	\end{align}
	Combining \eqref{eq:constant-trial} and \eqref{eq:integral-kernel-bound} yields $$\Lambda_1(\alpha,k)>1-\frac{\lambda_{n,k}}{n(n-2)\alpha}.$$
    By \eqref{eq:unitary-equivalence},
	\begin{equation}\label{eq:principal-lower-estimate}
		\mu_1(\alpha,k)=\frac{\alpha}{\cA_{n,k}}\Lambda_1(\alpha,k)
		>\frac{\alpha-\cB_{n,k}}{\cA_{n,k}},
	\end{equation}
	where $\cB_{n,k}=\frac{k(n+k-2)}{n(n-2)}$. This proves the upper bound
	\begin{equation}\label{al nk upper}
	\alpha_{n,k}<\cA_{n,k}+\cB_{n,k}.
	\end{equation} 

	\subsection{Proof of Theorem \ref{thm e value L}}
	\begin{proof}[Proof of Theorem \ref{thm e value L}]
    We now combine the preceding estimates with the monotonicity in $\alpha$ proved in Subsection \ref{sec:transversality}. By Proposition \ref{thm:mu2}, we only need to consider the principal eigenvalue $\mu_1(\alpha,k)$.
		
		In the regular case $\alpha=1$, the solution is explicitly given by
		\begin{equation}\label{v1 explicit}
		v_1(t)=-\log\cosh t.
		\end{equation}
		In this case $W_1(t)=n!\sech^n t$. Using \eqref{eq:mode-symbol} and \eqref{eq:sech-transform}, we have
		\begin{equation*}
			A_{n,k}(\sech^k t)=\frac{\Gamma(n+k)}{\Gamma(k)}\sech^{n+k} t.
		\end{equation*}
		Hence $W_1^{1/2}\sech^k t$ is a positive principal eigenfunction of $\mathscr K_{1,k}$, with eigenvalue
		\begin{equation*}
			\mu_1(1,k)=\frac{\Gamma(n+1)\Gamma(k)}{\Gamma(n+k)}.
		\end{equation*}
		Thus $\mu_1(1,1)=1$, while $\mu_1(1,k)<1$ for every $k\ge2$. By \eqref{eq:principal-lower-estimate}, $\mu_1(\cA_{n,k}+\cB_{n,k},k)>1$.  Continuity therefore yields a crossing and \eqref{eq:mu-al-increasing} makes it unique. The range \eqref{eq:main-range-intro} is given by \eqref{al nk lower} and \eqref{al nk upper}.
		
		By \eqref{eq:mu-kappa-decreasing}, different modes $k$ yield distinct crossing values $\alpha$. This completes the proof.
\end{proof}

We close this section by proving the nondegeneracy result Corollary \ref{coro:nondegen}.
\begin{proof}[Proof of Corollary \ref{coro:nondegen}]
    Let the spherical decomposition of $h$ be
    \begin{equation*}
        h(t,\omega)=\sum_{k=0}^\infty\sum_\ell h_{k,\ell}(t)Y_{k,\ell}(\omega).
    \end{equation*}
    Then each $h_{k,\ell}$ is bounded and satisfies $L_{\alpha,k}h_{k,\ell}=0$. By Theorem \ref{thm radial}, $h_0=c_0v_\alpha'$ for some $c_0$. For $k\ge 1$, Lemma \ref{lem bdd=spec} implies $h_{k,\ell}\in H^n(\R)$, so the spectral results in this section apply.

    Since $0<\alpha<\alpha_{n,2}$, from Theorem \ref{thm e value L} we conclude that $h_{k,\ell}=0$ for all $k\ge 2$. For $k=1$, we have $\mu_2(\alpha,1)< 1=\mu_1(1,1)$, and the strict monotonicity in $\alpha$ shows that a kernel occurs only at $\alpha=1$. By \eqref{v1 explicit}, $v_1'(t)=-\tanh{t}$. Since the space of degree-$1$ spherical harmonics $\mathcal{H}_1$ is spanned by coordinate functions $\omega_1,\dots, \omega_n$, we obtain \eqref{h kernel}. The Euclidean statement then follows by a change of variable $t=\log|x|$.
\end{proof}

\section{Nonradial solutions: a local bifurcation}\label{sec:nonradial}

In this section, we construct nonradial normal solutions by applying the Crandall--Rabinowitz theorem \cite{CR} in weighted Sobolev spaces. We verify the Fredholm property, reduce the kernel to one dimension by imposing a suitable angular symmetry, and establish transversality. We begin by recalling the version of the theorem used below.

\begin{othertheorem}[Crandall--Rabinowitz]\label{thm:quoted-CR}
	Let $X,Y$ be real Banach spaces, let $U\subset\R\times X$ be a neighborhood of
	$(\lambda_0,0)$ and let $F\in C^r(U,Y)$ with $r\ge2$ and $F(\lambda,0)=0$.  Put
	$L=D_xF(\lambda_0,0)$.  Assume:
	\begin{enumerate}[label=\textup{(\roman*)}]
		\item $L$ is Fredholm of index zero;
		\item $\Ker L=\Span\{\phi_0\}$ and $\operatorname{codim}\Ran L=1$;
		\item $D_{\lambda x}F(\lambda_0,0)\phi_0\notin\Ran L$.
	\end{enumerate}
	Then there are $\varepsilon>0$ and $C^{r-1}$ functions	$s\mapsto\lambda(s)$ and $s\mapsto z(s)$, with	$z(s)$ taking values in a fixed complement of $\Span\{\phi_0\}$, such that $\lambda(0)=\lambda_0$, $z(0)=0$, and every nontrivial solution near $(\lambda_0,0)$ is of the form
	\begin{equation*}
		(\lambda(s),\,s\phi_0+s z(s)),
		\qquad |s|<\varepsilon.
	\end{equation*}
\end{othertheorem}

\subsection{Weighted Fredholm theory}\label{subsec:fredholm}
Weighted Fredholm theory on cylindrical ends is classical for differential and pseudodifferential operators; see \cite{KohrNistorWendland,LockhartMcOwen1985}. In this subsection, we give a direct proof for the present critical symbols.

Let $G\subset O(n)$ be a closed subgroup.  Recall that $\mathcal{H}_k$ is the space of spherical harmonics of degree $k$ and $\{Y_{k,\ell}\}_{\ell}$ is an orthonormal basis of $\mathcal{H}_k$. Define
\begin{equation}\label{eq:kG}
	k_G:=\min\{k\ge1:\mathcal{H}_k^G\ne\{0\}\},
\end{equation}
with the convention $k_G=\infty$ if no positive degree occurs. Here, $$\mathcal H_k^G
 =\left\{Y\in\mathcal H_k:
 Y(g\omega)=Y(\omega)
 \text{ for all }g\in G,\ \omega\in\Sph^{n-1}\right\}.$$  A superscript $G$ means $G$-invariance in $\omega$. 

Choose
\begin{equation}\label{eq:weight-condition}
	0<\nu<\min\{2,k_G,n\alpha\}.
\end{equation}

For a function $v$, we write its spherical decomposition as follows:
\begin{equation}\label{v spher decom}
	v(t,\omega)=\sum_{k=0}^\infty\sum_\ell v_{k,\ell}(t)Y_{k,\ell}(\omega).
\end{equation}
For $r\in\R$, the norm of $H^r(\Cyl)$ can be rewritten as
\begin{equation*}
\norm{v}_{H^r}^2=\sum_{k=0}^\infty\sum_\ell\int_\R(1+\xi^2+\lambda_k)^r\abs{\widehat{v_{k,\ell}}(\xi)}^2\dd\xi.
\end{equation*}
  Set
\begin{equation}\label{eq:weighted-sobolev}
	H_\nu^r(\Cyl):=\left\{u:e^{\nu t}u\in H^r(\Cyl),\ e^{-\nu t}u\in H^r(\Cyl)\right\},
\end{equation}
with norm
\begin{equation}\label{eq:weighted-norm}
	\norm{u}_{H_\nu^r}^2:=\norm{e^{\nu t}u}_{H^r}^2+\norm{e^{-\nu t}u}_{H^r}^2.
\end{equation}

We choose the domain and range of $L_\alpha$ to be
\begin{equation}\label{eq:X-Y}
	\cX_\nu^{G,e}:=\R\oplus H_\nu^{n}(\Cyl)^{G,e},
	\qquad \cY_\nu^{G,e}:=L^2_\nu(\Cyl)^{G,e},
\end{equation}
Here $H_\nu^k(\Cyl)$ is the weighted Sobolev space defined in
\eqref{eq:weighted-sobolev}. For each integer $k\geq0$, we write
\begin{equation*}
 H_\nu^k(\Cyl)^{G,e}
 :=\left\{
 u\in H_\nu^k(\Cyl):
 \begin{array}{l}
 u(t,g\omega)=u(t,\omega)\quad\text{for every }g\in G,\\
 u(-t,\omega)=u(t,\omega)
 \end{array}
 \right\}    
\end{equation*}
and set
\begin{equation*}
    L_\nu^2(\Cyl)^{G,e}:=H_\nu^0(\Cyl)^{G,e}.
\end{equation*}
Thus the superscripts $G$ and $e$ indicate $G$-invariance in $\omega$ and evenness in $t$, respectively.   A function in $\cX_\nu^{G,e}$ has a unique decomposition $c+\phi$ with $\phi\in H_\nu^{n}(\Cyl)^{G,e}$, so we can define the norm
\begin{equation*}
	\norm{c+\phi}_{\cX_\nu^{G,e}}^2
	:=|c|^2+\norm{\phi}_{H_\nu^{n}}^2.
\end{equation*}

The spherical averaging projection is
\begin{equation*}
	(\Pi_0v)(t):=\frac{1}{|\Sph^{n-1}|}\int_{\Sph^{n-1}}v(t,\omega)\dd\omega,
\end{equation*}
and we write $\Pi_\perp:=I-\Pi_0$. The main Fredholm property is the following.
\begin{proposition}\label{prop:main-fredholm}
Assume $n\ge3$, $\alpha>0$ and \eqref{eq:weight-condition} holds.	Then
\begin{equation}\label{eq:linearized-map}
		{L}_\alpha=P_n-W_\alpha:	\cX_\nu^{G,e}\longrightarrow\cY_\nu^{G,e}
	\end{equation}
	is a bounded Fredholm operator of index zero. 
	
	For the operator $P_n$, we have 
	\begin{equation}\label{eq:free-kernel}
		\Ker\bigl(P_n:\cX_\nu^{G,e}\to\cY_\nu^{G,e}\bigr)=\Span\{1\},
	\end{equation}
	and
	\begin{equation}\label{eq:free-range}
		\Ran\bigl(P_n:\cX_\nu^{G,e}\to\cY_\nu^{G,e}\bigr)=\left\{f\in\cY_\nu^{G,e}: \int_\R	\int_{\Sph^{n-1}}f(t,\omega)\dd\omega\dd t=0\right\}.
	\end{equation}
	On the nonradial $G$-invariant subspace, $P_n: \Pi_\perp H_\nu^n(\Cyl)^{G,e}\to \Pi_\perp L_\nu^2(\Cyl)^{G,e}$ is an isomorphism.
\end{proposition}

Before going to the proof, we need some facts about the Fourier symbol $a_{n,k}(\xi)$. By \eqref{eq:mode-symbol}, we see that for $k\ge1$, $a_{n,k}(\xi)$ is holomorphic and nonzero in the strip $|\operatorname{Im}\xi|<k$.
For odd $n$, the zeros of $a_{n,k}$ are
\begin{equation*}
 \xi=\pm i(k+2j),\qquad j=0,1,2,\ldots.   
\end{equation*}
For even $n=2m$, \eqref{a 2m symbol} shows that the zeros of $a_{2m,k}$ are $\xi=\pm i(k+2j)$ with $j=0,\ldots,m-1$.

For the radial symbol, the only zero in $|\Im\zeta|<2$ is the double zero at $\zeta=0$. Consequently,
\begin{equation*}
 b_n(\zeta)=\frac{a_{n,0}(\zeta)}{\zeta^2} \quad\text{is analytic and nonzero in }|\Im\zeta|<2,
\end{equation*}
where the value at $\zeta=0$ is defined by continuity. This explains the first two restrictions on $\nu$ in \eqref{eq:weight-condition}.

By Stirling's formula, if $\operatorname{Re}(z)>c_0>0$, then we have
\begin{equation*}
	\abs{\frac{\Gamma(z+n/2)}{\Gamma(z)}}\sim (1+|z|)^{n/2}.
\end{equation*}
 Therefore, by \eqref{eq:mode-symbol}, there is a constant $C\ge1$, depending only on $n,k_G$ and $\nu$, such that for every $k\ge k_G$, $\xi\in\R$ and $|\sigma|\le\nu$,
\begin{equation}\label{eq:shifted-symbol-estimate}
	C^{-1}(1+k^2+\xi^2)^{n/2}
	\le \abs{a_{n,k}(\xi+i\sigma)}
	\le C(1+k^2+\xi^2)^{n/2}.
\end{equation}
Moreover, if $0<\nu<2$, then
\begin{equation}\label{eq:shifted-b-estimate}
	C^{-1}(1+\xi^2)^{(n-2)/2}
	\le \abs{b_n(\xi+i\sigma)}
	\le C(1+\xi^2)^{(n-2)/2}
\end{equation}
for $\xi\in\R$ and $|\sigma|\le\nu$.

\begin{proof}[Proof of Proposition \ref{prop:main-fredholm}]
We first determine the kernel and range of $P_n$ by treating radial and nonradial functions separately. We then obtain the Fredholm property and index of $L_\alpha$ by compact perturbation.

We now prove the weighted estimate for the nonradial component.  Let $v$ be smooth and compactly supported in $t$, and suppose that $v_{k,\ell}\neq 0$ for only finitely many pairs $(k,\ell)$. Here,  $v_{k,\ell}$ is the coefficient in (\ref{v spher decom}). For $\sigma=\pm\nu$, exponential conjugation and the Fourier convention give
	\begin{equation}\label{eq:conjugation}
		\widehat{e^{\sigma t}(P_nv)_{k,\ell}}(\xi)
		=a_{n,k}(\xi+i\sigma)\widehat{e^{\sigma t}v_{k,\ell}}(\xi).
	\end{equation}
	Indeed, multiplication by $e^{\sigma t}$ corresponds to shifting the Fourier variable to $\xi+i\sigma$.  Since $1+\lambda_k+\xi^2\sim 1+k^2+\xi^2$ uniformly in $k$, applying Plancherel's theorem, summing in $k$ and $\ell$ and using \eqref{eq:shifted-symbol-estimate} yields
	\begin{equation}\label{Pn H perp}
	C^{-1}\norm{v}_{H_\nu^{n}}\le	\norm{P_n v}_{L_\nu^2}\le C\norm{v}_{H_\nu^{n}}, \quad\text{whenever } \Pi_0v=0
	\end{equation}
for some constant $C>1$.  Smooth functions in $\Pi_\perp H_\nu^n(\Cyl)^{G,e}$ that are compactly supported in $t$ and whose spherical-harmonic expansions involve only finitely many degrees form a dense subspace. Hence \eqref{Pn H perp} is valid for all $v\in \Pi_\perp H_\nu^n(\mathcal{C})^{G,e}$.
	
We now prove surjectivity onto
$\Pi_\perp L_\nu^2(\Cyl)^{G,e}$.  Smooth functions in this space that are compactly supported in $t$ and whose spherical-harmonic expansions involve only finitely many degrees form a dense subspace. Let $f$ be such a function. Expand $f$ in spherical harmonics and define $v_{k,\ell}=G_{n,k}*f_{k,\ell}$. Then we have $\widehat{v_{k,\ell}}(\xi)=\frac{\widehat{f_{k,\ell}}(\xi)}{a_{n,k}(\xi)}$. Defining $v$ by \eqref{v spher decom}, the preceding argument shows that $P_nv=f$ and \eqref{Pn H perp} holds. For a general $f\in \Pi_\perp L_\nu^2(\mathcal{C})^{G,e}$, the same result follows from approximation. Thus, $P_n$ is surjective on the nonradial $G$-invariant subspace. In conclusion, $P_n: \Pi_\perp H_\nu^n(\Cyl)^{G,e}\to \Pi_\perp L_\nu^2(\Cyl)^{G,e}$ is an isomorphism.

\medskip

We now prove the solvability criterion and weighted estimate for $-y''=f$.  Since $a_{n,0}(\xi)=\xi^2b_n(\xi)$ has a double zero at zero, we first consider the operatoe $-\partial_t^2$. If $y\in H_\nu^2(\R)$ and $-y''=f$, the decay of $y'$ gives $\int_\R f=0$.
Conversely, let $f\in L_\nu^2(\R)^e$ have zero integral. Cauchy--Schwarz gives $sf(s)\in L^1$ and evenness gives $\int_\R sf(s)\dd s=0$. The absolutely convergent integral
\begin{equation}\label{-y''=f}
 y(t):=-\frac12\int_\R|t-s|f(s)\dd s
\end{equation}
defines an even solution of $-y''=f$ in distributions. Since $f$ is even, for $t\ge0$ and $F(t)=e^{\nu t}f(t)$, we can write
\begin{equation*}
 e^{\nu t}y(t)=-\int_0^\infty r e^{-\nu r}F(t+r)\dd r,\qquad
 e^{\nu t}y'(t)=\int_0^\infty e^{-\nu r}F(t+r)\dd r.  
\end{equation*}
Young's inequality then implies
\begin{equation*}
    \|e^{\nu t}y\|_{L^2(0,\infty)}+\|e^{\nu t}y'\|_{L^2(0,\infty)}\le C_\nu \|e^{\nu t}f\|_{L^2(0,\infty)}.
\end{equation*}
Together with $-y''=f$ and evenness, this yields
\begin{equation}\label{eq:second-derivative-estimate}
 \norm{y}_{H_\nu^2(\R)}\le C\norm{f}_{L_\nu^2(\R)}.
\end{equation}

We now prove that $B_n:H_\nu^n(\R)^e\to H_\nu^2(\R)^e$ is an isomorphism.  Since $\nu<2$, we have $e^{\pm\nu t}g_n\in L^1(\R)$.
Exponential conjugation and \eqref{eq:shifted-b-estimate}
therefore give
\[
 \norm{B_n\phi}_{H_\nu^2}\le C\norm{\phi}_{H_\nu^n},
 \qquad
 \norm{g_n*y}_{H_\nu^n}\le C\norm{y}_{H_\nu^2}.
\]
These operators are inverse Fourier multipliers and preserve evenness. This proves that $B_n:H_\nu^n(\R)^e\to H_\nu^2(\R)^e$ is an isomorphism.

We now prove the kernel and range formulas for $A_{n,0}$.  Identify $(c,\phi)\in\R\times H_\nu^n(\R)^e$ with $c+\phi$. Since $A_{n,0}(c+\phi)=-(B_n\phi)''$, a kernel element has $B_n\phi$ affine and in $H_\nu^2$, hence $B_n\phi=0$ and $\phi=0$. Thus
\begin{equation}\label{eq:radial-kernel}
 \Ker A_{n,0}=\Span\{1\}.
\end{equation}
Every range element has zero integral. Conversely, for an even $f$ with zero integral, take $y$ from \eqref{-y''=f} and $\phi=B_n^{-1}y$. Then $A_{n,0}\phi=f$ and by \eqref{eq:second-derivative-estimate},
\begin{equation*}
    \norm{\phi}_{H_\nu^n}\le C\norm{y}_{H_\nu^2}\le C\norm{f}_{L_\nu^2}.
\end{equation*}
Therefore
\begin{equation}\label{eq:radial-range}
 \Ran A_{n,0}=\left\{f\in L_\nu^2(\R)^e:\int_\R f(t)\dd t=0\right\}.
\end{equation}
Integration is a nonzero continuous functional on $L_\nu^2(\R)^e$, so the range is closed with codimension one and $\ind A_{n,0}=0$.  Combining this with the nonradial isomorphism gives \eqref{eq:free-kernel},
\eqref{eq:free-range} and $\ind P_n=0$.

We now prove that multiplication by $W_\alpha$ is compact. The result for $L_{\alpha}$ then follows by compact perturbation. Normality gives $W_\alpha(t)=O(e^{-n\alpha|t|})$. Since $\nu<n\alpha$, $W_\alpha\in L_\nu^2(\Cyl)$, so $c\mapsto cW_\alpha$ is bounded and finite rank. Choose smooth cutoffs $0\le\chi_R\le1$, equal to one for $|t|\le R$ and zero for $|t|\ge2R$. Rellich's theorem on bounded cylinders makes $\phi\mapsto\chi_RW_\alpha\phi$ compact from $H_\nu^n(\Cyl)$ to $L_\nu^2(\Cyl)$ and
\begin{equation*}
 \norm{(1-\chi_R)W_\alpha\phi}_{L_\nu^2}
 \le \sup_{|t|>R}|W_\alpha(t)|\,\norm{\phi}_{L_\nu^2}
 \le Ce^{-n\alpha R}\norm{\phi}_{H_\nu^n}.
 \end{equation*}

Operator-norm convergence and the finite-rank constant part therefore give compactness of multiplication by $W_\alpha$ from $\cX_\nu^{G,e}$ to $\cY_\nu^{G,e}$. The compact perturbation theorem \cite[Chapter~IV, Theorem~5.26]{Kato1995} now shows that $L_\alpha=P_n-W_\alpha$ is Fredholm with $\ind L_\alpha=\ind P_n=0$. This completes the proof.
\end{proof}

We next prove smooth dependence of the radial solution on $\alpha$. Since its asymptotic slope varies with $\alpha$, we first remove this prescribed linear part. For this purpose, we construct a fixed smooth even function $\tau$ that behaves like $|t|$ at infinity and work with $v_\alpha+\alpha\tau$ in a fixed weighted space.

Take a smooth even function $\eta\in C_c^\infty(\R)$ such that
\begin{equation*}
	\int_\R \eta(t)\dd t=1.
\end{equation*}
Define
\begin{equation}\label{def tau}
	\tau(t)=2\beta_n (\mathcal{K}_n*\eta)(t)+2\beta_n\int_0^\infty sg_n(s)\dd s.
\end{equation}

\begin{corollary}\label{coro v al smooth}
	$v_\alpha$ depends smoothly on $\alpha$. To be more precise, we fix $\alpha_0>0$ and $0<\nu<\min\{1,n\alpha_0\}$. Then $0<\nu<\min\{1,n\alpha\}$ for all $\alpha\in I$, a small neighborhood of $\alpha_0$. For $\tau(t)$ defined in \eqref{def tau}, the map
    \begin{equation*}
        \alpha\mapsto v_\alpha+\alpha\tau 
    \end{equation*}
    is smooth from $I$ to $\R\times H_\nu^n(\R)^e$
\end{corollary}
\begin{proof}
The definition of $\tau$ gives
\begin{equation*}
    A_{n,0}\tau=-2\beta_n\eta,
\end{equation*}
and by \eqref{eq:K-linear-revised}, we see that
\begin{equation*}
	\tau(t)=|t|+O(e^{-2|t|}),\qquad |t|\to\infty.
\end{equation*}

By \eqref{eq:v al infty} and \eqref{int e nv}, the function $h_\alpha:=(n-1)!e^{nv_\alpha}-2\beta_n\alpha\eta$ satisfies
\begin{equation*}
	h_\alpha\in L_\nu^2(\R)^e, \qquad \int_\R h_\alpha(t)\dd t=0.
\end{equation*}
Equations \eqref{eq:radial-kernel} and \eqref{eq:radial-range} imply that $A_{n,0}$ is an isomorphism from ${H}^n_{\nu}(\R)^e$ onto the zero-mean subspace of $L^2_\nu(\R)^e$. Therefore, there exists a unique $\varphi_\alpha\in {H}^n_{\nu}(\R)^e$ such that
\begin{equation*}
	A_{n,0}\varphi_\alpha=h_\alpha.
\end{equation*}
Consequently,
\begin{equation*}
	A_{n,0}(v_\alpha+\alpha\tau-\varphi_\alpha)=0.
\end{equation*}
The function in parentheses is an even tempered distribution. By Lemma \ref{lem:tempered-radial-kernel}, there is a constant $c_\alpha$ such that
\begin{equation}\label{eq:v-alphadecomposition}
	v_\alpha=-\alpha\tau+c_\alpha+\varphi_\alpha.
\end{equation}

We define the map
\begin{equation*}
	\mathscr{F}: I \times\R\times H^{n}_\nu(\R)^e
	\longrightarrow L^2_\nu(\R)^e
\end{equation*}
by
\begin{equation*}
	\mathscr{F}(\alpha,c,\varphi)=A_{n,0}\varphi+2\beta_n\alpha\eta-(n-1)!e^{n(-\alpha\tau+c+\varphi)}.
\end{equation*}
Then for $v=-\alpha\tau+c+\varphi$, $\mathscr{F}(\alpha,c,\varphi)=0$ is equivalent to the radial equation \eqref{eq:radial-equation}.

Using the embedding $H^n_\nu(\R)\to L^\infty(\R)$ and the fact that $e^{-n\alpha\tau},$ together with all of its local $\alpha$-derivatives, belongs to $L^2_\nu(\R)$, we obtain that $\mathscr{F}$ is smooth in all its arguments. Moreover, its derivative with respect to $(c,\varphi)$ is 
\begin{equation*}
    D_{(c,\varphi)}\mathscr{F}(\alpha,c,\varphi)[d,\psi]=A_{n,0}\psi-n!e^{n(-\alpha\tau+c+\varphi)}(d+\psi),
\end{equation*}
which agrees with $L_{\alpha,0}$ when $v_\alpha=-\alpha\tau+c_\alpha+\varphi_\alpha$. By Proposition~\ref{prop:radial-nondegeneracy}, the kernel of $L_{\alpha,0}$ is spanned by the odd function $v_\alpha'$, so $L_{\alpha,0}(d+\psi)=0$ implies that $d=0$ and $\psi=0$. On the other hand, Proposition \ref{prop:main-fredholm} implies that $L_{\alpha,0}$ is Fredholm of index zero, so it is an isomorphism. Therefore, by the implicit function theorem (see \cite[Theorem~15.1 and Corollary~15.1]{Deimling1985}), the map 
\begin{equation*}
	\alpha \mapsto (c_\alpha, \varphi_\alpha)
\end{equation*}
is smooth into $\R\times H^{n}_\nu(\R)^e$. This produces a unique local smooth family
\begin{equation*}
    \tilde{v}_\alpha=-\alpha\tau+\tilde c_\alpha+\tilde\varphi_\alpha.
\end{equation*}
By construction, $\tilde{v}_\alpha$ is an even normal solution of \eqref{eq:radial-equation}. The uniqueness therefore gives $\tilde{v}_\alpha=v_\alpha$. This proves the smooth dependence.
\end{proof}

\medskip
Using Theorem \ref{thm e value L}, we obtain the following corollary at the bifurcation point $\alpha_{n,k}$.
\begin{corollary}\label{cor:crossing-range}
	Let $\alpha_*=\alpha_{n,k}$ with $k\ge2$. Assume that $\dim\mathcal{H}_k^G=1$. Let $Y_k$ span $\mathcal{H}_k^G$ and let	$\psi_{\alpha_*,k}$ span $\Ker L_{\alpha_*,k}$. Set
	\begin{equation*}
		\Phi_k(t,\omega):=\psi_{\alpha_*,k}(t)Y_k(\omega).
	\end{equation*}
	Then for $L_{\alpha_*}: \cX_\nu^{G,e} \to \cY_\nu^{G,e}$, we have
	\begin{equation}\label{eq:crossing-kernel}
		\Ker L_{\alpha_*}=\Span\{\Phi_k\},
	\end{equation}
	and
	\begin{equation}\label{eq:crossing-range}
		\Ran L_{\alpha_*}=\left\{f\in\cY_\nu^{G,e}:\int_\Cyl f(t,\omega)\Phi_k(t,\omega)\dd t\dd\omega=0\right\}.
	\end{equation}
\end{corollary}
\begin{proof}
    We first show that the eigenfunction $\psi_{\alpha_*,k}$ is in the desired weighted space. In fact, since $W_\alpha\psi_{\alpha_*,k}\in L^2_\nu(\R)$, the equation $\psi_{\alpha_*,k}=A_{n,k}^{-1}(W_\alpha \psi_{\alpha_*,k})$ and the weighted inverse estimate in the proof of Proposition \ref{prop:main-fredholm} show that $\psi_{\alpha_*,k}\in H^n_{\nu}(\R)$ for $\nu<\min\{1,n\alpha_*\}$. Then \eqref{eq:crossing-kernel} follows directly from Theorem \ref{thm e value L}. Now let $u=c+\phi\in \cX_\nu^{G,e}$. Since $\Phi_k$ is orthogonal to constants, by self-adjointness, we have
    \begin{equation*}
        \int_{\Cyl}L_{\alpha_*}(c+\phi)\Phi_k=\int_{\Cyl}L_{\alpha_*}(\Phi_k)\phi=0.
    \end{equation*}
Consequently,
\begin{equation*}
    \Ran L_{\alpha_*}\subseteq\left\{f\in\cY_\nu^{G,e}:\int_\Cyl f(t,\omega)\Phi_k(t,\omega)\dd t\dd\omega=0\right\}.
\end{equation*}
Since $L_{\alpha_*}$ has index zero and $\Ker L_{\alpha_*}$ is one-dimensional, $\Ran L_{\alpha_*}$ has codimension $1$. Then \eqref{eq:crossing-range} follows.
\end{proof}
\subsection{Transversality}\label{sec:transversality}

In view of Corollary \ref{coro v al smooth}, we can define
\begin{equation*}
 q_\alpha:=\partial_\alpha v_\alpha.
\end{equation*}
Then $q_\alpha$ is even because $v_\alpha$ is even. Differentiating \eqref{eq:radial-equation} with respect to $\alpha$ gives
\begin{equation}\label{eq:q-jacobi}
 A_{n,0}q_\alpha=W_\alpha q_\alpha.
\end{equation}
Differentiating \eqref{eq:v-alphadecomposition} and using
Corollary~\ref{coro v al smooth} gives
\[
 q_\alpha=-\tau+\partial_\alpha c_\alpha
 +\partial_\alpha\varphi_\alpha,
 \qquad
 \partial_\alpha\varphi_\alpha\in H_\nu^n(\R)^e.
\]
Since $H_\nu^n(\R)\hookrightarrow L^\infty(\R)$ and $\tau(t)=|t|+O(e^{-2|t|})$, this yields $q_\alpha(t)=-|t|+O(1)$ as $|t|\to\infty$.

\begin{lemma}\label{thm:cumulative-positivity}

For $t\ge0$, define
\begin{equation*}
 M(t):=\int_0^tW_\alpha(s)q_\alpha(s)\dd s.
\end{equation*}  
Then $M(t)>0$ for all $t>0$.
\end{lemma}

\begin{proof}
We extend $M$ to an odd function on $\R$. By definition $M'=W_\alpha q_\alpha$. Applying $B_n^{-1}$ to \eqref{eq:q-jacobi} and integrating once, we obtain $-q_\alpha'=g_n*M$. In other words, we can write
\begin{equation}\label{eq:M-halfline}
 -\left(\frac{M'}{W_\alpha}\right)'(t)
 =\int_0^\infty\bigl[g_n(|t-s|)-g_n(t+s)\bigr]M(s)\dd s.
\end{equation}
The kernel $g_n(|t-s|)-g_n(t+s)$ is strictly positive because $g_n$ is strictly decreasing. 

On the other hand, since $\frac{W_\alpha'}{W_\alpha}=nv_\alpha'$, \eqref{eq:strict-concavity} gives
\begin{equation}\label{eq:N-halfline}
 -\left(\frac{W_\alpha'}{W_\alpha}\right)'(t)
 =\int_0^\infty\bigl[g_n(|t-s|)+g_n(t+s)\bigr]W_\alpha(s)\dd s.
\end{equation}

Differentiating the mass identity \eqref{eq:W-mass} gives
\begin{equation*}
 \lim_{t\to\infty}M(t)=\frac{1}{2}\int_\R W_\alpha q_\alpha\dd t=\beta_n>0.
\end{equation*}
Let $R=\frac{M}{W_\alpha}$. Then $R(0)=0$ and $R(t)\to+\infty$ as $t\to\infty$. Multiplying \eqref{eq:M-halfline} by $W_\alpha$, multiplying \eqref{eq:N-halfline} by $M$ and subtracting gives
\begin{align}\label{eq:ratio-identity}
 -\left({W_\alpha}R'\right)'(t)
 ={}&W_\alpha(t)\int_0^\infty
 \bigl[g_n(|t-s|)-g_n(t+s)\bigr]W_\alpha(s)[R(s)-R(t)]\dd s\notag\\
 &-2M(t)\int_0^\infty g_n(t+s)W_\alpha(s)\dd s.
\end{align}

Now we prove that $R\geq 0$.
If $R$ were negative somewhere, then it would have a negative global minimum at some $t_0>0$.  At $t=t_0$, the left side of \eqref{eq:ratio-identity} is nonpositive.  At $t=t_0$, the first term on the right is nonnegative, while the second is strictly positive because $M(t_0)<0$.  This is impossible.  Thus $R\ge0$.

Now we prove that $R>0$.
If $R(t_0)=0$ for some $t_0\in (0,\infty)$, then $t_0$ would be a global minimum equal to zero. The second term of the right hand side of \eqref{eq:ratio-identity} would vanish, but the first would be strictly positive because $R$ is not identically zero.  This is again impossible.  Hence $R>0$ and $M=RW_\alpha>0$ in $(0,\infty)$.

For $n=3$, $g_3$ has a logarithmic singularity at zero but
belongs to $L^1(\R)$, so its convolutions with the bounded
functions $M$ and $W_\alpha$ are continuous.
The identities therefore hold pointwise and the same minimum argument applies.
\end{proof}

Now we are in a position to prove \eqref{eq:mu-al-increasing}.

\begin{proof}[Proof of \eqref{eq:mu-al-increasing}]
The monotonicity is stated for every real $\kappa>0$. Here we only prove it for positive integer $k$, and the same argument works for all $\kappa>0$.
Let $\psi_{\alpha,k}>0$ be the even principal eigenfunction defined by \eqref{eq:generalized-eigenpair}. Then by Proposition \ref{thm:mu2}, $\partial_t\psi_{\alpha,k}(t)<0$ for $t>0$.

An integration by parts, together with
Lemma~\ref{thm:cumulative-positivity} and
$\partial_t(\psi_{\alpha,k}^2)<0$ on $(0,\infty)$, gives
\begin{equation*}
 \int_\R W_\alpha q_\alpha\psi_{\alpha,k}^2\dd t
 =-2\int_0^\infty M(t)(\psi_{\alpha,k}^2)'(t)\dd t>0.
\end{equation*}
The boundary terms vanish because $M(0)=0$, $M$ is bounded at infinity and
$\psi_{\alpha,k}$ decays exponentially.

Corollary \ref{coro v al smooth} proves smooth dependence of the radial profile $v_\alpha$ on $\alpha$ in a weighted space. Moreover, locally in $\alpha$, the exponential decay of $W_\alpha$ implies the differentiability of the integral kernel $\eqref{int kernel K}$. Hence the operator $\mathscr{K}_{\alpha,k}$ depends smoothly on $\alpha$ in operator norm. By the simplicity of the principal Birman--Schwinger eigenvalue and perturbation theory (see \cite[Chapter VII]{Kato1995}), the simple eigenvalue has a locally smooth normalized eigenfunction.

Now we can differentiate the generalized eigenvalue equation 
$A_{n,k}\psi_{\alpha,k}=\tau_{\alpha,k} W_\alpha\psi_{\alpha,k}$ with respect to $\alpha$  to obtain
\begin{equation*}
 \partial_\alpha\tau_{\alpha,k}
 =-n\tau_{\alpha,k}
 \frac{\int_\R W_\alpha q_\alpha\psi_{\alpha,k}^2\dd t}
 {\int_\R W_\alpha\psi_{\alpha,k}^2\dd t}<0.
\end{equation*}
Since $\tau_{\alpha,k}=\frac{1}{\mu_1({\alpha,k})}$, this proves \eqref{eq:mu-al-increasing}.
\end{proof}

Let $\lambda_{n,k}(\alpha)$ be the corresponding principal eigenvalue of $L_{\alpha,k}$. Since $\partial_\alpha W_\alpha=nW_\alpha q_\alpha$, we have
\begin{equation}\label{eq:lambda-transverse}
 \lambda_{n,k}'(\alpha_*)=-n\frac{\int_\R W_{\alpha_*}q_{\alpha_*}\psi_{\alpha_*,k}^2\dd t}{\int_\R\psi_{\alpha_*,k}^2\dd t}<0.
\end{equation}
Thus, by \eqref{eq:crossing-range} and \eqref{eq:lambda-transverse}, we have the following transversality property.
\begin{proposition}\label{prop:transversality}
For $\alpha_*=\alpha_{n,k}$, we have
\begin{equation*}
 \partial_\alpha L_{\alpha,k}\big|_{\alpha=\alpha_*}\psi_{\alpha_*,k}
 =-nW_{\alpha_*}q_{\alpha_*}\psi_{\alpha_*,k} \notin\Ran L_{\alpha_*,k}.
\end{equation*}
\end{proposition}

\begin{proof}
Since $L_{\alpha_*,k}$ is self-adjoint and
$L_{\alpha_*,k}\psi_{\alpha_*,k}=0$, every element of its range
is $L^2$-orthogonal to $\psi_{\alpha_*,k}$.
On the other hand, \eqref{eq:crossing-range} and \eqref{eq:lambda-transverse} give
\[
 \left\langle
 \left.\partial_\alpha L_{\alpha,k}\right|_{\alpha=\alpha_*}
 \psi_{\alpha_*,k},\psi_{\alpha_*,k}
 \right\rangle_{L^2}
 =-n\int_\R
 W_{\alpha_*}q_{\alpha_*}\psi_{\alpha_*,k}^2\dd t<0,
\]
which proves the assertion.
\end{proof}

\subsection{Local nonradial bifurcation}\label{sec:bifurcation}
\begin{proof}[Proof of Theorem \ref{thm:main-bifur}]
Let $G=O(n-1)$ be the subgroup that fixes the last coordinate axis.  For each degree $j$, the $G$-invariant spherical-harmonic space $\mathcal{H}^G_j$ is one-dimensional, generated by
\begin{equation*}
 Y_j(\omega)=c_j\,C_j^{\frac{n-2}{2}}(\omega_n),
\end{equation*}
where $C_j^{\frac{n-2}{2}}$ is the Gegenbauer polynomial and $c_j$ is a nonzero normalizing constant. Thus by \eqref{eq:kG}, $k_G=1$.  For $\lambda>0$, the Gegenbauer polynomials $C_j^\lambda$ are defined by the generating function \cite[Eq.~(18.12.4)]{DLMF}
\[
 (1-2z\rho+\rho^2)^{-\lambda}
 =\sum_{j=0}^\infty C_j^\lambda(z)\rho^j,
 \qquad z\in[-1,1],\quad |\rho|<1.
\]

Fix $k\ge2$, set $\alpha_*=\alpha_{n,k}$ and choose
\begin{equation*}
 0<\nu<\min\{1,n\alpha_*\}.
\end{equation*}
Let $\cX_\nu^{G,e}$ and $\cY_\nu^{G,e}$ be defined in \eqref{eq:X-Y}. Define the map
\begin{equation*}
 \mathcal F(\alpha,w):=P_nw-(n-1)!e^{nv_\alpha}(e^{nw}-1).
\end{equation*}
Let $I$ be a neighborhood of $\alpha_*$ such that $\nu<n\alpha$ for all $\alpha\in I$. Then
\begin{equation*}
 \mathcal F:I\times\cX_\nu^{G,e} \longrightarrow \cY_\nu^{G,e}
\end{equation*}
is smooth near $(\alpha_*,0)$ and
\begin{equation*}
 \mathcal F(\alpha,0)=0, \qquad D_w\mathcal F(\alpha,0)=P_n-W_\alpha= L_\alpha.
\end{equation*}

At $\alpha_*=\alpha_{n,k}$, the centered radial mode is invertible.  Since $\dim\mathcal H_k^G=1$, Corollary~\ref{cor:crossing-range}
shows that
\[
 \Ker\bigl(L_{\alpha_*}:\cX_\nu^{G,e}\to\cY_\nu^{G,e}\bigr)
 =\Span\{\psi_{\alpha_*,k}(t)Y_k(\omega)\}.
\]
The Fredholm property of $L_\alpha$ is proved in Proposition \ref{prop:main-fredholm} and Proposition \ref{prop:transversality} gives the transversality condition.

In summary, all hypotheses of Theorem \ref{thm:quoted-CR} are satisfied. As a result, a local branch of $O(n-1)$-axisymmetric, nonradial solutions bifurcates from $v_{\alpha_*}$ for $\alpha_*=\alpha_{n,k}$. More precisely, there are $\varepsilon_0>0$ and smooth functions $s\mapsto(\alpha(s),w(s))$ for $|s|<\varepsilon_0$ such that
\begin{equation*}
 \alpha(0)=\alpha_{*}, \qquad w(0)=0, \qquad w(s)=s\psi_{\alpha_*,k}(t)Y_k(\omega)+o(s) \text{ in } \cX_\nu^{G,e},
\end{equation*}
and $V_s(t,\omega)=v_{\alpha(s)}(t)+w(s;t,\omega)$ solves \eqref{eq:cylinder-equation}.

Returning to Euclidean variables, we set
\[
 u_s(r,\omega)
 =V_s(\log r,\omega)-\alpha(s)\log r.
\]
Since $w(s)\in\cX_\nu^{G,e}$, we write
\[
 w(s)=c_s+\varphi_s,\qquad
 \varphi_s\in H_\nu^n(\Cyl)^{G,e}.
\]
Weighted Sobolev embedding gives
$\varphi_s(t,\omega)=O(e^{-\nu|t|})$ uniformly in $\omega$. Hence $u_s$ extends continuously to the origin,
\[
 u_s(x)=-2\alpha(s)\log|x|+O(1)
 \quad\text{as }|x|\to\infty,
\]
and
\begin{equation*}
\int_{\R^n}|x|^{n(\alpha(s)-1)}e^{nu_s}\dd x=\int_\Cyl e^{n(v_{\alpha(s)}+w(s))}\dd t\dd\omega<\infty.    
\end{equation*}
 
The equation holds away from the origin by conformal covariance. Lemma~\ref{lem:removability-normality} therefore shows that $u_s$ is a normal solution on $\R^n$. This completes the proof.
\end{proof}

\section{Radial symmetry for small $\gamma$}\label{sec:rad small ga}

In this section, we prove Theorem~\ref{thm:rad small ga}.
By Theorem~\ref{thm:kelvin}, every normal solution satisfying
the median normalization \eqref{rig:eq:median} obeys $u(x)=u({x}/{|x|^2})-2\alpha\log|x|$
 for $x\ne0$.
The mass identity \eqref{eq:normal-mass-intro}, Lemma~\ref{rig:lem:translation-identity}, and Kelvin symmetry exclude concentration as $\gamma\downarrow0$. Compactness and classification then give convergence to the standard bubble \eqref{rig:eq:standard-bubble}.

We argue by contradiction and compare a hypothetical nonradial solution with the normalized radial solution at the same parameter and normalize their difference in $L^\infty(\R^n)$. Kelvin symmetry ensures that its norm is attained on $\overline B_1$, so local convergence produces a nonzero bounded solution of the linearized equation at the standard bubble. Nondegeneracy identifies this limit as a linear combination of the basis of the kernel. Lemma~\ref{rig:lem:translation-identity} gives a vanishing gradient at the origin, which eliminates the translation modes. Kelvin symmetry eliminates the dilation mode, yielding a contradiction. Uniqueness under the median normalization then follows from Theorem~\ref{thm radial}.

\subsection{Normalization and compactness as \texorpdfstring{$\gamma\downarrow0$}{gamma tends to zero}}

Recall that every normal solution of \eqref{eq:euclidean} is Kelvin symmetric by Theorem \ref{thm:kelvin}. Then by \eqref{R half mass}, we see that each normal solution has a unique dilation satisfying
\begin{equation}\label{rig:eq:median}
 \int_{B_1}|x|^{n\gamma}e^{nu}\dd x=\frac{\alpha}{2}|\Sph^n|
\end{equation}
where $\alpha=1+\gamma$.

Our main result in this subsection is:

\begin{proposition}\label{rig:prop:compactness}
Let $\gamma_j\downarrow0$ and let $u_j$ be normal solutions normalized by
\eqref{rig:eq:median}.  Then, after passing to a subsequence,
\begin{equation*}
 u_j\longrightarrow U \quad\text{in }C^0_{\mathrm{loc}}(\R^n)
 \text{ and in }C^\ell_{\mathrm{loc}}(\R^n\setminus\{0\})
 \text{ for every }\ell\geq0,
\end{equation*}
where
\begin{equation}\label{rig:eq:standard-bubble}
 U(x)=\log\frac{2}{1+|x|^2}.
\end{equation}
\end{proposition}

The proof is divided into several steps: a local estimate under a small-mass condition, the exclusion of a possible blow-up profile, and the passage to a normalized limit. In this section, we set
\begin{equation*}
 f_j=|x|^{n\gamma_j}e^{nu_j}, \qquad
 \alpha_j=1+\gamma_j,\qquad
 \Lambda_j=\alpha_j|\Sph^n|.    
\end{equation*}
We introduce an auxiliary kernel
\begin{equation*}
    K(x,y)=\log\frac{\sqrt{1+|x|^2}\sqrt{1+|y|^2}}{|x-y|}.
\end{equation*}
Then
\begin{equation*}
    K\ge 0, \qquad K(x,y)=K(\frac{x}{|x|^2}, \frac{y}{|y|^2}).
\end{equation*}
Define
\begin{equation}\label{def Hj}
    H_j(x)=\int_{\R^n}K(x,y)f_j(y)\dd y.
\end{equation}
Then, the normal representation of $u_j$ can be written in the form
\begin{equation}\label{u rep new}
    u_j(x)=d_j-\alpha_j\log(1+|x|^2)+\frac{2}{|\Sph^n|}H_j(x).
\end{equation}
Consequently
\begin{equation}\label{f_j rep new}
    f_j(x)=e^{nd_j}\frac{|x|^{n\gamma_j}}{(1+|x|^2)^{n\alpha_j}}e^{\frac{2n}{|\Sph^n|}H_j(x)}.
\end{equation}
Since we consider $\gamma_j\downarrow 0$, we can assume without loss of generality that $0<\gamma_j<1$. Since $H_j\ge 0$, integrating \eqref{f_j rep new} gives a uniform upper bound for $d_j$.

 In the spirit of \cite{BrezisMerle1991}, we derive a small-mass logarithmic-potential estimate and establish a uniform $L^2$ estimate.
\begin{lemma}\label{lem: fj L2 uniform}
Suppose that, for a fixed ball $B_{2r}(q)$,
\begin{equation*}
    m_j:=\int_{B_{2r}(q)}f_j\dd x\le \frac{|\Sph^n|}{16}.
\end{equation*}
Then $\|f_j\|_{L^2(B_{r}(q))}\le C$ for some $C$ that depends only on $q$ and $r$ and is uniform in $j$.
\end{lemma}
\begin{proof}
    For $x\in B_r(q)$, the kernel $K(x,y)$ is uniformly bounded for $y\notin B_{2r}(q)$. For $y\in B_{2r}(q)$, we have $K(x,y)\le C+\log^+\frac{1}{|x-y|}$. Thus, by \eqref{def Hj} and \eqref{f_j rep new}, we have
    \begin{equation*}
        f_j(x)^2\le C\exp\Big\{\frac{4n}{|\Sph^n|}\int_{B_{2r}(q)} \log^+\frac{1}{|x-y|} f_j(y)\dd y\Big\}
    \end{equation*}
    Applying Jensen's inequality to the probability measure $\frac{\chi_{B_{2r}(q)}}{m_j}f_j(y)\dd y$, we obtain
    \begin{equation*}
        f_j(x)^2\le \frac{C}{m_j}\int_{B_{2r}(q)} \max\{1,{|x-y|}^{-\frac{4n}{|\Sph^n|}m_j} \}f_j(y)\dd y.
    \end{equation*}
    Since $\frac{4n}{|\Sph^n|}m_j\le \frac{n}4<n$,  integrating in $x\in B_r(q)$ gives
    \begin{equation*}
           \int_{B_{r}(q)}f_j(x)^2\dd x\le \frac{C}{m_j}\int_{B_{2r}(q)} \int_{B_r(q)}\max\{1,{|x-y|}^{-\frac{4n}{|\Sph^n|}m_j} \}\dd xf_j(y)\dd y\le C.
    \end{equation*}
    This proves the lemma.
\end{proof}

To prove Proposition \ref{rig:prop:compactness}, a crucial step is to exclude blow-up of the sequence $\{u_j\}$. The following lemma shows that if $\{u_j\}$ blows up, then after scaling, it converges to a standard bubble of a regular $Q$-curvature equation.
\begin{lemma}\label{rig:lem:blowup-profile}
Let $\{u_j\}$ be a sequence in Proposition~\ref{rig:prop:compactness}. Assume that $\max_{\R^n}u_j\to\infty$.  Let $x_j$ be a maximum point and let $r_j>0$ be the unique solution of
\begin{equation}\label{rig:eq:implicit-scale}
 r_j e^{u_j(x_j)}(|x_j|+r_j)^{\gamma_j}=1.
\end{equation}
Then $r_j\to0$ and up to a subsequence,
\begin{equation*}
\eta_j(z):=u_j(x_j+r_jz)-u_j(x_j)
\end{equation*}
converges in $C_{loc}(\R^n)$ to a normal solution of
\begin{equation}\label{rig:eq:limit-equation}
 (-\Delta)^{n/2}\eta=(n-1)!e^{n\eta}\quad\text{in }\R^n
\end{equation}
with
\begin{equation}\label{rig:eq:bubble-mass}
 \int_{\R^n}e^{n\eta}\dd z=|\Sph^n|.
\end{equation}
\end{lemma}

\begin{proof}
The left-hand side of \eqref{rig:eq:implicit-scale} is strictly increasing from zero to infinity, so $r_j$ exists uniquely. Moreover, we get $r_j^{1+\gamma_j}\le e^{-u_j(x_j)}$ and hence $r_j\to0$.
The rescaled density is
\begin{equation}\label{fj til}
 \widetilde f_j(z):=K_j(z)e^{n\eta_j(z)}\text{ with } K_j(z)= \left(\frac{|x_j+r_jz|}{|x_j|+r_j}\right)^{n\gamma_j}.
\end{equation}
Then we have
 \begin{equation*}
     \int_{\R^n}\widetilde f_j\dd z=\int_{\R^n}f_j\dd x\le C.
 \end{equation*}
 
We can assume $\gamma_j\le1$ after discarding finitely many terms.  For $R\ge1$ and $|z|\le R$, we have $K_j(z)\le R^n$.
Passing to a subsequence, $\frac{x_j}{r_j}$ either converges or tends to infinity and $K_j\to1$ almost everywhere.

Subtracting the normal representations of $u_j$ at $x_j+r_jz$ and $x_j$ gives
\begin{equation}\label{eta rep}
 \eta_j(z)=\frac{2}{|\Sph^n|}\int_{\R^n}\log\frac{|\zeta|}{|z-\zeta|}\widetilde f_j(\zeta)\dd\zeta.   
\end{equation}

Since $\eta_j\le0$, $\widetilde f_j$ is uniformly bounded on every fixed ball. Fix $R\ge1$. Since $\widetilde f_j$ is locally bounded and belongs to $L^1(\R^n)$, we differentiate the logarithmic potential in the weak sense (see \cite[Lemma~2.3(i)]{Hyder2019}). This gives
\begin{equation*}
 \nabla\eta_j(z)=-\frac{2}{|\Sph^n|} \int_{\R^n}\frac{z-\zeta}{|z-\zeta|^2}\widetilde f_j(\zeta)\dd\zeta.    
\end{equation*}

For $z\in B_R$, the integral over $B_{2R}$ satisfies
 \begin{equation*}
 \int_{B_{2R}}\frac{\widetilde f_j(\zeta)}{|z-\zeta|}\dd\zeta
 \le \|\widetilde f_j\|_{L^\infty(B_{2R})}\int_{B_{3R}}\frac{\dd y}{|y|}
\le C_nR^{n-1}\|\widetilde f_j\|_{L^\infty(B_{2R})}.
 \end{equation*}
On the complement, $|z-\zeta|\ge|\zeta|-R\ge R$, so
\begin{equation*}
 \int_{\R^n\setminus B_{2R}}\frac{\widetilde f_j(\zeta)}{|z-\zeta|}\dd\zeta
 \le R^{-1}\|\widetilde f_j\|_{L^1(\R^n)}.
\end{equation*}
Consequently,
\begin{equation*}
 \|\nabla\eta_j\|_{L^\infty(B_R)}
 \le C_n\left(R^{n-1}\|\widetilde f_j\|_{L^\infty(B_{2R})}+R^{-1}\|\widetilde f_j\|_{L^1(\R^n)}\right)
 \le C_R,
\end{equation*}
where $C_R$ is independent of $j$.
Together with $\eta_j(0)=0$, this yields local $W^{1,\infty}$ bounds. Thus, up to a subsequence, $\eta_j\to\eta$ in $C^0_{\mathrm{loc}}(\R^n)$ and $\widetilde f_j\to e^{n\eta}$ in $L^1_{\mathrm{loc}}(\R^n)$. Fatou's lemma gives $\int_{\R^n}e^{n\eta}\dd z<\infty$.

We now pass to the limit in \eqref{eta rep}. For $|\zeta|>R>2|z|$, we have
\begin{equation*}
 \left|\log\frac{|\zeta|}{|z-\zeta|}\right|\le C\frac{|z|}{|\zeta|}.   
\end{equation*}
On fixed balls, we use the uniform convergence obtained above. Therefore, letting first $j\to\infty$ and then $R\to\infty$ in \eqref{eta rep} gives
\begin{equation*}
 \eta(z)=\frac{2}{|\Sph^n|}\int_{\R^n}\log\frac{|y|}{|z-y|}e^{n\eta(y)}\dd y
 =c_\eta+\frac{2}{|\Sph^n|}\int_{\R^n}\log\frac{1+|y|}{|z-y|}e^{n\eta(y)}\dd y.
\end{equation*}
Here $c_\eta$ is finite because $\eta\le0$ and $e^{n\eta}\in L^1$. This representation gives \eqref{rig:eq:limit-equation} and normality. Local regularity gives smoothness. The classifications \cite[Theorem~2]{Martinazzi2009} in even dimensions and \cite[Theorem~1.3]{Hyder2019} in odd dimensions now show that $\eta$ is a standard bubble. In particular, \eqref{rig:eq:bubble-mass} holds.
\end{proof}

Now we can rule out blow-up.
\begin{lemma}\label{rig:lem:no-concentration}
Let $\{u_j\}$ be a sequence in Proposition~\ref{rig:prop:compactness}.  Then $\max\limits_{\R^n}u_j\le C$ for some finite $C$ independent of $j$.
\end{lemma}

\begin{proof}
Suppose for contradiction that $\max\limits_{\R^n}u_j\to\infty$. We  take the maximum points $x_j$ and scales $r_j$ from Lemma~\ref{rig:lem:blowup-profile}. Let $\eta$ be the limit function in Lemma~\ref{rig:lem:blowup-profile}. After passing to a subsequence, $x_j\to x_*\in\R^n\cup\{\infty\}$.
For every $\varepsilon>0$, we can choose $R$ such that
\begin{equation*}
     \int_{B_R}e^{n\eta}\dd z>|\Sph^n|-\varepsilon.
\end{equation*}
For $f_j=|x|^{n\gamma_j}e^{nu_j}$, the local convergence of the rescaled densities \eqref{fj til} gives
\begin{equation*}
 \liminf_{j\to\infty}\int_{B_{Rr_j}(x_j)}f_j(x)\dd x
 \ge\int_{B_R}e^{n\eta}\dd z
 >|\Sph^n|-\varepsilon.
\end{equation*}
Since $r_j\to0$ and the total mass of $f_j\dd x$ is $\alpha_j|\Sph^n| \to |\Sph^n|$, the mass outside these shrinking balls tends to zero as $\varepsilon\downarrow0$. Consequently,
\begin{equation}\label{conver fj dx}
    f_j\dd x\rightharpoonup|\Sph^n|\delta_{x_*}
\end{equation}
The median normalization \eqref{rig:eq:median} rules out both $x_*=0$ and $x_*=\infty$, because otherwise the mass of $B_1$ would then tend respectively to the full mass or to zero.  Hence $x_*\in\R^n\setminus\{0\}$.

Since $x_*\neq 0$, we can choose $r_0>0$ such that $x_*\notin B_{2r_0}(0)$. Then by concentration at $x_*$, we can apply Lemma \ref{lem: fj L2 uniform} to obtain $\sup_j\|f_j\|_{L^2(B_{r_0}(0))}\le C_{r_0}$. Hence, for $0<\delta<r_0$, the Cauchy--Schwarz inequality implies
\begin{equation*}
     \int_{|x|<\delta}|x|^{-1}f_j(x)\dd x\le C_{r_0}(\int_{|x|<\delta}|x|^{-2}\dd x)^{\frac{1}{2}}\le C_{r_0}'\delta^{\frac{n-2}{2}}.
\end{equation*}
At infinity, we have the uniform bound
\begin{equation*}
   \int_{|x|>R}|x|^{-1}f_j(x)\dd x\le CR^{-1}.
\end{equation*}
Therefore by \eqref{conver fj dx}, we obtain
\begin{equation*}
     \int_{\R^n}\frac{x}{|x|^2} f_j(x)\dd x\longrightarrow |\Sph^n|\frac{x_*}{|x_*|^2}
\end{equation*}
because there is no loss of mass at $0$ or infinity. Passing to the limit in \eqref{rig:eq:inverse-moment} therefore gives $0=|\Sph^n|x_*$, contradicting \(x_*\ne0\). Hence concentration is impossible, and the lemma is proved.
\end{proof}

\begin{proof}[Proof of Proposition~\ref{rig:prop:compactness}]
By Lemma~\ref{rig:lem:no-concentration}, the functions $u_j$ are uniformly bounded above. Hence $f_j$ is uniformly bounded on every fixed ball. Using \eqref{grad u rep f} and splitting the integral into $|x-y|<1$ and its complement, we obtain, for every $R\ge1$,
\begin{equation*}
 \|\nabla u_j\|_{L^\infty(B_R)}
 \le \frac{2}{|\Sph^n|}\left(\|f_j\|_{L^\infty(B_{R+1})}\int_{B_1}|z|^{-1}\dd z+\|f_j\|_{L^1(\R^n)}\right)
 \le C_R.
\end{equation*}
On the other hand, since $\gamma_j>0$, the median normalization
\eqref{rig:eq:median} gives
\begin{equation*}
 \frac{|\Sph^n|}{2}\le \int_{B_1}f_j\dd x\le |B_1|\exp\left(n\sup_{B_1}u_j\right).\end{equation*}
Thus $\sup_{B_1}u_j$ is uniformly bounded below. Together with
the gradient estimate and the uniform upper bound, this yields
\begin{equation*}
 \sup_j\|u_j\|_{W^{1,\infty}(B_R)}<\infty \qquad\text{for every }R\ge1.   
\end{equation*}
The Arzel\`a--Ascoli theorem therefore
gives, after passing to a subsequence,
\begin{equation*}
 u_j\longrightarrow u_0\quad\text{in }C^0_{\mathrm{loc}}(\R^n).    
\end{equation*}

Let $\Omega'\Subset\Omega\Subset\R^n\setminus\{0\}$ and $1<p<\infty$. The local potential estimates used in the proof of \cite[Theorem~2.1]{HMM2021}, applied to the normal
representation, give
\[
 \|u_j\|_{W^{n,p}(\Omega')}
 \le C\left(
 \|f_j\|_{L^p(\Omega)}
 +\|f_j\|_{L^1(\R^n)}
 +\|u_j\|_{L^\infty(\Omega)}
 \right)\le C,
\]
where $C$ is independent of $j$. Since the weights $|x|^{n\gamma_j}$ have uniformly bounded derivatives of every order on $\Omega$, the regularity iteration in the same proof yields
\begin{equation*}
 \sup_j\|u_j\|_{W^{m,p}(\Omega')}<\infty
 \qquad\text{for every integer }m\ge1.   
\end{equation*}
Compact Sobolev embeddings therefore give, up to a subsequence,
\[
 u_j\longrightarrow u_0
 \quad\text{in }C^\ell_{\mathrm{loc}}(\R^n\setminus\{0\})
 \qquad\text{for every }\ell\ge0.
\]

We next control the mass at infinity. By the uniform bound of $u_j$ in $B_1$ and the Kelvin symmetry, we obtain
\begin{equation*}
     f_j(x)\leq C|x|^{-2n-n\gamma_j}, \qquad |x|>1
\end{equation*}
for some constant $C$ independent of $j$. Consequently,
\begin{equation}\label{int fj uniform}
 \sup_j\int_{|x|>R}(1+|x|)f_j(x)\dd x\leq CR^{-(n-1)}.   
\end{equation}
Thus both the masses and the first moments are uniformly integrable. Since \(\gamma_j\to0\) and \(u_j\to u_0\) locally uniformly,
\begin{equation*}
 f_j=|x|^{n\gamma_j}e^{nu_j}
 \longrightarrow e^{nu_0}
 \quad\text{in }L^1_{\mathrm{loc}}(\R^n).    
\end{equation*}
Combining this local convergence with \eqref{int fj uniform},
we obtain
\begin{equation*}
     \int_{\R^n}e^{nu_0}\dd x=\lim_{j\to\infty}\int_{\R^n}f_j\dd x
 =|\Sph^n|.
\end{equation*}

It remains to verify normality. Fix \(x_0\in\R^n\). Subtracting the normal representations at \(x\) and \(x_0\) gives
\begin{equation}\label{uj rep x x0}
 u_j(x)-u_j(x_0)=\frac{2}{|\Sph^n|}\int_{\R^n}\log\frac{|x_0-y|}{|x-y|}f_j(y)\dd y.   
\end{equation}
On each fixed ball, we pass to the limit by local convergence. On its complement,
\begin{equation*}
     \left|\log\frac{|x_0-y|}{|x-y|}\right|\leq\frac{C_{x,x_0}}{|y|}
\end{equation*}
for large $|y|$, so the uniform mass bound controls the integral outside fixed balls. Therefore, we can pass to the limit in \eqref{uj rep x x0} to obtain
\begin{equation*}
    u_0(x)-u_0(x_0)=\frac{2}{|\Sph^n|}\int_{\R^n}\log\frac{|x_0-y|}{|x-y|}e^{nu_0(y)}\dd y.
\end{equation*}
Rearranging this identity yields
\begin{equation*}
 u_0(x)=c_0+\frac{2}{|\Sph^n|}\int_{\R^n}\log\frac{1+|y|}{|x-y|}e^{nu_0(y)}\dd y   
\end{equation*}
for some $c_0\in\R$. The fundamental-solution identity
\begin{equation*}
 (-\Delta)^{n/2}\log\frac1{|x|}=\frac{(n-1)!|\Sph^n|}{2}\delta_0    
\end{equation*}
therefore gives
\begin{equation*}
    (-\Delta)^{n/2}u_0=(n-1)!e^{nu_0} \quad\text{in }\mathcal D'(\R^n).
\end{equation*}
Thus $u_0$ is a normal solution of the unweighted equation. The classifications \cite[Theorem~2]{Martinazzi2009} in even dimensions and \cite[Theorem~1.3]{Hyder2019} in odd dimensions now give
\begin{equation*}
 u_0(x)=\log\frac{2\lambda}{1+\lambda^2|x-\xi|^2}   
\end{equation*} 
for some $\lambda>0$ and $\xi\in\R^n$ . By the uniform integrability \eqref{int fj uniform} and
Corollary \ref{coro:translation-identity},
\begin{equation*}
0=\lim_{j\to\infty}\int_{\R^n}x f_j(x)\dd x=\int_{\R^n}x e^{nu_0(x)}\dd x=|\Sph^n|\xi,    
\end{equation*}
so \(\xi=0\). Finally, passing to the limit in \eqref{rig:eq:median} and
changing variables $y=\lambda x$ give
\begin{equation*}
    \int_{B_\lambda}\left(\frac{2}{1+|y|^2}\right)^n\dd y=\int_{B_1}e^{nu_0(x)}\dd x
 =\frac{|\Sph^n|}{2}.
\end{equation*}
The left-hand side is strictly increasing in $\lambda$ and equals $\frac{|\Sph^n|}{2}$ at $\lambda=1$. Therefore $\lambda=1$. Hence
\begin{equation*}
 u_0(x)=\log\frac{2}{1+|x|^2}=U(x)    
\end{equation*}
\end{proof}

\subsection{Proof of Theorem~\ref{thm:rad small ga}}

\begin{proof}[Proof of Theorem~\ref{thm:rad small ga}] 
By Theorem \ref{thm:kelvin}, we know that every normal solution is Kelvin symmetric. Let $u_\gamma^{\text{rad}}$ be the unique radial normal solution such that it is Kelvin symmetric with respect to $B_1$. 
We argue by contradiction. If the theorem were false, we could choose $\gamma_j\downarrow0$ and a normal solution $u_j$ of equation \eqref{eq:euclidean} with parameter $\gamma_j$, which is different from $u_j^{\text{rad}}:=u_{\gamma_j}^{\text{rad}}$ and also Kelvin symmetric with respect to $B_1$. By \eqref{rig:eq:median}, the Kelvin symmetry is equivalent to
\begin{equation*}
	\int_{B_1}(n-1)!|x|^{n\gamma_j}e^{nu_j^{\text{rad}}}\,dx=
	\int_{B_1}(n-1)!|x|^{n\gamma_j}e^{nu_j}\,dx=\frac{(n-1)!(1+\gamma_j)}{2}|\Sph^n|.
\end{equation*}

By Proposition \ref{rig:prop:compactness}, we have
\begin{equation}\label{uj conver U}
 u_j, u_j^{\text{rad}}\longrightarrow U=\log\frac{2}{1+|x|^2} \quad\text{in }C^0_{\mathrm{loc}}(\R^n)
\end{equation}

Set
\begin{equation*}
h_j:=\frac{u_{j}^{\text{rad}}-u_{j}}{\|u_{j}^{\text{rad}}-u_{j}\|_{L^\infty(\R^n)}}.
\end{equation*}
Then
\begin{equation}\label{eq:maximum}
h_j(\frac{x}{|x|^2})=h_j(x),\qquad
\|h_j\|_{L^\infty(\R^n)}=\max_{\overline B_1}|h_j|=1.
\end{equation}
We are going to show $h\equiv 0$ to obtain a contradiction.

Subtracting the equations of $u_j$ and $u_j^{\text{rad}}$ yields
\begin{equation*}\label{eq:linear}
(-\Delta)^{n/2}h_j=a_jh_j,\qquad
a_j(x)=n!|x|^{n\gamma_j}\int_0^1 e^{n((1-s)u_{j}^{\text{rad}}(x)+su_{j}(x))}\,ds.
\end{equation*}
For every fixed $R$,  \eqref{uj conver U} gives
\begin{equation}\label{eq:coeff}
0\le a_j\le C_R\quad\hbox{on }B_R,\qquad
a_j\longrightarrow n!e^{nU}\quad\hbox{in }L^p(B_R)
\quad(1\le p<\infty).
\end{equation}

There is also a uniform global $L^1$ bound requiring only the mass identity: 
\begin{equation}\label{eq:L1}
\int_{\R^n}a_j\,dx\le\frac{n!}{2}\int_{\R^n}|x|^{n\gamma_j}
\bigl(e^{nu_{j}^{\text{rad}}}+e^{nu_{j}}\bigr)\,dx=n!(1+\gamma_j)|\Sph^n|\le C.
\end{equation}

Subtracting the two normal representations of $u_j$ and $u_j^{\text{rad}}$ gives
\begin{equation}\label{eq:repr}
h_j(x)-h_j(0)=-\frac{2}{(n-1)!|\Sph^n|}\int_{\R^n}\log\frac{|x-y|}{|y|}\,a_j(y)h_j(y)\,dy.
\end{equation}
Hence
\begin{equation}\label{grad hj}
\nabla h_j(x)=-\frac{2}{(n-1)!|\Sph^n|}\int_{\R^n}\frac{x-y}{|x-y|^2}\,a_j(y)h_j(y)\,dy.
\end{equation}

 For $x\in B_R$, we split the integral into $|x-y|<1$ and its complement. Differentiate \eqref{grad hj} once more, by \eqref{eq:coeff} and \eqref{eq:L1}, we obtain
\begin{equation*}
	\begin{aligned}
		|\nabla h_j(x)|+|D^2h_j(x)|\le &C\int_{|x-y|<1}
		\frac{1}{|x-y|^2}\,a_j(y)|h_j(y)|\,dy
		+C\int_{|x-y|\geq 1}
		\frac{1}{|x-y|}\,a_j(y)|h_j(y)|\,dy\\
		\le& C_R\int_{|z|<1}
		\bigl(|z|^{-2}\bigr)\,dz+C\int_{\R^n}a_j\,dy\\
		\le &C_R.
	\end{aligned}
\end{equation*}
Here $D^2h_j$ is first understood in the weak sense. Thus, by the Arzel\`a--Ascoli theorem, after passing to a subsequence, we have
\begin{equation}\label{eq:convergence}
h_j\longrightarrow h\quad\hbox{in }C^1_{\mathrm{loc}}(\R^n).
\end{equation}
In particular, convergence is uniform on $\overline B_1$. Equation \eqref{eq:maximum} therefore gives
\begin{equation}\label{max h=1}
\max_{\overline B_1}|h|=1,\qquad \|h\|_{L^\infty(\R^n)}=1.
\end{equation}

Now we pass to the limit in \eqref{eq:repr}. For fixed $x$ and $R>2|x|+1$, we have
\begin{equation*}
\left|\log\frac{|x-y|}{|y|}\right|\le\frac{C_x}{|y|}, \quad |y|>R,
\end{equation*}
so by \eqref{eq:L1}, 
\begin{equation*}
	\int_{|y|>R}\left|\log\frac{|x-y|}{|y|}\right|a_j(y)|h_j(y)|\,dy\le\frac{C_x}{R}.
\end{equation*}
 On bounded sets, we have the convergence \eqref{eq:coeff} and \eqref{eq:convergence}. Consequently,
\begin{equation}\label{eq:limitintegral}
h(x)-h(0)=-\frac{2}{(n-1)!|\Sph^n|}\int_{\R^n}\log\frac{|x-y|}{|y|}\,n!e^{nU(y)}h(y)\,dy.
\end{equation}

We observe that \eqref{eq:limitintegral} is exactly the integral form of the linearized equation at $U$. In other words,
\begin{equation*}
	(-\Delta)^{n/2}h=n!e^{nU}h\quad\hbox{in }\R^n.
\end{equation*}
By Corollary \ref{coro:nondegen}, the nondegeneracy of $U$ implies that $h=\sum_{i=0}^n c_iZ_i$ for some $c_i\in\R$, where $Z_0,\dots, Z_n$ are defined by \eqref{eq:kernel}.

Since $\nabla u_j(0)=\nabla u_j^{\text{rad}}(0)=0$ for all $j$, we have  $\nabla h_j(0)=0$. By \eqref{eq:convergence}, $\nabla h(0)=0$. Since
\begin{equation*}
\nabla Z_0(0)=0,\quad \partial_k Z_i(0)=2\delta_{ik}, \qquad i,k=1,\dots,n,
\end{equation*}
we get $c_1=\dots=c_n=0$.

 For every $x\ne0$, equation \eqref{eq:maximum} passes to the limit and gives $h(\frac{x}{|x|^2})=h(x)$. Since
\begin{equation*}
Z_0(\frac{x}{|x|^2})=-Z_0(x),\quad Z_i(\frac{x}{|x|^2})=Z_i(x),\qquad i=1,\dots,n,
\end{equation*}
we get $c_0=0$.

In conclusion, we have $h\equiv0$, contradicting \eqref{max h=1}. The uniqueness up to a dilation follows immediately. This completes the proof of Theorem \ref{thm:rad small ga}.

\end{proof}

\section*{Acknowledgment}

H. Li would like to extend his sincere gratitude to Professor Lei Zhang (UFL) for his inspiring suggestions, invaluable guidance and constant support over the years. The research of J. Wei is partially supported by GRF of RGC of Hong Kong entitled ``On critical and supercritical Fujita equation''. The research of W. Yang is partially supported by National Key R\&D Program of China 2022YFA1006800, NSFC No. 12531010, FDCT No. 0070/2024/RIA1, Multi-Year Research Grant No. MYRG-GRG2024-00082-FST-UMDF, MYRG-GRG202500051-FST and UMDF No. TISF/2025/006/FST. The authors acknowledge the use of AI tools. All mathematical arguments and proofs in the final manuscript were checked and written by the authors.

\bigskip

\noindent
\textsc{Haoyu Li}\\
Department of Mathematics, Faculty of Science and Technology, University of Macau, Taipa, Macau, China\\
\textit{Email address: hyli1994@hotmail.com}

\medskip

\noindent
\textsc{Tuoxin Li}\\
Department of Mathematics, Faculty of Science and Technology, University of Macau, Taipa, Macau, China\\
\textit{Email address: tuoxinli@um.edu.mo} 

\medskip

\noindent
\textsc{Juncheng Wei}\\
Department Of Mathematics, The Chinese University Of Hong Kong, Shatin, Nt, Hong Kong\\
\textit{Email address:  wei@math.cuhk.edu.hk} 

\medskip

\noindent
\textsc{Wen Yang}\\
Department of Mathematics, Faculty of Science and Technology, University of Macau, Taipa, Macau, China\\
\textit{Email address: wenyang@um.edu.mo}


\begin{thebibliography}{99}

\bibitem{AWZ2022}
M.~Ahmedou, L.~Wu and L.~Zhang,
\emph{Classification and a priori estimates for the singular prescribing
$Q$-curvature equation on $4$-manifold},
J. Funct. Anal. \textbf{283} (2022), Paper No.~109649.



\bibitem{Branson1995}
T.~P. Branson,
\emph{Sharp inequalities, the functional determinant and the complementary
series},
Trans. Amer. Math. Soc. \textbf{347} (1995), no.~10, 3671--3742.





\bibitem{BrezisMerle1991}
H.~Br\'ezis and F.~Merle,
\emph{Uniform estimates and blow-up behavior for solutions of $-\Delta u=V(x)e^u$ in two dimensions},
Comm. Partial Differential Equations \textbf{16} (1991), no.~8--9, 1223--1253.



\bibitem{CLY2019}
J.~S. Case, Y.-J. Lin and W.~Yuan,
\emph{Conformally variational Riemannian invariants},
Trans. Amer. Math. Soc. \textbf{371} (2019), no.~11, 8217--8254.

\bibitem{CM2023}
J.~S. Case and A.~Malchiodi,
\emph{A factorization of the GJMS operators of special Einstein products and
applications},
J. Lond. Math. Soc. (2) \textbf{110} (2024), no.~5,
Paper No.~e70023, 17 pp.

\bibitem{ChangChen2001}
S.-Y.~A. Chang and W.~Chen,
\emph{A note on a class of higher order conformally covariant equations},
Discrete Contin. Dyn. Syst. \textbf{7} (2001), no.~2, 275--281.

\bibitem{ChangGonzalez2011}
S.-Y.~A. Chang and M.~d.~M. Gonz\'alez,
\emph{Fractional Laplacian in conformal geometry},
Adv. Math. \textbf{226} (2011), no.~2, 1410--1432.




\bibitem{ChenLi1993}
W.~Chen and C.~Li,
\emph{Qualitative properties of solutions to some nonlinear elliptic equations
in $\R^2$},
Duke Math. J. \textbf{71} (1993), no.~2, 427--439.



\bibitem{ChenLiOu2006}
W.~Chen, C.~Li and B.~Ou,
\emph{Classification of solutions for an integral equation},
Comm. Pure Appl. Math. \textbf{59} (2006), no.~3, 330--343.



\bibitem{CR}
M.~G. Crandall and P.~H. Rabinowitz,
\emph{Bifurcation from simple eigenvalues},
J. Funct. Anal. \textbf{8} (1971), 321--340.

\bibitem{Deimling1985}
K.~Deimling,
\emph{Nonlinear Functional Analysis},
Springer-Verlag, Berlin, 1985.

\bibitem{DelaTorreGonzalez2018}
A.~DelaTorre and M.~d.~M.~Gonz\'alez,
\emph{Isolated singularities for a semilinear equation for the fractional Laplacian arising in conformal geometry},
Rev. Mat. Iberoam. \textbf{34} (2018), no.~4, 1645--1678.

\bibitem{DEM2012}
M.~del Pino, P.~Esposito and M.~Musso,
\emph{Nondegeneracy of entire solutions of a singular Liouville equation},
Proc. Amer. Math. Soc. \textbf{140} (2012), no.~2, 581--588.



\bibitem{DHL2000}
Z.~Djadli, E.~Hebey and M.~Ledoux,
\emph{Paneitz-type operators and applications},
Duke Math. J. \textbf{104} (2000), no.~1, 129--169.

\bibitem{DM2008}
Z.~Djadli and A.~Malchiodi,
\emph{Existence of conformal metrics with constant $Q$-curvature},
Ann. of Math. (2) \textbf{168} (2008), no.~3, 813--858.

\bibitem{Durrett2019}
R.~Durrett,
\emph{Probability: Theory and Examples},
fifth edition,
Cambridge University Press, 2019.

\bibitem{FangMa2022}
H.~Fang and B.~Ma,
\emph{Constant $Q$-curvature metrics on conic $4$-manifolds},
Adv. Calc. Var. \textbf{15} (2022), no.~2, 235--264.

\bibitem{FG2013}
C.~Fefferman and C.~R.~Graham,
\emph{Juhl's formulae for GJMS operators and $Q$-curvatures},
J. Amer. Math. Soc. \textbf{26} (2013), no.~4, 1191--1207.

\bibitem{GhoshHyder2025}
M.~Ghosh and A.~Hyder,
\emph{Bol's type inequality for singular metrics and its application to prescribing $Q$-curvature problems},
preprint, arXiv:2512.24828 (2025).

\bibitem{GGN2013}
F.~Gladiali, M.~Grossi and S.~L.~N.~Neves,
\emph{Nonradial solutions for the H\'enon equation in $\R^N$},
Adv. Math. \textbf{249} (2013), 1--36.

\bibitem{GJMS1992}
C.~R. Graham, R.~Jenne, L.~J. Mason and G.~A.~J. Sparling,
\emph{Conformally invariant powers of the Laplacian, I: Existence},
J. London Math. Soc. (2) \textbf{46} (1992), no.~3, 557--565.

\bibitem{GLW2026}
Z.~Guo, Z.~Liu and F.~Wan,
\emph{New types of singular solutions of conformal $Q$-curvature equations},
Acta Math. Sin. (Engl. Ser.) \textbf{42} (2026), no.~2, 481--507.

\bibitem{GurMal2015}
M.~J. Gursky and A.~Malchiodi,
\emph{A strong maximum principle for the Paneitz operator and a non-local flow
for the $Q$-curvature},
J. Eur. Math. Soc. (JEMS) \textbf{17} (2015), no.~9, 2137--2173.

\bibitem{HormanderI}
L.~H\"ormander,
\emph{The Analysis of Linear Partial Differential Operators I:
Distribution Theory and Fourier Analysis},
Classics in Mathematics,
Springer-Verlag, Berlin, 2003.

\bibitem{HormanderIII}
L.~H\"ormander,
\emph{The Analysis of Linear Partial Differential Operators III:
Pseudo-Differential Operators},
Classics in Mathematics,
Springer, Berlin, 2007.


\bibitem{Hyder2016}
A.~Hyder,
\emph{Existence of entire solutions to a fractional Liouville equation in
$\R^n$},
Rend. Lincei Mat. Appl. \textbf{27} (2016), no.~1, 1--14.

\bibitem{Hyder2019}
A.~Hyder,
\emph{Structure of conformal metrics on $\mathbb R^n$ with constant
$Q$-curvature},
Differential Integral Equations \textbf{32} (2019), no.~7--8, 423--454.

\bibitem{HLW2020}
A.~Hyder, C.-S.~Lin and J.~Wei,
\emph{The $\mathrm{SU}(3)$ Toda system with multiple singular sources},
Pacific J. Math. \textbf{305} (2020), no.~2, 645--666.

\bibitem{HMM2021}
A.~Hyder, G.~Mancini and L.~Martinazzi,
\emph{Local and nonlocal singular Liouville equations in Euclidean spaces},
Int. Math. Res. Not. IMRN (2021), no.~15, 11393--11425.

\bibitem{HM2015}
A.~Hyder and L.~Martinazzi,
\emph{Conformal metrics on $\R^n$ with constant $Q$-curvature, prescribed
volume, asymptotic behavior},
Discrete Contin. Dyn. Syst. \textbf{35} (2015), no.~1, 283--299.

\bibitem{HM2021Power}
A.~Hyder and L.~Martinazzi,
\emph{Normal conformal metrics on $\R^4$ with $Q$-curvature having power-like growth},
J. Differential Equations \textbf{301} (2021), 37--72.

\bibitem{JSY2025}
A.~Jevnikar, Y.~Sire and W.~Yang,
\emph{Prescribing $Q$-curvature on even-dimensional manifolds with conical singularities},
Rev. Mat. Iberoam. \textbf{41} (2025), no.~1, 1--28.

\bibitem{JX2026}
Q.~Jiang and J.~Xiong,
\emph{Existence of non-radial entire solutions for the H\'enon equation
beyond even exponents},
preprint, arXiv:2606.31670.

\bibitem{JSTW2022}
J.~Jin, L.~Shu, D.~Tai and D.~Wu,
\emph{Normal conformal metrics on $\R^4$ with singular $Q$-curvature having power-like growth},
J. Differential Equations \textbf{327} (2022), 64--108.

\bibitem{JMMX2015}
T.~Jin, A.~Maalaoui, L.~Martinazzi and J.~Xiong,
\emph{Existence and asymptotics for solutions of a non-local $Q$-curvature
equation in dimension three},
Calc. Var. Partial Differential Equations \textbf{52} (2015), no.~3--4,
469--488.

\bibitem{Karlin}
S.~Karlin,
\emph{Total Positivity, Vol.~I},
Stanford University Press, Stanford, CA, 1968.

\bibitem{Kato1995}
T.~Kato,
\emph{Perturbation Theory for Linear Operators},
second edition, Classics in Mathematics, Springer-Verlag, Berlin, 1995.

\bibitem{KohrNistorWendland}
M.~Kohr, V.~Nistor and W.~L.~Wendland,
\emph{Layer potentials and essentially translation invariant pseudodifferential
operators on manifolds with cylindrical ends},
in \emph{Postpandemic Operator Theory},
Theta Ser. Adv. Math., pp.~61--116, 2023;
arXiv:2308.06308.

\bibitem{KonigLaurain2022}
T.~K\"onig and P.~Laurain,
\emph{Constant $Q$-curvature metrics with a singularity},
Trans. Amer. Math. Soc. \textbf{375} (2022), no.~4, 2915--2948.




\bibitem{Lin1998}
C.-S.~Lin,
\emph{A classification of solutions of a conformally invariant fourth order
equation in $\R^n$},
Comment. Math. Helv. \textbf{73} (1998), no.~2, 206--231.





\bibitem{LockhartMcOwen1985}
R.~B.~Lockhart and R.~C.~McOwen,
\emph{Elliptic differential operators on noncompact manifolds},
Ann. Scuola Norm. Sup. Pisa Cl. Sci. (4) \textbf{12} (1985), no.~3, 409--447.





\bibitem{Martinazzi2009}
L.~Martinazzi,
\emph{Classification of solutions to the higher order Liouville's equation on
$\R^{2m}$},
Math. Z. \textbf{263} (2009), no.~2, 307--329.




\bibitem{McCann1997}
R.~J.~McCann,
\emph{A convexity principle for interacting gases},
Adv. Math. \textbf{128} (1997), no.~1, 153--179.



\bibitem{DLMF}
NIST Digital Library of Mathematical Functions,
\emph{Chapters 5 and 15},
\url{https://dlmf.nist.gov/}.

\bibitem{Pinkus}
A.~Pinkus,
\emph{Totally Positive Matrices},
Cambridge University Press, Cambridge, 2010.

\bibitem{PrajapatTarantello2001}
J.~Prajapat and G.~Tarantello,
\emph{On a class of elliptic problems in $\R^2$: Symmetry and uniqueness
results},
Proc. Roy. Soc. Edinburgh Sect. A \textbf{131} (2001), no.~4, 967--985.


\bibitem{ReedSimonI}
M.~Reed and B.~Simon,
\emph{Methods of Modern Mathematical Physics, Vol.~I:
Functional Analysis},
second edition,
Academic Press, San Diego, 1980.


\bibitem{ReedSimonIV}
M.~Reed and B.~Simon,
\emph{Methods of Modern Mathematical Physics, Vol.~IV:
Analysis of Operators},
Academic Press, New York, 1978.


\bibitem{Rudin1991}
W.~Rudin,
\emph{Functional Analysis}, second edition,
McGraw--Hill, New York, 1991.




\bibitem{Santambrogio2016}
F.~Santambrogio,
\emph{Dealing with moment measures via entropy and optimal transport},
J. Funct. Anal. \textbf{271} (2016), no.~2, 418--436.

\bibitem{Schaefer1974}
H.~H. Schaefer,
\emph{Banach Lattices and Positive Operators},
Grundlehren der mathematischen Wissenschaften, Vol.~215,
Springer-Verlag, Berlin--Heidelberg--New York, 1974.


%\bibitem{Simon}
%B.~Simon,
%\emph{Trace Ideals and Their Applications},
%second ed., American Mathematical Society, Providence, RI, 2005.

\bibitem{Villani}
C.~Villani,
\emph{Topics in Optimal Transportation},
American Mathematical Society, Providence, RI, 2003.

\bibitem{WeiXu1999}
J.~Wei and X.~Xu,
\emph{Classification of solutions of higher order conformally invariant
equations},
Math. Ann. \textbf{313} (1999), no.~2, 207--228.

\bibitem{WeiYe2008}
J.~Wei and D.~Ye,
\emph{Nonradial solutions for a conformally invariant fourth order equation in
$\R^4$},
Calc. Var. Partial Differential Equations \textbf{32} (2008), no.~3, 373--386.



\end{thebibliography}
\end{document}